%% file: mettok.tex
\documentclass[11pt]{amsart}
\usepackage[margin=1.25in]{geometry}

\usepackage{amsfonts,amsmath,amssymb,amsthm,graphics,subfig,enumerate,
color,xcolor,cite,verbatim, mathdots, url, appendix}
\usepackage[mathscr]{euscript}
\usepackage[bookmarks=false]{hyperref}
\usepackage[shellescape]{gmp}

\newcommand{\C}{\mathbb{C}}
\newcommand{\R}{\mathbb{R}}
\newcommand{\Q}{\mathbb{Q}}
\newcommand{\Z}{\mathbb{Z}}
\newcommand{\N}{{\mathbb {N}}}

\newcommand{\ra}{\rightarrow }

\newcommand{\z}{{\mathbf {z}}}
\newcommand{\x}{{\mathbf {x}}}
\newcommand{\y}{{\mathbf {y}}}

\newcommand{\la}{\lambda }
\newcommand{\La }{{\Lambda }}
\newcommand{\Ga}{{\Gamma}}

\newcommand{\oo}{{\mathfrak {o}}}

\newcommand{\resn}[2]{\left(\frac{#1}{#2}\right)_n}

\newcommand{\polyring}{\mathcal {A}}
\newcommand{\ffield}{\mathcal {K}}
\newcommand{\innprod}[2]{\langle #1, #2 \rangle}

\newcommand{\Dem}{{\mathcal {D}}}

\newcommand{\T}{{\mathcal {T}}}

\newcommand{\tDelta}{{\Delta _v}}
\newcommand{\ttDelta}{{\Delta _t}}

\newcommand{\Cr}{{\mathcal {C}}}
\newcommand{\IP}{{\mathfrak {T}}}

\newcommand{\G}{{\mathbf {G}}}
\newcommand{\Grm}{{\mathrm {G}}}
\newcommand{\SG}{{\mathrm {SG}}}
\newcommand{\GL}{{\mathrm {GL}}}
\newcommand{\gl}{{\mathfrak {gl}}}
\newcommand{\bG}{\mathbf{G}}

\newcommand{\mm}{{\mathbf {m}}}

\newcommand{\OO}{{\mathcal {O}}}

\newcommand{\wght}{{\mathrm {wt}}}
\newcommand{\W}{{\mathcal {W}}}

\newcommand{\tG}{\widetilde{G}}
\newcommand{\tT}{\widetilde{T}}
\newcommand{\tB}{\widetilde{B}}

\newcommand{\sgn}{{\mathrm {sgn}}}
\newcommand{\ifi}{{\mathrm{if}}\ }
\renewcommand{\tilde}{\widetilde} 
\newcommand{\Le}{{\mathfrak{L}}}
\newcommand{\Ri}{{\mathfrak{R}}}
\newcommand{\Ind}{{\mathrm {Ind}}}
\newcommand{\lv}{\lambda ^{\vee}}
\newcommand{\Lv}{\Lambda ^{\vee}}

\newtheorem{lemma}{Lemma}
\newtheorem{prop}[lemma]{Proposition}
\newtheorem{theorem}[lemma]{Theorem}
\newtheorem{corollary}[lemma]{Corollary}
\newtheorem{claim}[lemma]{Claim}

\theoremstyle{definition}
\newtheorem{remark}{Remark}
\newtheorem{example}{Example}
\newtheorem{definition}{Definition}

\numberwithin{equation}{section}

\author{Anna Pusk\'as}

\address{Department of Mathematical and Statistical Sciences,
University of Alberta,
Edmonton, AB T6G 2G1, Canada
}

\email{puskas@ualberta.ca}

\title{Whittaker functions on metaplectic covers of $GL(r)$}
\date\today

\begin{document}

\begin{abstract}
\input{abstract.tex}
\end{abstract}

\subjclass[2010]{Primary 22E50; Secondary 11F68, 05E10}
\keywords{Whittaker functions, metaplectic groups, Demazure-Lusztig operators, Tokuyama's Theorem, highest weight crystals, Iwahori-Whittaker functions, Weyl group multiple Dirichlet series}

\maketitle


\setcounter{tocdepth}{1}
\tableofcontents 

\input{intro.tex}

\input{prelim.tex}

\input{demazure.tex}

\input{crystals_patterns.tex}

\input{Tokuyama_rephrase.tex}

\input{DemCrystals.tex}

\input{main_ind.tex}

\input{tech_ind.tex}

\input{whittaker.tex}

\appendix
\input{rankone_comp.tex}

\bibliographystyle{amsalpha_no_mr_initials.sty}
\bibliography{bib}

\end{document}

%% file: abstract.tex
This paper establishes a combinatorial link between different approaches to constructing Whittaker functions on a metaplectic group over a non-archimedean local field. We prove a metaplectic analogue of Tokuyama's Theorem and give a crystal description of polynomials related to Iwahori-Whittaker functions. The proof relies on formulas of metaplectic Demazure and Demazure-Lusztig operators, proved previously in joint work with Gautam Chinta and Paul E. Gunnells \cite{cgp}.

%% file: intro.tex
\section{Introduction}


\subsection{Motivation}\label{subsection:motivation}

The study of metaplectic groups was initiated by Matsumoto \cite{matsumoto1969sous}. Analytic number theory, in particular questions about the mean values of $L$-functions led to research on multiple Dirichlet series, which in turn motivated interest in Whittaker coefficients of metaplectic Eisenstein series. Whittaker functions are higher dimension generalizations of Bessel functions and are associated to principal series representations of a reductive group over a local field. Kubota \cite{kubota1971some} was the first to closely examine Eisenstein series on higher covers of $\GL_2$, and the theory of associated Whittaker functions was further developed by Kazhdan and Patterson \cite{MR743816}. In recent years, this development gained further impetus from unexpected connections to other areas, such as combinatorial representation theory, the geometry of Schubert varieties, and solvable lattice models. While the theory of metaplectic Whittaker functions is familiar in the case of double covers of reductive groups, it is less well understood in the case of higher covers. 

\subsubsection{The Casselman-Shalika formula} \label{subsubsection:Casselman-Shalika}

The Casselman-Shalika formula is an explicit formula for values of a spherical Whittaker function over a local $p$-adic field in terms of a character of a reductive group. It is a central result in understanding the local and global theory of automorphic forms and their $L$-functions. A metaplectic analogue, describing Whittaker functions on $n$-fold covers of a reductive group, has similar significance in the study of Dirichlet series of several variables. 
Different approaches to generalize the Casselman-Shalika formula to the metaplectic setting have recently emerged. 

\subsubsection{Metaplectic analogues}\label{subsubsection:Metaplectic_CS}

Chinta-Offen \cite{co-cs} and McNamara \cite{mcnamara2} generalize the Casselman-Shalika formula by replacing the character with a metaplectic analogue: a sum over the Weyl group involving a modified action of the Weyl group that depends on the metaplectic cover. Brubaker-Bump-Friedberg \cite{bbf-annals} and McNamara \cite{mcnamara} express a type $A$ Whittaker function as a sum over a crystal base. Both constructions produce the Whittaker function as a polynomial determined by combinatorial data: the root datum of the group, a dominant weight, and the degree $n$ of the metaplectic cover. The first one handles all types of root datum, while the second one makes it possible to compute the coefficients of the polynomial individually. The fact that the descriptions are purely combinatorial in nature, and rely heavily on Weyl group combinatorics and on the structure of the crystal graph, indicates that deeper properties of these constructions can be understood using methods of combinatorial representation theory.

\subsubsection{Combinatorial link}\label{subsubsection:comblink}

In the present paper we develop a combinatorial understanding of the relationship of the two approaches described in section \ref{subsubsection:Metaplectic_CS}. This is one aspect of our main result (Theorem \ref{THM:MAIN}); an other is giving a crystal description of polynomials related to Iwahori-Whittaker functions. Both of these aspects will be made explicit in section \ref{section:Whittaker}. 

Furthermore, both approaches to constructing Whittaker functions, i.e. summing over the Weyl group, or, respectively, over a crystal graph, also make sense in the nonmetaplectic setting. In this special case, a theorem of Tokuyama provides a combinatorial link between them \cite{tokuyama-generating}. In the metaplectic setting, the constructions of section \ref{subsubsection:Metaplectic_CS} naturally extend the meaning of respective sides of Tokuyama's identity. Thus explicitly relating the two constructions is in essence proving a metaplectic analogue of Tokuyama's formula. Viewed purely as an identity about the combinatorial data, the special case of Theorem \ref{THM:MAIN} stated as Theorem \ref{thm:metaplectic_Tokuyama} is this metaplectic analogue.

\subsection{Methods and tools}

The connection between Tokuyama's theorem and the constructions of metaplectic Whittaker functions also gives a hint as to where the difficulty in this project lies. The classical proof of Tokuyama's formula is by induction on the rank using Pieri-rules: in the metaplectic setting, these have no convenient analogue. We sidestep this obstacle by refining the metaplectic statement to allow for a finer induction. Theorem \ref{thm:metaplectic_Tokuyama} (the metaplectic version of Tokuyama's theorem) is as a statement about the long element in the Weyl group, while the more general Theorem \ref{THM:MAIN} is the same statement for any ``beginning section'' of a particular long word. 
 
Phrasing a statement that lends itself to finer induction requires us to exploit interesting properties of the respective constructions. On the one hand, our joint work with Gautam Chinta and Paul E. Gunnells \cite{cgp} interprets the formulas of Chinta-Offen and McNamara \cite{co-cs, mcnamara2} in terms of metaplectic Demazure-Lusztig operators. On the other hand, one may exploit the branching structure of highest weight crystals of Dynkin type $A$ to relate the formulas of Brubaker-Bump-Friedberg and McNamara \cite{bbf-annals, mcnamara} to smaller, similar expressions on Demazure crystals. The connection to Demazure-Lusztig operators is well motivated, and yields several avenues of possible applications; we mention these below. 

\subsubsection{Demazure operators}\label{subsubsection:intro_dem} The relevance of classical Demazure and Demazure-Lusztig operators to the study of Whittaker functions was first indicated by work of Littelmann and Kashiwara \cite{kashiwara1992crystal} giving character formulas on a crystal, and of Brubaker, Bump and Licata \cite{bbl-demazure} relating them to Iwahori-fixed Whittaker functions. They have since been used by Patnaik \cite{patnaik2014unramified} to give a generalization of the Casselman-Shalika formula to Whittaker functions on the $p$-adic points of an affine Kac-Moody group. 

The metaplectic versions of these operators were introduced by G. Chinta, P. E. Gunnells and the present author \cite{cgp}; the definitions involve the Chinta-Gunnells action of the Weyl group on rational functions over the weight lattice. This action was first used in \cite{cg-jams} to construct $p$-parts of multiple Dirichlet series, and have since proved instrumental in metaplectic constructions, for example the ones mentioned above in \ref{subsubsection:Metaplectic_CS}.

\subsection{Applications}\label{subsection:applications}

We mention connections to the literature and avenues of further research utilizing the methods and results of this paper. 

\subsubsection{Iwahori-Whittaker functions}\label{subsubsection:BBLextend}

In \cite{bbl-demazure}, the authors use Demazure and Demazure-Lusztig operators to compute values of Iwahori-Whittaker functions in terms of Hecke algebras, the geometry of Bott-Samelson varieties and the combinatorics of Macdonald polynomials. The analogies between these topics are intertwined with the combinatorics of the Bruhat order on the Weyl group, and identities satisfied by the Demazure and Demazure-Lusztig operators. Furthermore, the non-metaplectic version of the operator in Theorem \ref{THM:MAIN} is related in \cite{bbl-demazure} to Iwahori-Whittaker functions. Recent work by  Lee, Lenart, and Liu \cite{lee2016whittaker} applies these results to compute coefficients of the transition matrix between natural bases of Iwahori-Whittaker functions. On the other hand, the work of Patnaik \cite{patnaik2014unramified} generalizing the Casselman-Shalika formula to the affine Kac-Moody setting also involves computing Iwahori-Whittaker functions, and their recursion in terms of Demazure-Lusztig operators. (The results of Brubaker-Bump-Licata and Patnaik are recalled in detail in section \ref{section:Whittaker}.)

The metaplectic Demazure and Demazure-Lusztig operators satisfy analogous identities to the classical, nonmetaplectic ones. Thus some of the above results will generalize to the metaplectic setting. Recent joint work with Manish Patnaik \cite{patnaik2015casselman} shows that the connection between Iwahori-Whittaker functions and Demazure-Lusztig operators extends to the metaplectic setting (see section \ref{subsubect:MMP_PA}). It is natural to ask if the explicit crystal description of Iwahori-Whittaker functions given by Theorem \ref{THM:MAIN} leads to a better understanding of all these results. It is especially interesting to consider how the connection with Iwahori-Whittaker functions, perhaps together with a more type-independent combinatorial description as mentioned in section \ref{subsubsection:alcovepath} below, would elucidate the situation in the affine setting (see \ref{subsubsection:affine_wmds}).s

\subsubsection{The Alcove Path Model}\label{subsubsection:alcovepath}
The construction of Whittaker functions as a sum over the Weyl group \cite{co-cs,mcnamara2} has the key feature that the Weyl group functional equations satisfied by the Whittaker function become very apparent. These functional equations play a key role in the analytic  construction of global multiple Dirichlet series constructed from the Whittaker functions. Moreover, they have proven useful in studying certain affine analogues. 

The functional equations are less explicit in the description by crystal graphs. However, the crystal construction gives explicit formulas for individual coefficients of the Whittaker function. Reasons for trying to understand these coefficients are mentioned in section \ref{subsubsection:affine_wmds}. 

Various authors have worked on generalizing the crystal approach to root systems of other types: Chinta and Gunnells for type $D$ \cite{chinta2012littelmann}, Beineke \cite{beineke2012crystal}, Brubaker, Bump, Chinta, Gunnells \cite{brubaker2012metaplectic}, Frechette, Friedberg and Zhang \cite{friedberg2015eisenstein} for type $B$ and $C$, McNamara \cite{mcnamara} working, less explicitly, with crystal bases in general. The resulting formulas are all significantly more intricate than the type $A$ construction in \cite{bbf-wmdbook}. 

A possible applications of this paper is to understand how the crystal approach extends to other types. Preliminary work by Beazley and Brubaker suggests that perhaps the alcove path model is better suited for creating a construction that generalizes the type $A$ crystal approach. Demazure and Demazure-Lusztig operators promise to give a metaplectic Casselman-Shalika formula in terms of the alcove path model; we are currently investigating this avenue in joint work with Gautam Chinta, Cristian Lenart and Dan Orr. The resulting construction might better reflect the Weyl group symmetry of the individual coefficients of Whittaker functions. 

%

\subsubsection{Affine Weyl group multiple Dirichlet Series and metaplectic Whittaker functions}\label{subsubsection:affine_wmds}

Recent work of Bucur-Diaconu \cite{bucur2010moments}, Lee-Zhang \cite{lee2012weyl} and Whitehead \cite{whitehead2014affine} attempts to extend the theory of multiple Dirichlet series to the affine setting. There the theory of Eisenstein series is not (yet) available. These authors construct multiple Dirichlet series that satisfy functional equations corresponding to an affine Weyl group. The coefficients of these power series can be explicitly related to character sums and coefficients of $L$-functions \cite{whitehead2014affine}. Some of our methods may lead to a combinatorial understanding of these coefficients. Furthermore, it would be especially interesting to understand the connection between affine Weyl group multiple Dirichlet series and Whittaker functions on $p$-adic points of affine Kac-Moody groups introduced by Patnaik \cite{patnaik2014unramified}. The possibility of extending the affine construction to the metaplectic setting is currently investigated by Patnaik and the author of this paper.


\subsection{Acknowledgements}
The results of this paper form part of my doctoral thesis at Columbia University. I am greatly indebted to my adviser Gautam Chinta for his support and invaluable advice over the years. While writing the paper I was supported from Manish Patnaik's Subbarao Professorship in Analytic Number Theory, an NSERC discovery grant, and an University of Alberta startup grant. I thank Manish Patnaik for this support and for his guidance. I thank the organizers and participants of the 2013 ICERM semester on ``Automorphic Forms, Combinatorial Representation Theory and Multiple Dirichlet Series,'' the Department of Mathematics at Columbia University and the University of Alberta for an inspiring environment. I am very grateful to the following people for helpful conversations and insights: my executive adviser Dorian Goldfeld, Ben Brubaker, Corrin Clarkson, Solomon Friedberg, Amy Feaver, Holley Friedlander, Paul E. Gunnells, Christian Lenart, Peter McNamara, Dinakar Muthiah, Ben Salisbury, Michael Thaddeus, Ian Whitehead and Wei Zhang.

%% file: prelim.tex
\section{Preliminaries and Statement of the Main Theorem}\label{sect:preliminaries}

This section is dedicated to the statement of the main result of the paper (Theorem \ref{THM:MAIN}), and a brief outline of the methods and structure of the proof. 

\subsection{Background}
We start by introducing some notation and background. We restrict ourselves to what is necessary to state the main theorem, Theorem \ref{THM:MAIN}, and to give an outline of the methods of the paper. Much of the background will be covered in more detail in later sections.

\subsubsection{Notation}\label{subsubs:notation_in_prelim} Let $\Lambda $ be the weight lattice corresponding to a root system $\Phi $ of type $A_r.$ We identify $\C (\Lambda )$ with a ring of rational functions $\C(\x ),$ where $\x=(x_1,\ldots ,x_{r+1})$ and $\x^{\alpha _1}=x_1/x_2.$ The Weyl group $W$ is generated by $\sigma _i$ simple reflections. Let $w_0\in W$ be the long element. We favour a particular reduced decomposition for $w_0$ (see \eqref{eq:def_of_favourite_w0} for this ``favourite long word''); Theorem \ref{THM:MAIN} is stated for elements $w\in W$ whose reduced decomposition is a ``beginning segment'' of this favourite long word (Definition \ref{def:begsectlw}).
The integer $n$ denotes the degree of the metaplectic cover of a split reductive algebraic group corresponding to $\Phi $. We also introduce the indeterminate $t,$ and $v=t^n;$ in applications, we set $v=q^{-1},$ where $q$ is the order of the residue field of a nonarchimedean local field. 

\subsubsection{Crystals and Gelfand-Tsetlin coefficients} The highest weight crystal $\Cr _{\la +\rho }$ and its parameterizations will be introduced in Section \ref{section:HWtCrystalsandGTpatterns}. For now, it suffices to say that it is a graph whose vertices are in bijection with a basis of the irreducible representation of highest weight $\la +\rho ,$ where $\lambda \in \Lambda$ is dominant and $\rho $ is the Weyl vector. Vertices of a crystal can be parameterized by arrays of integers in various ways (using Gelfand-Tsetlin patterns, $\Gamma$-arrays, or Berenstein-Zelevinsky-Littelmann paths). To state Theorem \ref{THM:MAIN} we need two functions on the vertices of a crystal $\Cr _{\la +\rho }:$ the weight function $v\mapsto \wght(v)\in \Z^{r+1}$ and the Gelfand-Tsetlin coefficient $v\mapsto G^{(n,\la+\rho )}(v)\in \C[t].$ This is the usual Gelfand-Tsetlin coefficient, described for nonmetaplectic and metaplectic cases in \cite{bbf-wmdbook}. It depends on the degree $n$ of the metaplectic cover via Gauss-sums. Furthermore, for every $w$ beginning segment of the long word, we shall define $\Cr_{\la +\rho }^{(w)},$ the Demazure crystal corresponding to $w.$ This is a subgraph of $\Cr _{\la +\rho }$ spanned by certain vertices depending on $w$ (see Definition \ref{def:Demazure_subcrystal}).

\subsubsection{Demazure operators} Demazure operators $\Dem _{w}$ and Demazure-Lusztig operators $\T_w$ correspond to elements of the Weyl group, and act on $\C (\Lambda )$. The definitions of the nonmetaplectic operators involve the natural action of the Weyl group $W$ on $\C (\Lambda ),$ inherited from the action of $W$ on the weight lattice. In the metaplectic setting, this normal permutation action can be replaced by the Chinta-Gunnells action, and one may define metaplectic Demazure and Demazure-Lusztig operators, whose meaning depends on $n.$ The definitions and properties are recalled - in the notation specific to type $A_r$ - in section \ref{sect:Demazure}; these play a key role in the proof of Theorem \ref{THM:MAIN}. 

\subsubsection{Tokuyama's theorem}\label{subsubsection:Tokuyama}
Strictly speaking, this section is not necessary to understand the statement of Theorem \ref{THM:MAIN}; however, it provides motivation, and crucial guidance to the shape of the statement. 
As mentioned in section \ref{subsubsection:comblink} generalizing Tokuyama's theorem to the metaplectic setting and linking the constructions of metaplectic Whittaker functions is closely related; in fact the constructions give rise to the statement of a metaplectic version. We explain this briefly here; the theorem will be discussed in detail in Section \ref{section:Tokuyama}. 
Tokuyama's theorem is a deformation of the Weyl character formula in type $A:$
\begin{equation}\label{eq:Tokuyama}
 \x ^{\rho }\cdot \prod _{\alpha \in \Phi ^+} (1-v\cdot \x^{\alpha }) \cdot s_{\la }(\x)=\sum _{b\in \Cr _{\la +\rho }} G(b)\cdot \x^{\wght(b)},
\end{equation}
where $s_{\la }$ is the Schur function. The left hand side essentially agrees with the Casselman-Shalika formula for Whittaker functions (with the deforming parameter $v$ specialized to $q^{-1}$). The Schur function can be expressed by the Weyl character formula as 
\begin{equation}\label{eq:Schur}
  s_{\la }(\x)=\frac{1}{\x^{\rho }\cdot \prod _{\alpha \in \Phi ^+}(1-\x^{\alpha })}\cdot \sum _{w\in S_{r}} (-1)^{\ell(w)}\cdot \x ^{w(\la +\rho )}.
\end{equation}
Chinta-Offen \cite{co-cs} show what a correct metaplectic analogue of the right hand side in \eqref{eq:Schur} is, replacing the action of the Weyl group $W$ on $\C(\Lambda )$ by the Chinta-Gunnells metaplectic action. One may use the results of \cite{cgp} to reformulate the ``left-hand side'' in terms of Demazure-Lusztig operators, i.e. as
\begin{equation}
\sum _{u\in W} \T _u
\end{equation}
acting on a monomial. The necessary background will be covered in detail in Section \ref{sect:Demazure}.

On the right hand side of \eqref{eq:Tokuyama}, the Gelfand-Tsetlin coefficients $G(b)=G^{(1,\la+\rho )}(b)$ appear. This reproduces the construction of the same Whittaker function as a sum over a crystal base (Brubaker-Bump-Friedberg \cite{bbf-annals}) in both the nonmetaplectic case (for $n=1$) and the metaplectic setting (for higher $n$). 

\subsection{Statement of Main Theorem}
The main result of the paper is a crystal description of sums of Demazure-Lusztig operators in type $A.$ More precisely, we prove the following. 

\begin{theorem}\label{THM:MAIN}
 Let $\lambda =(\la _1,\ldots ,\la_r ,\la _{r+1})$ be any dominant, effective weight, $\rho =(r,r-1,\ldots, 1, 0),$ $w$ a beginning section of the long word. Then 
\begin{equation}\label{eq:thm_main_statement}
 \left(\sum _{u\leq w} \T _u\right) \x ^{w_0(\la )}=\x^{-w_0(\rho )}\cdot \sum _{v\in \Cr_{\la +\rho }^{(w)}} G^{(n,\la+\rho )}(v)\cdot \x^{\wght(v)}.
\end{equation}
Here $\leq $ is the Bruhat order, $G^{(n,\la+\rho )}(v)$ is the Gelfand-Tsetlin coefficient corresponding to $v$ (by Definition \ref{def:GTcoeff}), and $\Cr_{\la +\rho }^{(w)}$ is the Demazure-crystal corresponding to $w$ (as is Definition \ref{def:Demazure_subcrystal}). 
\end{theorem}

The statement \eqref{eq:thm_main_statement} provides the combinatorial link between the approaches to constructing metaplectic Whittaker functions described in section \ref{subsubsection:Metaplectic_CS}. The special case of this statement for $w=w_0$ and $n=1$ is exactly Tokuyama's theorem (See section \ref{section:Tokuyama}). The statement is formally stronger than Tokuyama's theorem even in the nonmetaplectic setting, and provides a metaplectic analogue for higher $n.$ We state this analogue on its own. 

\begin{theorem}\label{thm:metaplectic_Tokuyama}
(Tokuyama's Theorem, Metaplectic Version.) Let $\lambda =(\la _1,\ldots ,\la _{r+1})$ be any dominant, effective weight and $\rho =(r,\ldots, 1, 0).$ Then 
\begin{equation}\label{eq:thm_metaplectic_tokuyama_statement}
 \left(\sum _{u\in W} \T _u\right)(\x^{w_0(\la )})=\x^{-w_0\rho }\cdot \sum_{v\in \Cr _{\la +\rho }} G^{(n,\la +\rho )}(v)\cdot \x ^{\wght(v)}.
\end{equation}
\end{theorem}

This special case of the identity is present when work of Brubaker-Bump-Friedberg-Hoffstein \cite{bbf-annals}, Chinta-Gunnells-Offen \cite{cg-jams,co-cs}, and McNamara \cite{mcnamara,mcnamara2} are combined, but Theorem \ref{THM:MAIN} provides a much more direct connection. In addition, as mentioned in \ref{subsubsection:BBLextend}, the operators 
\begin{equation}\label{eq:LHSops}
\sum _{u\leq w} \T _u
\end{equation}
are related to the construction of Iwahori-Whittaker functions; in this sense Theorem \ref{THM:MAIN} may be interpreted as a crystal description of Iwahori-Whittaker functions. We shall make the connection between Theorem \ref{THM:MAIN} and Whittaker and Iwahori-Whittaker functions more explicit in section \ref{section:Whittaker}.

\subsection{Methods and Outline}

We give an overview of the methods and structure of the proof of Theorem \ref{THM:MAIN}. 

Tokuyama's proof of the identity \eqref{eq:Tokuyama} uses Pieri rules, i.e. is by induction on the rank $r$ of the type $A_r$ of the root system $\Phi .$ Pieri rules are not available in the metaplectic setting, so instead we ``refine'' the induction. Theorem \ref{THM:MAIN} interprets Tokuyama's formula in type $A_r$ as a statement about the (favourite) long word $w_0^{(r)}.$ This is the following reduced decomposition of the long element.
\begin{equation*}
  w_0=w_0^{(r)}=\sigma _1\sigma _2\sigma _1\cdots \sigma _{r-1}\cdots \sigma _1\sigma _r\cdots \sigma _1.
 \end{equation*}
This word has the property that it begins with $w_0^{(r-1)},$ the favourite long word in type $A_{r-1}.$  Since Theorem \ref{THM:MAIN} formulates an identity for every beginning section of the word $w_0^{(r)},$ if we want to prove this statement by induction, the step from $A_{r-1}$ to $A_r$ is now not one step, but $r$ ``smaller'' steps. 

This is the main idea - the proof is in fact by induction; the main tools are branching properties of type $A$ highest weight crystals, and identities of Demazure(-Lusztig) operators. 

\subsubsection{Structure of the induction} 
The edges of a highest weight crystal $\Cr_{\la +\rho }$ of type $A_r$ are labelled by indices $1,2,\ldots , r.$ Removing the edges labelled with $r$ breaks the crystal into the disjoint union of highest weight crystals of type $A_{r-1}.$ The same is true for a Demazure crystal $\Cr_{\la +\rho }^{(w)}$ as long as $w$ is a beginning section of the favourite long word $w_0^{(r)}.$ This fact is crucial to the proof. Let $w$ be a beginning section of long word $w_0^{(r)}$ that is {\em{not}} a beginning section of $w_0^{(r-1)},$ i.e. has the form 
\begin{equation*}
 w=w_0^{(r-1)}\sigma _r\cdots \sigma _{r-k},
\end{equation*}
where $k=\ell(w)-\ell(w_0^{(r-1)})-1.$ Call the statement of Theorem \ref{THM:MAIN} for this particular $w$ and fixed $n$ (but for any $\la $) $IW_{r,k}^{(n)}.$ For $k=r-1,$ i.e. $w=w_0^{(r)},$ Theorem \ref{THM:MAIN} specializes to Theorem \ref{thm:metaplectic_Tokuyama}, and thus we use the notation $Tok_{r}^{(n)}=IW_{r,r-1}^{(n)}.$

It follows from the branching property of Demazure crystals mentioned above, as well as from a simple fact about the Bruhat order (Lemma \ref{lem:Bruhat_shorterwords}) that $IW_{r,k}^{(n)}$ can be reduced to $Tok_{r-1}^{(n)},$ and statements describing the action of simpler operators on a monomial. The full reduction argument will be explained later; for now we only say that in addition to $IW_{r,k}^{(n)}$ and $Tok_{r}^{(n)},$ some auxiliary statements will be phrased: $M_{r,k}^{(n)}$ and $N_{r,k}^{(n)},$ and (the special case) $N_{r,r-1}^{(n)}.$ The statement $N_{r,k}^{(n)},$ for example, concerns the action of the operator $T_rT_{r-1}\ldots T_{r-k}$ on a monomial. We show in Section \ref{chap:equivalent_reformulations} that to prove $IW_{r,k}^{(n)}$ for any $r$ and any $0\leq k<r$ it suffices to prove the statement $N_{r,r-1}^{(n)}$ for any $r.$ This reduction of $IW_{r,k}^{(n)}$ to $N_{r,r-1}^{(n)}$ essentially follows from the branching of Demazure crystals and Gelfand-Tsetlin coefficients, and some properties of the metaplectic Demazure-Lusztig operators.

By the end of Section \ref{chap:equivalent_reformulations}, the only thing remaining from the proof of Theorem \ref{THM:MAIN} is to prove $N_{r,r-1}^{(n)},$ a statement about the action of the operator $T_rT_{r-1}\ldots T_1$ on a monomial. This statement is proved by a (somewhat technical) induction in Section \ref{chap:technical_induction}, with a rank one auxiliary computation included in Appendix \ref{app:rankonecomp}. 

\subsubsection{Outline}

The necessary background is summarized in three sections. Section \ref{sect:Demazure} explains the Chinta-Gunnells action and metaplectic Demazure operators; Section \ref{section:HWtCrystalsandGTpatterns} describes parameterizations and branching of type $A$ highest weight crystals and contains the definition of Gelfand-Tsetlin coefficients; Section \ref{section:Tokuyama} contains the re-phrasal of Tokuyama's result into the language of Demazure-Lusztig operators and crystals. 

The proof of Theorem \ref{THM:MAIN} spans three sections. Section \ref{section:DemCrystals} is preparation: it defines Demazure crystals and examines the Gelfand-Tsetlin coefficients on these in terms of the branching properties discussed in Section \ref{section:HWtCrystalsandGTpatterns}. Some helpful conventions, designed to make the notation of the proof lighter, are also introduced here. Section \ref{chap:equivalent_reformulations} contains the proof of Theorem \ref{THM:MAIN} through reduction to a sequence of simpler statements (from $IW_{r,k}^{(n)}$ to $N_{r,r-1}^{(n)}$), as explained above. The final statement of the sequence, $N_{r,r-1}^{(n)}$ is then proved in Section \ref{chap:technical_induction} (and Appendix \ref{app:rankonecomp}). 

Section \ref{section:Whittaker} relates Theorem \ref{THM:MAIN} to metaplectic Whitaker functions and Iwahori-Whittaker functions. The constructions mentioned in the Introduction are recalled in a little more detail to demonstrate how the formulas line up with the expressions in Theorem \ref{THM:MAIN}.

%% file: demazure.tex
\section{Metaplectic Demazure and Demazure-Lusztig operators}\label{sect:Demazure}

Theorem \ref{THM:MAIN} describes the action of metaplectic Demazure-Lusztig operators on a monomial. As mentioned in \ref{subsubsection:intro_dem} the metaplectic analogues of the classical Demazure and Demazure-Lusztig operators were introduced in \cite{cgp}. In this section, we briefly review the results of that paper, specializing to type $A$ root systems. The definition, elementary properties, and the identities Theorem \ref{THM:LONG_WORD} and Theorem \ref{THM:T_SUM} will be necessary for the proof. The metaplectic operators are built on the Chinta-Gunnells action; we recall the definition in Section \ref{subsect:CGaction}. We restrict our attention to type $A,$ hence some of the machinery that is necessary in \cite{cgp} can be spared. 


\subsection{Notation}\label{subs:CGDemNotation}

The following is standard notation for root systems and the Weyl group. The reader may refer to \cite{humph} as a source. 

Let $\Phi$ be an irreducible reduced root system of type $A_r$ with Weyl
group $W$. We may view $\Phi $ as embedded into $\R^{r+1}.$ Let $e_1,\ldots ,e_{r+1}$ denote the standard basis of $\R^{r+1}$, and take
$$\Phi =\{e_i-e_j\in \R^{r+1}\mid 1\leq i\neq j\leq r+1\}.$$
Let $\Phi = \Phi^+ \cup \Phi^{-}$ be the decomposition into positive and negative roots ($e_i-e_j\in \Phi^+$ if $i<j$). Let
$\{\alpha_1, \alpha_2, \ldots, \alpha_r\}$ be the set of simple roots; $\alpha _i=e_i-e_{i+1}$ ($1\leq i\leq r$).
and let $\sigma_i$ be the Weyl group element corresponding to the
reflection through the hyperplane perpendicular to $\alpha_i$. 
Set
\begin{equation}\label{def:Phi_w} 
\Phi(w)=\{\alpha\in\Phi^+: w(\alpha)\in\Phi^-\}.
\end{equation}
Consider the weight lattice
$$\La =\{\la =(\la _1,\la _2,\ldots ,\la _r,\la _{r+1})\in \Z ^{r+1}\};$$
then $\Lambda \subset \R^{r+1}$ contains $\Phi$ as a subset.
Let $\polyring=\C[\Lambda]$ be the ring of Laurent polynomials on
$\Lambda.$ Let $\ffield$ be the field of
fractions of $\polyring.$ The action of $W$ on the lattice $\Lambda$
induces an action of $W$ on $\ffield$: we put
\begin{equation}
  \label{eq:W_on_K}
  (w, x^\lambda)\longmapsto x^{w\lambda} =:w.x^{\lambda },
\end{equation}
and then extend linearly and multiplicatively to all of $\ffield$.
We denote this action using the lower dot $(w,f)\longmapsto w.f$ (to distinguish it from the metaplectic $W$-action on $\ffield$ constructed below in \eqref{eq:metaWnotation}) and refer to this as the ``nonmetaplectic'' group action. 

Let $\x=(x_1,\ldots ,x_r,x_{r+1}).$ We may identify $\ffield $ with $\C(x_1,\ldots ,x_{r+1})=\C(\x)$ by writing $x_i=\x^{e_i}.$
In general, for $\la =\sum _i \lambda _i e_i \in \La $ as above, we write $\x ^{\la }=x_1^{\la _1}\cdot x_2^{\la _2}\cdots x_{r+1}^{\la _{r+1}}.$
Note that the Weyl group $W\cong S_{r+1},$ and the nonmetaplectic action \eqref{eq:W_on_K} of $\sigma _i$ on $\C(\x)$ is by swapping $x_i$ and $x_{i+1}.$

The definition of the Chinta-Gunnells action in \cite{cgp} requires a $W$-invariant $\Z$-valued quadratic form $Q$ defined on $\Lambda$, which defines a bilinear form $B (\alpha ,\beta)=Q (\alpha +\beta) - Q (\alpha)-Q(\beta)$. Fix a positive integer $n$; $n$ determines a
collection of integers $\{m (\alpha) : \alpha \in \Phi\}$ by
\begin{equation}\label{eqn:defofm}
  m(\alpha)=n/\gcd(n, Q (\alpha )),
\end{equation}
and a sublattice $\Lambda_0\subset \Lambda$ by 
\begin{equation}\label{eqn:defoflambda0}
\Lambda_{0} = \{\lambda \in \Lambda : \text{$B (\alpha ,\lambda) \equiv 0
\bmod n$ for all simple roots $\alpha$}\}.
\end{equation}

In type $A,$ we may give an explicit example of a $W\cong S_{r+1}$-invariant, integer-valued quadratic form. For $\la \in \La ,$ (and $c$ arbitrary), let 
\begin{equation}\label{eq:Qdef}
Q(\la)= -\sum _{h<j} \la _h\la _j -c\cdot \left(\sum _{h=1}^{r+1} \la _h\right)^2 =  -(1+2c)\cdot \sum _{h<j} \la_h\la _j -c\cdot \sum _{h=1}^{r+1} \la _h^2.
\end{equation}

Then certainly $Q$ (and thus $B$) are integer valued on $\La $ and $n\La \subseteq \La _0.$ Furthermore, it is easy to check that $Q(\alpha _i)=1$ and $B(\alpha _i,\la )=\la _{i}-\la _{i+1}.$ This implies that the sublattice $\La _0$ is 
\begin{equation}\label{eqn:lambda0_intypeA}
\Lambda_{0} = \{\lambda \in \Z^{r+1} : \text{$\la_i \equiv \la _j\ \bmod n$ for all $1\leq i,j\leq r+1$}\}.
\end{equation}

Since all roots are of the same length, $m(\alpha )=n/\gcd(n,Q(\alpha ))$ is the same for every root. In particular, with the choice of $Q$ above, $Q(\alpha )=1$ and hence $m(\alpha )=n.$ 

\begin{remark}\label{LEMMA:SUBLATTICE}
  For any simple root $\alpha$, we have $m(\alpha)\alpha=n\alpha\in  \Lambda_0$. (This is a special case of \cite[Lemma 1]{cgp}.)
\end{remark}

\subsection{The Chinta-Gunnells action}\label{subsect:CGaction}

The Chinta-Gunnells action is a ``metaplectic'' action of a Weyl group on a ring of rational functions; the action depends on the metaplectic degree. We use the same definition as in \cite{cgp}, which in turn is the same as the one defined in Chinta-Gunnells \cite{cg-jams} and specializes to the type $A$ action in Chinta-Offen \cite{co-cs}. 


Following \cite[Section 2]{cgp}, let $\lambda \mapsto \bar \lambda $ be the projection $\Lambda
\rightarrow \Lambda /\Lambda_{0}$ and $(\Lambda/\Lambda_0)^*$ be the
group of characters of the quotient lattice.
Any $\xi\in (\Lambda/\Lambda_0)^*$ induces a field automorphism of
$\ffield/\C$ by setting $\xi(x^\lambda)=\xi(\bar\lambda)\cdot x^\lambda$
for $\lambda\in\Lambda.$ This leads to the direct sum decomposition
\begin{equation}
  \label{eq:direct_sum}
  \ffield=\bigoplus_{\bar\lambda\in \Lambda/\Lambda_0}
  \ffield_{\bar\lambda}
\end{equation}
where $\ffield_{\bar\lambda}=\{f\in\ffield: \xi(f)=\xi(\bar\lambda)\cdot f
\mbox{ for all $\xi\in(\Lambda/\Lambda_0)^* $}\}.$

The next ingredient 
is a set of complex parameters $v,g_0, \ldots, g_{n-1}$ 
satisfying
\begin{equation}
  \label{eq:qt}
  g_0=-1 \mbox{  and  } g_ig_{n-i}=v^{-1}\mbox{  for  } i=1,\ldots, n-1.
\end{equation}
For all other $j$ we define $g_{j}:=g_{r_{n} (j)}$, where $0\leq r_{n} (j)<n-1$ denotes the remainder upon dividing $j$ by $n$. The parameters $g_0, \ldots, g_{n-1}$ will be chosen in section \ref{subsubsect:Gauss_sums} to be Gauss sums. 

We may now recall the definition of the metaplectic action of the Weyl group $W$ on $\ffield.$

\begin{definition}\label{def:action}
\cite[Section 2, (7)]{cgp} For $f\in \ffield_{\bar\lambda}$ and the
generator $\sigma_\alpha\in W$ corresponding to a simple root $\alpha$, define
 \begin{equation}\label{eq:def:action}
  \begin{split}
\sigma _i(f)=\frac{\sigma _i. f}{1-vx^{m(\alpha _i)\alpha _i}}&\cdot 
\left[x^{-r_{m(\alpha _i)}\left(-\frac{B(\la,\alpha _i)}{Q(\alpha
        _i)}\right)\cdot \alpha _i}\cdot (1-v)\right.\\
&\left. {}-v\cdot g_{Q(\alpha _i)-B(\la ,\alpha _i)}\cdot 
x^{(1-m(\alpha _i))\alpha _i}\cdot (1-x^{m(\alpha _i)\alpha _i})\right] 
  \end{split}
 \end{equation}
where $\lambda$ is any lift of $\bar \lambda$ to $\Lambda$.
Here the quantity in brackets depends only on $\bar
\lambda$.  We extend the definition of $\sigma_\alpha$ to $\ffield$ by
additivity.  Then (\ref{def:action}) extends to an action of the full Weyl
group $W$ on $\ffield$, which we denote 
\begin{equation}\label{eq:metaWnotation}
(w,f) \longmapsto w (f).
\end{equation}
\end{definition}

Using the notation specific to type $A,$ Definition \ref{def:action} can be rewritten for $f=\x^{\la }$ as follows.
\begin{equation}\label{eq:def:typeAaction}
  \begin{split}
\sigma _i(f)=\frac{\sigma _i. f}{1-v\cdot \left(\frac{x_i}{x_{i+1}}\right)^{n}}&\cdot 
\left[\left(\frac{x_i}{x_{i+1}}\right)^{-r_{n}\left(\la _{i+1}-\la _i\right)}\cdot (1-v)\right.\\
&\left. {}-v\cdot g_{1+\la _{i+1}-\la _i}\cdot 
\left(\frac{x_i}{x_{i+1}}\right)^{1-n}\cdot \left(1-\left(\frac{x_i}{x_{i+1}}\right)^{n}\right)\right]    .
  \end{split}
 \end{equation}


The following Lemma is crucial in computations; it is used repeatedly, if implicitly, in the proof of Theorem \ref{THM:MAIN}. It relies on the fact that the quantity in brackets in \eqref{def:action} depends only on $\bar\lambda$ and not $\lambda .$

\begin{lemma}
  \label{lemma:h_exchange}
 \cite[Lemma 2]{cgp} Let $f\in \ffield$ and $h\in \ffield_0$.  Then for any $w\in W$,
$$w(hf)=(w.h)\cdot w(f).$$
Here $w.h$ means the non-metaplectic action, while $\cdot$ denotes multiplication in $\ffield$.\hfill \qed
\end{lemma}

The significance of Lemma \ref{lemma:h_exchange} is due to the fact that the action of $W$ on $\ffield$
defined by \eqref{def:action}, though $\C$-linear, is {\em{not}} by
endomorphisms of that ring, i.e.~it is not in general multiplicative.
The point of Lemma \ref{lemma:h_exchange} is that if we have a product
of two terms $hf$, the first of which satisfies $h\in \ffield_{0}$ (e.g. the exponents of $h$ are divisible by $n$), then we can apply $w$ to the product $hf$ by performing the usual permutation action on $h$ and then acting on $f$ 
by the metaplectic $W$-action.

The following lemma shows a symmetric monomial with respect to $\sigma _i.$ It will be of use in computations. We omit the (straightforward) proof.

\begin{lemma}\label{lem:invariant_monom}
 $\sigma _i (x_i^{a+1}x_{i+1}^{a+n})=x_i^{a+1}x_{i+1}^{a+n}$ for every $n.$
\end{lemma}


\subsection{Gauss sums}\label{subsubsect:Gauss_sums}
The complex parameters $v, g_0, \ldots, g_{n-1}$ of the Chinta-Gunnells action are chosen to be Gauss sums in applications. Similar Gauss sums ($g^{\flat }$ and $h^{\flat }$) are used in \cite{bbf-wmdbook} to define Gelfand-Tsetlin coefficients on a crystal graph (see section \ref{section:HWtCrystalsandGTpatterns}).

We make the choice of parameters explicit here. We start by describing the functions $g^{\flat }$ and $h^{\flat },$ following \cite[Chapter 1]{bbf-wmdbook} for notation and definitions. We will then choose the parameters $v, g_0, \ldots, g_{n-1}$ to satisfy the conditions of \eqref{eq:qt}. For facts about the power residue symbol we refer the reader to \cite{bbf-wmd2}.

\subsubsection{Notation}
Let $F$ be an algebraic number field containing the group $\mu _{2n}$ of $2n$-th roots of unity. Let $S$ be a finite set of places of $F$,  large enough that it contains all the places that are Archimedean or ramified over $\Q,$ and the ring of $S$-integers $\oo _S=\{x\in F\mid |x|_v\leq 1\ \text{for}\ v\notin S \}$ is a principal ideal domain. Let $\psi $ be a character on $F_S$ of conductor $\oo_S.$ For any $m,c\in \oo_S,$ $c\neq 0,$ consider the $n$-th power residue symbol $\resn{m}{c}.$ Recall that $\resn{m}{c}$ is zero unless $m$ is prime to $c.$ It is multiplicative, i. e. $\resn{m}{c}\cdot \resn{m}{b}=\resn{m}{bc}.$
If $p$ is a prime and $m$ is coprime to $p,$ then $\resn{m}{p}$ is the element of $\mu _n$ satisfying $\resn{m}{p}\equiv m^{\frac{\N p-1}{n}} \ \bmod p .$

With the notation above, define the Gauss sum 
\begin{equation}\label{eq:def_of_g(m,c)}
 g(m,c)=\sum _{a\ \bmod \ c}\resn{a}{c}\psi \left(\frac{am}{c}\right).
\end{equation}

Fix a $p$ prime in $\oo _S,$ and let $q$ be the cardinality of the residue field $\oo_S/p\oo_S.$ We assume $q\equiv 1$ modulo $2n.$ Define $g(a)=g(p^{a-1},p^{a})$ and $h(a)=g(p^{a},p^{a})$ for any $a>0.$ In this case we have 
$$g(a)=\sum _{b\ \bmod \ p^a}\resn{b}{p^a}\psi \left(\frac{b}{p}\right)=q^{a-1}\cdot \sum _{b\ \bmod \ p}\resn{b}{p}^a\psi \left(\frac{b}{p}\right)$$
and 
$$h(a)=\sum _{b\ \bmod \ p^a}\resn{b}{p^a}\psi \left(b\right)=q^{a-1}\cdot \sum _{b\ \bmod \ p}\resn{b}{p}^a\cdot 1=\left\lbrace \begin{array}{ll}
0 & n\!\!\!\not| a;\\
(q-1)\cdot q^{a-1} & n|a. 
\end{array}
\right.$$

\subsubsection{Choice of parameters}\label{subsubsect:gflathflatdef}
We are ready to define the functions $g^{\flat }$ and $h^{\flat }.$ These appear in Section \ref{section:HWtCrystalsandGTpatterns} in the definition of Gelfand-Tsetlin coefficients, and the proof of Theorem \ref{THM:MAIN} depends on computations that use $g^{\flat }$ and $h^{\flat }.$ 
Let 
\begin{equation}\label{eq:def_of_g_flath_flat}
 g^{\flat }(a)=q^{-a}\cdot g(a)\hskip .5 cm \text{and} \hskip .5 cm h^{\flat }(a)=q^{-a}\cdot h(a).
\end{equation}
The following identities imply that the value of both $g^{\flat }(a)$ and $h^{\flat }(a)$ only depend on the residue of $a$ modulo $n.$ 
\begin{equation}\label{eq:h_and_g_flatdef}
h^{\flat }(a)=\left\lbrace \begin{array}{ll}
0 & n\!\!\!\not| a;\\
1-\frac{1}{q} & n|a. 
\end{array}
\right. \hskip.5 cm \text{ and }\hskip.5 cm g^{\flat }(a)=q^{-1}\cdot \sum _{b\ \bmod \ p}\resn{b}{p}^a\psi \left(\frac{b}{p}\right).
\end{equation}

If $a$ is divisible by $n$ then 
\begin{equation}\label{eq:gflat0}
 g^{\flat }(a)=-q^{-1},
\end{equation}
and if $0<a<n$ then 
\begin{equation}\label{eq:gflat_identity}
 g^{\flat }(a)\cdot g^{\flat }(n-a)=q^{-1}.
\end{equation}

Recall the conditions \eqref{eq:qt} imposed on the parameters $v,g_0,\ldots ,g_{n-1}$. The parameters must satisfy $g_0=-1$ and $g_ig_{n-i}=v^{-1}$ for $1\leq i\leq n-1.$ We can choose these parameters by modifying the functions $g^{\flat}$ and $h^{\flat }.$ Take $v=q^{-1}$ and 
\begin{equation}\label{eq:parameterdef}
  g_i=v^{-1}\cdot g^{\flat }(i)=q\cdot g^{\flat }(i)=\sum _{b\ \bmod \ p}\resn{b}{p}^i\psi \left(\frac{b}{p}\right)\mbox{  for  } i=1,\ldots, n-1.
\end{equation}
Then \eqref{eq:gflat0} implies $g_0=q\cdot (-q^{-1})=-1$ and\eqref{eq:gflat_identity} implies 
$$g_ig_{n-i}=v^{-2}\cdot g^{\flat }(i)\cdot g^{\flat }(n-i)=v^{-2}\cdot v=v^{-1}.$$

%
%
%

We summarize the choices of parameters in the following claim. The notation $t^n=v=q^{-1}$ is introduced for later convenience.

\begin{claim}\label{claim:GaussSumRelationship}
If $n\nmid a,$ then $h^{\flat}(a)=0,$ and 
$$v\cdot g_a=q^{-1}\cdot {\mathfrak{g}}^{\psi }(-a)=\gamma (a)=g^{\flat }(a)=q^{-1}\cdot \sum _{b\ \bmod \ p}\resn{b}{p}^a\psi \left(\frac{b}{p}\right).$$
 However, if $n|a,$ then $h^{\flat }(a)=1-v,$  $\gamma (a)=g_a=g_0=-1,$ and $g^{\flat }(a)=-q^{-1}=-v=-t^n.$
\end{claim}

\subsection{Metaplectic Demazure and Demazure-Lusztig operators}
The definitions below follow \cite{cgp}, making use of the identification of $\ffield$ and $\C(\x)$ and the Chinta-Gunnells action introduced in section \ref{subsect:CGaction}. Both the Demazure operators and the Demazure-Lusztig operators are divided difference operators on $\ffield .$ 

Let  $1\leq i\leq r$ and $f\in \C(\x).$ We define the \emph{Demazure operators} by
\begin{equation}\label{def:dem}
  \Dem_i(f)=\Dem_{\sigma _i}(f)=
  \frac{f-\x^{n\alpha _i}\cdot \sigma _i(f)}
  {1-\x^{n\alpha _i}},
\end{equation}
and the \emph{Demazure-Lusztig operators} by
\begin{equation}\label{def:dl}
  \begin{split}
    \T _i(f)=\T _{\sigma _i}(f)&=
    \left(1-v\cdot \x^{n\alpha _i}\right)\cdot \Dem _i(f) -f\\
    &=\left(1-v\cdot \x^{n\alpha _i}\right)\cdot
    \frac{f-\x^{n\alpha _i}\cdot \sigma _i(f)}
    {1-\x^{n\alpha _i}}-f.
  \end{split}
\end{equation}
Recall that here $\x^{n\alpha _i}$ is shorthand for $\dfrac{x_i^n}{x_{i+1}^n}.$ When there is no danger of confusion, we write more simply
\begin{equation*}
  \Dem_i=
  \frac{1-\x^{n\alpha _i}\sigma _i}
  {1-\x^{n\alpha _i}}\hskip .5 cm \mbox{  and  }\hskip .5 cm 
  \T_i=\left(1-v\cdot \x^{n\alpha _i}\right)\cdot \Dem_i - 1,
\end{equation*}
that is, a rational function $h$ in the above equations is interpreted
to mean the ``multiplication by $h$'' operator. The rational functions
here are in $\ffield _0$ (see Remark \ref{LEMMA:SUBLATTICE}).  

The operators $\Dem _i$ and $\T_i$ satisfy the same braid relations as the $\sigma _i$ \cite[Proposition 7.]{cgp}.
Consequently, one may define $\Dem_w$ and $\T_w$ for any $w\in W$: let $w=\sigma_{i_1}\cdots \sigma _{i_l}$ be a reduced
expression for $w$ in terms of simple reflections. Then 
$$\Dem _w:=\Dem _{i_1}\cdots \Dem _{i_l}
\mbox{\ \ \ \ and\ \ \ \ } \T _w:=\T _{i_1}\cdots \T _{i_l}.$$

%


We also introduce a metaplectic analogue of the Weyl denominator. Let 
\begin{equation}\label{eq:deformeddenom_metaplectic}
  \tDelta =\tDelta ^{(n)}=\prod _{\alpha \in \Phi ^{+}} 
  \bigl(1-v\cdot \x^{n\alpha }\bigr).
\end{equation}
If $v=1$ we write simply $\Delta_v=\Delta .$ Now we are ready to state the metaplectic Demazure formula and Demazure-Lusztig formula. (As before, the notation is specific to type $A$)

\begin{theorem}\label{THM:LONG_WORD}
\cite[Theorem 3.]{cgp} For the long element $w_0$ of the Weyl group $W$ we have
$$\Dem _{w_0} =
\frac{1}{\Delta } \cdot \sum _{w\in W} \sgn (w)\cdot \prod _{\alpha
  \in \Phi(w^{-1})} \x^{n\alpha }\cdot w .$$
\end{theorem}

\begin{theorem}\label{THM:T_SUM}
  \cite[Theorem 4.]{cgp} We have
$$\tDelta \cdot \Dem _{w_0} =\sum _{w\in W} \T _w.$$
\end{theorem}


The following technical lemmas about polynomials annihilated by Demazure operators are of use in the proof of Theorem \ref{THM:MAIN}.

\begin{lemma}\label{lemma:Demazure_annihilation}
We have the following. 
\begin{enumerate}[(i)]
 \item A polynomial $f$ is annihilated by $\Dem _i=\Dem _{\sigma _i}$ if and only if $\sigma _i(x_{i+1}^n\cdot f)=x_{i+1}^n\cdot f.$
 \item If $\Dem _w (g)=0$ for some $w$ in the Weyl group $W$, and $w_0$ is the long element of $W$, then $\Dem _{w_0}g=0.$
\end{enumerate}
\end{lemma}
\begin{proof}
The proof of (i) is obvious from the definition of $\Dem _i$ and Lemma \ref{lemma:h_exchange}. 
For (ii), let $u=w_0w^{-1},$ so that $w_0=u\cdot w.$ Since $w_0$ is the longest element, we have $\ell(w_0)=\ell(u)+\ell(w),$ and as a consequence $\Dem _{w_0}=\Dem _u\circ \Dem _w.$ Thus $\Dem _{w_0}g=\Dem _u(\Dem _wg)=\Dem _u(0)=0.$
\end{proof}

The following is a trivial corollary of Lemmas \ref{lemma:Demazure_annihilation} and \ref{lem:invariant_monom}, as the action of $\sigma _i$ only involves the exponents of $x_i$ and $x_{i+1}.$

\begin{corollary}\label{cor:annihilation}
 If $\beta =(\beta _1,\ldots ,\beta _{r+1})$ and $\beta _{i}=\beta _{i+1}+1,$ then $\Dem _i(\x^{\beta })=0.$
\end{corollary}

%% file: crystals_patterns.tex
\section{Highest weight crystals and Gelfand-Tsetlin patterns}\label{section:HWtCrystalsandGTpatterns}


We turn our attention to the ``crystal side'' of Theorem \ref{THM:MAIN}: a sum whose terms involve Gelfand-Tsetlin coefficients, and are summed over a crystal. In this section, we present a primer on objects in this picture. Crystals can be parametrized in more than one way, we shall see that moving back and forth between parameterizations is not particularly difficult, hence one may choose the language that is most convenient in any given context. Gelfand-Tsetlin coefficients are defined in terms of these parameterizations.

We describe, in turn, highest weight crystals (\ref{section:Crystals}), Gelfand-Tsetlin patterns with ``$\Gamma $-arrays'' (\ref{section:GTpatterns}), and Berenstein-Zelevinsky-Littelmann paths (\ref{section:BZLpaths}). Gelfand-Tsetlin patterns are arrays of integers; the ones with a fixed top row are in bijection with vertices of a highest weight crystal. The bijection is via Berenstein-Zelevinsky-Littelmann paths, and the $\Gamma $-array corresponding to a pattern. The precise statement of this bijection is the content of Proposition \ref{prop:crystal_parametrization}. Gelfand-Tsetlin coefficients are defined in section \ref{section:GTcoeffs}. Finally, in section \ref{subs:crystal_branching}, we recall the branching property of type $A$ highest-weight crystals. This will be revisited for Demazure crystals in Section \ref{section:DemCrystals}, and is a crucial ingredient in the proof of Theorem \ref{THM:MAIN}. 

Throughout the section, we follow the presentation of Chapter $2$ of Brubaker-Bump-Friedberg \cite{bbf-wmdbook}, in less detail. We (implicitly) rely on other sources as well. In particular, for the combinatorial definition of a crystal graph, we use Hong-Kang \cite{hong2002introduction} and Kashiwara \cite{kashiwara1995crystal}. For the correspondence between Gelfand-Tsetlin patterns and highest weight crystals, Berenstein-Zelevinsky \cite{berenstein1993string, berenstein1996canonical}, Littelmann \cite{littelmann1998cones}, or Lusztig \cite{lusztig1990canonical} are further references. 


\subsection{Highest weight crystals}\label{section:Crystals}

The general definition of a crystal can be found in Kashiwara \cite{kashiwara1995crystal}. Here we only consider type $A$ highest weight crystals. 

Recall the notation introduced in section \ref{subs:CGDemNotation} for root systems of type $A_r$. In particular, recall that the weight lattice $\La $ is identified with $\Z^{r+1};$ $\alpha _i$ are the simple roots for $1\leq i\leq r.$ Let $h_i={\mathbf{e}}_i^{\ast }-{\mathbf{e}}_{i+1}^{\ast }\in \La ^{\ast }$ where ${\mathbf{e}}_1^{\ast },{\mathbf{e}}_2^{\ast },\ldots ,{\mathbf{e}}_{r+1}^{\ast }$ denotes the standard dual basis of $\R^{r+1}$. (We use ${\mathbf{e}}_i^{\ast }$ here to distinguish the basis vectors ${\mathbf{e}}_i$ from the Kashiwara operators $e_i$ below.) We have $(\cdot ,\cdot ):\La \times \La \ra \Q$ a bilinear symmetric form, and let $\langle \cdot ,\cdot \rangle :\La^{\ast }\times \La \ra \Z$ denote the canonical pairing.
Note that $(\alpha _i,\alpha _i)\in 2\Z_{>0},$ $\langle h_i ,\la \rangle =\frac{2(\alpha _i,\la )}{(\alpha _i,\alpha _i)}$ for $i\in I$ and $\la \in P,$ and $(\alpha _i,\alpha _j)\leq 0$ for $i,j\in I,$ $i\neq j.$ 
\newcommand{\Bc}{{\mathtt{B}}}

A type $A_r$ crystal $\Cr$ is a set $\Bc$ endowed with a weight function $\wght:\Bc\ra \La ,$ functions $\varepsilon _i :\Bc\ra \Z \sqcup \{-\infty \},$ $\varphi _i :\Bc\ra \Z \sqcup \{-\infty \}$ and Kashiwara operators $e _i :\Bc\ra \Bc\sqcup \{0\},$ $f _i :\Bc\ra \Bc \sqcup \{0\}$ for every $1\leq i\leq r.$ Elements of $\Bc$ are called elements or vertices of the crystal. A crystal satisfies the following axioms. (Let $-\infty +n=-\infty $ for every $n\in \Z .$)
\begin{enumerate}[(i)]
\item $\varphi _i(b)=\varepsilon _i(b)+\langle h_i,\wght (b)\rangle $ for every $1\leq i\leq r$
\item If $e_i(b)\neq 0,$ then 
$$\varepsilon _i(e_i (b))=\varepsilon _i(b)-1,$$
$$\varphi _i(e_i (b))=\varphi _i(b)+1,$$
$$\wght (e_i (b))=\wght (b)+\alpha _i.$$
\item If $f_i(b)\neq 0,$ then 
$$\varepsilon _i(e_i (b))=\varepsilon _i(b)+1,$$
$$\varphi _i(e_i (b))=\varphi _i(b)-1,$$
$$\wght (e_i (b))=\wght (b)-\alpha _i.$$
\item For $b_1,b_2\in \Bc,$ we have $b_2=f_i(b_1)$ if and only if $b_1=e_i(b_2).$
\item If $\varphi _i(b)=-\infty ,$ then $e_i(b)=f_i(b)=0.$
\end{enumerate}

Recall that the weight $\la =(\la _1,\la _2,\ldots ,\la _r,\la _{r+1})$ is called dominant if $\la _1\geq \la _2 \geq \cdots \geq \la _{r+1};$ strongly dominant if $\la _1>\la _2 >\cdots >\la _{r+1};$ $\la $ is effective if $\la _{r+1}\geq 0.$ There is a partial ordering on $\Z^{r+1}$ where $\mu \preccurlyeq \la $ if and only if $\la -\mu $ lies in the cone generated by simple roots. For every dominant weight $\la $ there is a corresponding crystal graph $\Cr _{\la}$ with highest weight $\la .$ The function $\wght$ maps the vertices of $\Cr_{\la }$ to weights of the representation $V_{\la }$ of $\gl _{r+1}(\C)$ of highest weight $\la .$ The Kashiwara operators determine a directed graph structure on $\Cr _{\la}:$ there is an edge $v\xrightarrow{i}w$ if and only if $f_i(v)=w\neq 0.$ We say this edge is labeled with $i.$ The number of vertices in $\Cr_{\la }$ with weight $\mu $ is equal to the multiplicity of the weight $\mu $ in the representation $V_{\la }$. In particular, $\Cr _{\la }$ has exactly one element $v_{highest}$ with weight $\la $. If $w_0$ denotes the longest element of the type Weyl group $W\cong S_{r+1},$ then $w_0\la =(\la _{r+1},\la _r,\ldots ,\la _2,\la _1)$, and $\Cr _{\la }$ has exactly one element $v_{lowest}$ with weight $w_0\la $ (this is the ``lowest'' element).

The edges labelled with the same index $i$ (for $1\leq i\leq r$) determine disjoint ``$i$-strings''  in the crystal. These are themselves isomorphic to type $A_1$ highest weight crystals. The functions $\varepsilon _i$ and $\varphi _i$ determine where a vertex is within an $i$-string:
$$\varepsilon _i(b)=\max \{n\geq 0| \ e_i^nb\neq 0\}, \hskip .5 cm \varphi _i(b)=\max \{n\geq 0| \ f_i^nb\neq 0\}.$$



We conclude this section by an example. 

\begin{example}\label{example:310crystal}
Figure \ref{fig:310crystal} shows a crystal of type $A_2$ corresponding to highest weight $(3,1,0).$ The red edges correspond to the label $1,$ the green edges to the label $2.$ Figure \ref{fig:310weights} shows the image of the same crystal under the weight map.

\begin{figure}[h!]
\centering
\begin{minipage}{0.45\textwidth}
\centering
\includegraphics{emptycrystal310-ps.mps} 
\caption{The crystal $\Cr_{(3,1,0)}.$}
\label{fig:310crystal}
\end{minipage}\hfill
\begin{minipage}{0.45\textwidth}
\centering
\includegraphics{310weights-ps.mps} 
\caption{The weights of the $\gl _{3}(\C)$ representation of highest weight $(3,1,0).$}
   \label{fig:310weights}\end{minipage}
\end{figure}

%

\end{example}

\subsection{Gelfand-Tsetlin patterns}\label{section:GTpatterns}

We recall the definition of Gelfand-Tsetlin patterns, the $\Gamma$-array and the weight associated to a pattern from \cite[Chapter 2]{bbf-wmdbook}.

\begin{definition}\label{defn:GTpatternsAndGammaArrays}
 A {\em{Gelfand-Tsetlin pattern}} of rank $r$ and top row $\la $ is an array of nonnegative integers 
 \begin{equation}\label{eq:GTpatternform}
\IP=\left(\begin{array}{cccccccccc}
             a_{00} & & a_{01} &  & a_{02} & \cdots  & a_{0,r-1} &  & a_{0r} \\
              & a_{11} &  & a_{12} &  & \cdots & & a_{1r} &  \\
              &  & \ddots & &  &   & \iddots  & &  \\
              &  &  & & a_{rr} &    &  &  &  
            \end{array}
\right)
 \end{equation}
where the top row is $\la=(a_{00},a_{01},\ldots ,a_{0,r-1},a_{0,r})=(\la _1,\la_2,\ldots ,\la_r,\la_{r+1}),$
and rows are non-increasing and interleave: $a_{i-1,j-1}\geq a_{ij}\geq a_{i-1,j}.$

For every $1\leq i\leq r,$ let
\begin{equation}\label{eq:defn:Gamma_entries}
\Gamma _{ij}=\Ga _{ij}(\IP)=\sum _{k=j}^r (a_{i,k}-a_{i-1,k}).
\end{equation}
This gives the {\em{$\Ga$-array}} of $\IP$
\begin{equation}\label{eq:defn:Gamma_array}
\Gamma (\IP)=\left[\begin{array}{cccc}
                      \Ga _{11} & \Ga _{12} & \cdots & \Ga _{1r}\\
                       & \Ga_{22}& \ldots & \Ga_{2r}\\
                         &  & \ddots & \vdots\\
                         &&& \Ga _{1r} 
                     \end{array}
\right].
\end{equation}
\end{definition}

\begin{remark}\label{rmk:Gamma_determines_pattern}
 Note that given the top row, the entries of the Gelfand-Tsetlin pattern $\IP $ can be recovered from the entries of $\Gamma (\IP ).$ That is, given $a_{0,i}$ and $\Gamma _{i,j}$ for $1\leq i\leq j\leq r,$ one can compute each $a_{i,j}.$ 
\end{remark}

Since the entries of the Gelfand-Tsetlin pattern $\IP$ satisfy $0\leq a_{i,k}-a_{i-1,k}\leq a_{i-1,k-1}-a_{i-1,k},$ we have
\begin{equation}\label{eq:GammaArrayIneq}
0\leq \Gamma _{ir}\leq a_{i-1,r-1}-a_{i-1,r};\hskip .5 cm \forall i\leq l\leq r-1\ \Gamma_{i,l+1}\leq \Gamma_{i,l}\leq \Gamma _{i,l+1}+a_{i-1,l-1}-a_{i-1,l}; 
\end{equation}
so the rows in $\Gamma (\IP)$ are nonnegative, non-increasing and there is an upper bound on the difference of consecutive entries in a row. The Gelfand-Tsetlin coefficient assigned to $\IP $ depends on the {\em{decoration}} of $\IP,$ i.e. whether these inequalities are strict or not. We recall the relevant terminology here. 

\begin{definition}\label{def:array_decorations}
({\em{Decorations}} of the entries of $\Gamma (\IP )$ and $\IP .$) An entry of $\Gamma (\IP )$ may be {\em{undecorated}}, {\em{circled}}, {\em{boxed}}, or {\em{both}}. The table below shows the (``right-leaning'') rules for decorating $\Gamma (\IP).$ (If $j=r,$ take $\Gamma_{i,r+1}=0$.)
\begin{equation}\label{eq:decoration_rules_Gamma}
\begin{array}{l|l}
                             \Gamma_{i,j+1}=\Gamma _{ij}<\Gamma_{i,j+1}+a_{i-1,j-1}-a_{i-1,j} & \Gamma _{ij}\ {\mathrm{is\ circled}}\\ \hline
                             \Gamma_{i,j+1}<\Gamma _{ij}<\Gamma_{i,j+1}+a_{i-1,j-1}-a_{i-1,j}& \Gamma _{ij}\ {\mathrm{is\ undecorated}}\\  \hline
                             \Gamma_{i,j+1}<\Gamma _{ij}=\Gamma_{i,j+1}+a_{i-1,j-1}-a_{i-1,j}&  \Gamma _{ij}\ {\mathrm{is\ boxed}}\\ \hline
                             \Gamma_{i,j+1}=\Gamma _{ij}=\Gamma_{i,j+1}+a_{i-1,j-1}-a_{i-1,j}&  \Gamma _{ij}\ {\mathrm{is\ circled\ and\ boxed}}
                                                                                        \end{array}
\end{equation}
We may phrase this as decorating the entries (below the top row) of the Gelfand-Tsetlin pattern $\IP$ itself. The decoration of $a_{i,j}$ is the same as that of $\Ga _{i,j}.$ 
\begin{equation}\label{eq:decoration_rules_pattern}
\begin{array}{l|l}
                              a_{i-1,j}=a_{ij}<a_{i-1,j-1}& a_{ij}\  {\mathrm{is\ circled}}\\ \hline
                              a_{i-1,j}<a_{ij}<a_{i-1,j-1}& a_{ij}\ {\mathrm{is\ undecorated}}\\  \hline
                              a_{i-1,j}<a_{ij}=a_{i-1,j-1}& a_{ij}\ {\mathrm{is\ boxed}}\\ \hline
                              a_{i-1,j}=a_{ij}=a_{i-1,j-1}& a_{ij}\ {\mathrm{is\ circled\ and\ boxed}}
                                                                                        \end{array}
\end{equation}
\end{definition}

Let $d_i$ denote the sum of the entries in the $i$-th row of $\IP ,$ that is, 
\begin{equation}\label{eq:def_di}
 d_i=d_i(\IP)=\sum _{j=i}^r a_{ij}.
\end{equation}
Then we may define the {\em{weight}} of a Gelfand-Tsetlin pattern $\IP .$
\begin{equation}\label{eq:def_pattern_weight}
\wght(\IP ):=(d_r,d_{r-1}-d_r,\ldots ,d_0-d_1)
\end{equation}


We conclude by an example. 
\begin{example}\label{example:patternofweight220}
 Consider Gelfand-Tsetlin patterns of top row $(3,1,0).$ One example of these is 
$$\IP =\left(\begin{array}{ccccc}
              3 & & 1 & & 0 \\
               & 3 &  & 1 & \\
               & & 2 & &  
             \end{array}
\right).$$
The corresponding $\Gamma $-array is 
\begin{equation}\label{eq:Gamma_example}
\Gamma (\IP )=\left[\begin{array}{cc}
                       3 & 1 \\
                        & 1 \\
                      \end{array}
\right]. 
\end{equation}
The sums of elements in the rows of the pattern $\IP $ are $d_0=4,$ $d_1=4$ and $d_2=2,$ hence 
$$\wght (\IP )=(d_2,d_1-d_2,d_0-d_1)=(2,2,0).$$
\end{example}
\subsection{Berenstein-Zelevinsky-Littelmann paths}\label{section:BZLpaths}

To a vertex $v$ in the crystal $\Cr_{\la  }$ and a choice of reduced decomposition for the long element $w_0\in W$ corresponds a Berenstein-Zelevinsky-Littelmann path. This is a path in the graph theoretic sense. It starts from $v,$ steps along the directed edges of the crystal, and ends in the lowest element, $v_{lowest}$. The steps correspond to applying successive Kashiwara operators $f_i$ to $v.$ The direction of steps is dictated by the choice of a long word $w_0.$ The notation follows \cite{bbf-wmdbook}; an explicit type $A_2$ example is included after the definition.

\subsubsection{Choice of the long word}
Let 
\begin{equation}\label{eq:def_of_favourite_w0}
 w_0=\sigma _1\sigma _2\sigma _1\cdots \sigma _{r-1}\cdots \sigma _1\sigma _r\cdots \sigma _1.
\end{equation}
This is our reduced expression of choice for the longest element in $S_{r+1}$ (our ``favourite long word''). Let $1\leq \Omega _i\leq r$ ($1\leq i\leq N=\ell(w_0)$) be the indices so that
\begin{equation}\label{eq:def_of_Omega_i}
 w_0=\sigma_{\Omega _1}\sigma_{\Omega _2}\cdots \sigma _{\Omega _N},
\end{equation}
is the same reduced expression as in \eqref{eq:def_of_favourite_w0}, i.e. $\Omega _1=1,\ \Omega _2=2,\ \Omega _3=1,\ \ldots ,\ \Omega _N=1.$ 

\subsubsection{Building the path}
Let $v$ be any element of the highest weight crystal $\Cr _{\lambda }.$ Recall that for any vertex $w\in \Cr _{\la  }$ and any $1\leq i\leq r,$ we may have either $f_i(w)\in \Cr_{\la  },$ in which case $\wght(f_i(w))=\wght(w)-\alpha _i,$ or $f_i(w)=0.$ Let $b_1:=\varphi _1(v),$ i.e. let $b_1$ be a largest integer such that $f^{b_1}_{\Omega _1}v\neq 0.$ Let $v_1=f^{b_1}_{\Omega _1}v,$ and similarly for $i=2,\ldots ,N$ let $b_i$ be the largest integer such that $(v_i:=)f^{b_i}_{\Omega _i}v_{i-1}\neq 0$ (i.e. $b_i:=\varphi _{\Omega _i}(v_{i-1})$).
We may write these integers into an array. 
\begin{equation}\label{eq:def_BZL(v)}
 BZL(v)=BZL_{\Omega }(v)=\left[\begin{array}{ccccc}
         b_{\binom{r}{2}+1}& b_{\binom{r}{2}+2} & \cdots & b_{\binom{r+1}{2}} \\
          & b_{\binom{r-1}{2}+1} & \cdots &b_{\binom{r}{2}}\\
          & \ddots &  & \\
          & & b_2 & b_3 \\
          & & & b_1 
        \end{array}
\right]
\end{equation}

\begin{example}\label{example:BZLofelement220}
 Let $r=2,$ and $\la =(3,1,0).$ We have $w_0=\sigma _1\sigma _2\sigma _1.$ Let $v=v_{(2,2,0)}$ be the single vertex of $\Cr _{(3,1,0)}$ with $\wght(v)=(2,2,0)$ (see Figure \ref{fig:310crystal} and Figure \ref{fig:310weights}). Then $b_1=1,$ $b_2=3$ and $b_3=1,$ and the $BZL$ array of $v$ is 
\begin{equation}\label{eq:BZL(example)}
 BZL(v)=\left[\begin{array}{cc}
                       3 & 1 \\
                        & 1 \\
                      \end{array}
\right]. 
\end{equation}
Notice that this is the same as the $\Gamma $-array in \eqref{eq:Gamma_example}. The $BZL$ path corresponding to $v$ is as shown on Figure \ref{fig:BZLpath}. 

\begin{figure}[h!]
   \centering
    \includegraphics[width=0.4\textwidth]{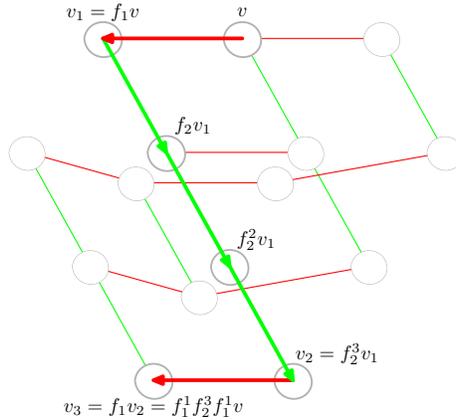}
\caption{The Berenstein-Zelevinsky-Littelmann path of $v_{(2,2,0)}\in \Cr_{(3,1,0)}.$}
\label{fig:BZLpath}
\end{figure}
\end{example}

\subsubsection{Correspondence of crystals and patterns}


\begin{prop}\label{prop:crystal_parametrization}
 Let $\la =(\la _1,\la _2,\ldots ,\la _r,\la _{r+1})$ be a dominant weight, $\Cr _{\la }$ the crystal with highest weight $\la .$ 
\begin{enumerate}[(i)]
 \item For any $v\in \Cr _{\la }$ the $BZL$-path of $v$ ``ends'' in the lowest element $v_{lowest}\in \Cr _{\la },$ i.e. $v_{\binom{r+1}{2}}=v_{lowest}.$
 \item A vertex $v$ can be recovered from $BZL(v).$
 \item For any $v\in \Cr _{\la}$ and $BZL(v)=(b_i)_{1\leq i\leq \binom{r+1}{2}}$ as above, we have 
\begin{equation}\label{eq:weight_of_v_from_BZL}
 \wght(v)-\wght(v_{lowest})=\sum _{i=1}^{\binom{r+1}{2}}b_i\cdot \alpha _{\Omega _i}.
\end{equation}
 \item Elements of the crystal $\Cr _{\la }$ are in bijection with Gelfand-Tsetlin patterns with top row $\la .$ The correspondence is given by assigning $\IP (v)$ to $v$ if and only if $BZL(v)=\Ga (\IP ).$
\item With the correspondence as above, we have $\wght(v)=\wght(\IP (v)).$
\end{enumerate}
\end{prop}
\begin{proof}
 Parts of this proposition are proved throughout Chapter $2$ of \cite{bbf-wmdbook}. In particular, \cite[Lemma 2.1]{bbf-wmdbook} proves (i) and (ii); \cite[Proposition 2.3]{bbf-wmdbook} proves (iii) and (v). The correspondence in (iv) is proved using Young-tableaux. 
Some of the relevant proofs in \cite{bbf-wmdbook} use Berenstein and Zelevinsky \cite{berenstein1993string,berenstein1996canonical}, Kirillov and Berenstein \cite{kirillov1996groups}, Littelmann \cite{littelmann1998cones} and Lusztig \cite{lusztig1990canonical} as a reference.
\end{proof}

\subsection{Gelfand-Tsetlin coefficients}\label{section:GTcoeffs}

In this section, we define the coefficients appearing on the right-hand side of Theorem \ref{THM:MAIN}. The definitions depend on a positive integer $n$ (the degree of the metaplectic cover), the corresponding Gauss sums $g^{\flat }(a)$ and $h^{\flat }(a)$ defined in section \ref{subsubsect:Gauss_sums}, and the decorations of arrays introduced in Definition \ref{def:array_decorations}.



By remark \ref{rmk:Gamma_determines_pattern}, a pattern $\IP $ can be recovered from $\Gamma (\IP )$ and the top row $\la  .$ Since many computations in the sequel involve a fixed $n$ and $\la$, we often suppress these from the notation. We write $G^{(n,\la )}(\IP)=G^{(n)}(\IP)=G(\IP )$ when $\IP$ is understood to be a pattern with top row $\la  .$ We write  $G(\IP)=G^{(n,\la )}(\Ga )=G^{(\la )}(\Ga )=G(\Ga )$
when $\Ga =\Ga (\IP),$ and $G^{(n,\la )}(v)=G^{(\la )}(v)=G(v)$
when $v\in \Cr _{\la  }$ corresponds to $\IP $ by Proposition \ref{prop:crystal_parametrization}. 

\begin{definition}\label{def:GTcoeff}
 Let $\IP$ be a Gelfand-Tsetlin pattern with top row $\la  ,$ $\Ga (\IP )=(\Ga _{ij})_{1\leq i\leq j\leq r}$ its $\Ga $-array as in \eqref{eq:defn:Gamma_array}. Then the degree $n$ Gelfand-Tsetlin coefficient corresponding to $\IP $ is
\begin{equation}\label{eq:GTcoeffdef}
 G^{(n)}(\IP)=\prod _{1\leq i\leq j\leq r} g_{ij}^{n}(\IP),
\end{equation}
where $g_{ij}(\IP)=g_{ij}^{(n)}(\IP)$ is given below. 
\begin{equation}\label{eq:GTfactorforGammadef_n}
g_{ij}(\IP) =
 \left\lbrace\begin{array}{lll}
                            1 &  \Gamma_{i,j+1}=\Gamma _{ij}<\Gamma_{i,j+1}+a_{i-1,j-1}-a_{i-1,j},&  {\text{i.e. $a_{ij}$ is undecorated}}\\
                            h^{\flat }(\Ga _{ij}) &  \Gamma_{i,j+1}<\Gamma _{ij}<\Gamma_{i,j+1}+a_{i-1,j-1}-a_{i-1,j},&  {\text{i.e. $a_{ij}$ is circled}}\\
                            g^{\flat }(\Ga _{ij}) & \Gamma_{i,j+1}<\Gamma _{ij}=\Gamma_{i,j+1}+a_{i-1,j-1}-a_{i-1,j},&  {\text{i.e. $a_{ij}$ is boxed}}\\
                            0  & \Gamma_{i,j+1}=\Gamma _{ij}=\Gamma_{i,j+1}+a_{i-1,j-1}-a_{i-1,j}, &  {\text{i.e. $a_{ij}$ is circled and boxed}}
                                                                                        \end{array}\right.
 \end{equation}
\end{definition}

The coefficient depends strongly on $n.$ To elucidate this, we give the examples of the nonmeatplectic case ($n=1$) and the simplest metaplectic case ($n=2$) explicitly below. Recall from section \ref{subsubsect:Gauss_sums} that $t^n=v=q^{-1},$ where $q$ is the cardinality of a residue field $\oo _S/p\oo _S.$  

\begin{example}\label{example:n=1GTdef}
 When ${{n=1}}$, the factors ${{g_{ij}^{(n)}(\IP)}}$ of the Gelfand-Tsetlin coefficient $G^{(n)}(\IP )$ are as follows. 
\begin{equation}\label{eq:GTfactorforGammadef_1}
 g_{ij}^{(1)}(\Ga) =\left\lbrace\begin{array}{ll}
                            1 & \Gamma_{i,j+1}=\Gamma _{ij}<\Gamma_{i,j+1}+a_{i-1,j-1}-a_{i-1,j}\\
                            1-t & \Gamma_{i,j+1}<\Gamma _{ij}<\Gamma_{i,j+1}+a_{i-1,j-1}-a_{i-1,j}\\
                            -t & \Gamma_{i,j+1}<\Gamma _{ij}=\Gamma_{i,j+1}+a_{i-1,j-1}-a_{i-1,j}\\
                            0 & \Gamma_{i,j+1}=\Gamma _{ij}=\Gamma_{i,j+1}+a_{i-1,j-1}-a_{i-1,j}
                                                                                        \end{array}\right.
\end{equation}
Let us compute the Gelfand-Tsetlin coefficient of the pattern in Example \ref{example:patternofweight220}. Recall that this pattern corresponds to the single element of $\Cr _{(3,1,0)}$ of weight $(2,2,0).$
$$\IP (v_{(2,2,0)}) =\left(\begin{array}{ccccc}
              3 & & 1 & & 0 \\
               & 3 &  & 1 & \\
               & & 2 & &  
             \end{array}
\right), \hskip .5 cm \text{and}\hskip .5 cm \Gamma (\IP (v_{(2,2,0)}))=\left[\begin{array}{cc}
                       3 & 1 \\
                        & 1 \\
                      \end{array}
\right].$$
Here $a_{11}$ and $a_{12}$ (or $\Ga _{11}$ and $\Ga _{12}$) are boxed, while $a_{22}$ (or $\Ga _{22}$) is undecorated. Thus we have $G^{(1)}(\IP(v_{(2,2,0)}) )=(-t)^2\cdot (1-t).$
\end{example}

\begin{example}\label{example:n=2GTdef}
Let $n=2.$ Then the factors ${{g_{ij}^{(n)}(\IP)}}$ of the Gelfand-Tsetlin coefficient $G^{(n)}(\IP )=G^{(n)}(\Ga )$ are 
\begin{equation}\label{eq:GTfactorforGammadef_2}
 g_{ij}^{(2)}(\Ga) =\left\lbrace\begin{array}{lll}
                                        1 & \Gamma _{i,j+1}=\Gamma _{i,j}<\Gamma _{i,j+1}+a_{i-1,j-1}-a_{i-1,j}& \Ga_{ij}\ {\mathrm{is\ circled}}\\
                                        1-t^2 & \Gamma _{i,j+1}<\Gamma _{i,j}<\Gamma _{i,j+1}+a_{i-1,j-1}-a_{i-1,j};\ 2\mid \Gamma_{i,j}& \Ga_{ij}\ {\mathrm{is\ undecorated}}\\
                                        0 & \Gamma _{i,j+1}<\Gamma _{i,j}<\Gamma _{i,j+1}+a_{i-1,j-1}-a_{i-1,j};\ 2\nmid \Gamma_{i,j}& \Ga_{ij}\ {\mathrm{is\ undecorated}}\\
                                        -t^2 & \Gamma _{i,j+1}<\Gamma _{i,j}=\Gamma _{i,j+1}+a_{i-1,j-1}-a_{i-1,j};\ 2\mid \Gamma_{i,j}& \Ga_{ij}\ {\mathrm{is\ boxed}}\\
                                        t & \Gamma _{i,j+1}<\Gamma _{i,j}=\Gamma _{i,j+1}+\la_{j}-\la_{j+1}+1;\ 2\nmid \Gamma_{i,j}& \Ga_{ij}\ {\mathrm{is\ boxed}}\\
                                        0 & \Gamma _{i,j+1}=\Gamma _{i,j}=\Gamma _{i,j+1}+\la_{j}-\la_{j+1}+1;& \Ga_{ij}\ {\mathrm{is\ circled,\ boxed}}
                           \end{array}\right.
\end{equation}
Notice that the factors depend on the residue of $\Gamma _{ij}$ modulo $n=2.$ Returning to the example of $v_{(2,2,0)}\in \Cr_{(3,1,0)},$ we see that since $\Ga _{22}=1$ is undecorated and odd, 
$G^{(2)}(\IP(v_{(2,2,0)}) )=t^2\cdot 0=0.$
\end{example}

\subsection{Branching properties}\label{subs:crystal_branching}

The following branching rule of type $A_r$ highest-weight crystals is well known. (See, for example, \cite[(2.4)]{bbf-wmdbook}.) We shall adapt it to the metaplectic setting, and Demazure crystals in section \ref{section:DemCrystals}; these adapted branching rules play a key role in the proof of Theorem \ref{THM:MAIN}. 

\begin{prop}\label{prop:crystal_branching}
 When all the edges of a highest weight crystal $\Cr _{\la+\rho }$ labelled by $r$ are removed, the connected components of the result are all isomorphic to highest weight crystals $\Cr _{\mu }$ of type $A_{r-1}.$ Omitting the last component of $\wght :\Cr _{\la+\rho } \ra \Z^{r+1},$ and restricting it to a connected component gives the weight function on that component:
\begin{equation}\label{eq:def_weight_on_component}
 \wght _{\mu } :\Cr_{\mu }\ra \Z ^{r}.
\end{equation}
The highest weights $\mu $ that appear in this decomposition are dominant and interleave with $\la +\rho .$ We identify the highest weight crystal $\Cr _{\mu }$ with the appropriate subcrystal of $\Cr _{\la+\rho }.$
That is, we have 
\begin{equation}\label{eq:crystal_branching}
 \Cr _{\la +\rho }=\bigcup _{\mu } \Cr _{\mu }.
\end{equation}
and the (disjoint) union is over all $\mu =(\mu _1,\mu _2\ldots ,\mu _r)$ such that 
\begin{equation}\label{eq:interleave_condition}
 \la _1+r \geq \mu _1\geq \la _2+r-1\geq \cdots \geq \la _r+1\geq \mu _r\geq \la _{r+1}.
\end{equation}
An element $v\in \Cr _{\la +\rho }$ belongs to $\Cr _{\mu }$ in the disjoint union \eqref{eq:crystal_branching} if the second row of the pattern $\IP (v)$ is $(a_{11},a_{12},\ldots ,a_{1r})=\mu .$ (Here $\IP (v)$ is the Gelfand-Tsetlin pattern with top row $\la +\rho $ corresponding to $v$ as in Proposition \ref{prop:crystal_parametrization}.)
\end{prop}

\begin{example}\label{example:crystal_branching}
 If $\la +\rho =(3,1,0),$ then the weights $\mu $ are $(3,1),$ $(3,0),$ $(2,1),$ $(2,0),$ $(1,1)$ and $(1,0).$ Figure \ref{fig:Crystal_310_branching} shows the corresponding components of $\Cr_{(3,1,0)}.$ These are of Cartan type $A_1.$ The highest element of each string is labeled with the corresponding weight $\mu .$ 
\begin{figure}[h!]
    \centering
    \includegraphics[width=0.4\textwidth]{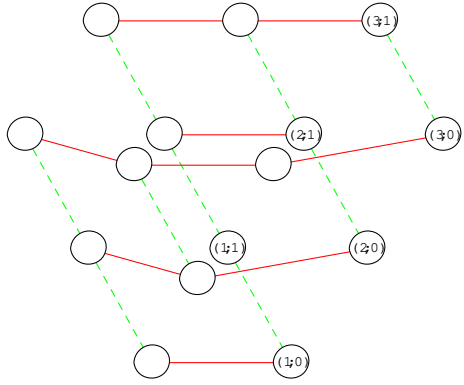}
\caption{The $A_1$ components ($1$-strings) of $\Cr_{(3,1,0)}.$}
\label{fig:Crystal_310_branching}
\end{figure}
\end{example}

The remainder of this section is dedicated to describing the weights and Gelfand-Tseltlin coefficients on components in \eqref{eq:crystal_branching} explicitly. The results are summarized in Proposition \ref{prop:banchcomp_explicit} below. As before, we use the notation $\x=(x_1,\ldots ,x_r,x_{r+1}),$ write $\y=(x_1,\ldots ,x_r)$ and let $d(\la)$ (or $d(\mu )$) denote the sum of the components of the weight $\la $ (respectively, $\mu $).

\begin{prop}\label{prop:banchcomp_explicit}
 Let $\Cr _{\mu }$ be one of the components in in the decomposition \eqref{eq:crystal_branching} of $\Cr_{\la +\rho},$ i.e. suppose $\mu $ and $\la +\rho $ interleave. Let $v$ be any element of $\Cr _\mu .$ Then we have the following. 
\begin{enumerate}[(a)]
\item If $\wght_{\mu }:\ \Cr_{\mu }\ra \Z^{r}$ denotes the weight function on $\Cr _{\mu },$ then
\begin{equation}\label{eq:wt_conversion_goal}
 \x^{\wght(v)}=\y^{\wght _{\mu }(v)}\cdot x_{r+1}^{d(\la +\rho )-d(\mu )}.
\end{equation}
\item Let $v_{\ast }$ denote the lowest element of $\Cr _{\mu}$ (as a type $A_{r-1}$ crystal).
\begin{equation}\label{eq:GT_conversion_goal_cr}
 G^{(n,\la +\rho )}(v)=G^{(n,\mu )}(v)\cdot G^{(n,\la +\rho )}(v_{\ast }).
\end{equation}
\end{enumerate}
\end{prop}
\begin{proof}
Let $\IP (v)=\IP _{\la+\rho } (v)$ be the Gelfand-Tsetlin pattern corresponding to $v\in \Cr _{\la +\rho }$ as in Proposition \ref{prop:crystal_parametrization}. By Proposition \ref{prop:crystal_branching}, as an element of $\Cr _{\mu },$ $v$ corresponds to the pattern $\IP _{\mu }(v),$ and $\IP _{\mu }(v)$ is the same as $\IP _{\la+\rho } (v)$ with its first row omitted. In particular, the second row of $\IP (v)$ is $\mu .$ Thus by \eqref{eq:def_di}, $d_0(\IP(v))=d(\la )$ and $d_1(\IP (v))=d(\mu).$ By \eqref{eq:def_pattern_weight}, the last coordinate of $\wght(\IP (v))$ is $d(\la )-d(\mu ).$ This implies (a). For (b), we restrict our attention to weights $\mu $ that are strongly dominant, i.e. we would like to assume $\mu _1>\mu _2>\cdots >\mu _r.$ We can do this because of the following remark. 

\begin{remark}\label{rmk:nonstrict_mu}
 The statement \eqref{eq:GT_conversion_goal_cr} is trivial if $\mu $ is not strongly dominant. By Remark \ref{rmk:strictpatterns}, the Gelfand-Tsetlin coefficient of a non-strict pattern is zero. The second row of $\IP (v)$ is $\mu ,$ hence if $\mu $ is not strongly dominant, then $\IP (v)$ is non-strict for any $v\in \Cr _{\mu },$ and $G^{(n,\la +\rho )}(v)=G^{(n,\la +\rho )}(v_{\ast })=0.$
\end{remark}

Assume that $\mu $ is strongly dominant, hence the first two rows of $\IP (v)$ are strict for every $v\in \Cr _{\la +\rho}.$ The next remark describes the $BZL$ path of elements in $\Cr_{\mu}.$ 

\begin{remark}\label{rmk:bzltolowest}
Recall from section \ref{section:BZLpaths} that the $BZL$ path of $v$ is a path in $\Cr _{\la +\rho }$ from $v$ to $v_{lowest}\in \Cr_{\la +\rho }.$ The $j$th segment of the path is along an edge of $\Cr _{\la +\rho }$ labelled by $\Omega _j,$ where $\Omega _j$ is defined in \eqref{eq:def_of_favourite_w0}. The chosen long word $w_0^{(r)}$ starts with $w_0^{(r-1)};$ in particular the first $\binom{r}{2}$ out of the $\binom{r+1}{2}$ segments are along edges {\em{not}} labelled by $r.$ This implies that these segments are contained in $\Cr_{\mu },$ and in fact are the $BZL$ path corresponding to $v$ as an element of $\Cr_{\mu },$ a crystal of type $A_{r-1}$. Hence the end of the first $\binom{r}{2}$ segments is the lowest element of that crystal, $v_{\ast }.$ Consequently, the first $\binom{r}{2}$ segments of the $BZL$ path of $v_\ast $ are trivial. If 
$b_j(v)=b_j(\IP(v))$ denotes the length of the $j$th segment of the $BZL$ path, then 
\begin{equation}\label{eq:bofvandvast}
b_j(v_{\ast})=0\text{ for }j\leq \binom{r}{2},\text{ and }b_j(v)=b_j(v_{\ast})\text{ for every }v\in \Cr_{\mu }\text{ and } j>\binom{r}{2}.
\end{equation}
\end{remark}

Thus $BZL(v_\ast )=\Ga (\IP (v_\ast ))$ has zeros everywhere below the first row. By \eqref{eq:defn:Gamma_entries}, this means that for any $1<i\leq j\leq r+1,$ the entry $a_{i,j}$ of $\IP (v_{\ast })$ satisfies $a_{i,j}=a_{1,j}=\mu _j.$ It follows that we have $a_{i-1,j-1}>a_{i,j}=a_{i-1,j}.$ According to Definition \ref{def:array_decorations}, these entries are all circled, but not boxed. By Definition \ref{def:GTcoeff} this implies 
\begin{equation}\label{eq:GTcoeffofvast}
G^{(n,\la +\rho )}(\IP (v_{\ast }))=\prod _{1\leq i\leq j\leq r} g_{ij}^{n,\la +\rho }(\IP (v_{\ast }))=\prod _{1\leq j\leq r} g_{1j}^{n,\la +\rho }(\IP (v_{\ast }))\cdot \prod _{2\leq i\leq j\leq r} 1.
\end{equation}

The first two rows of $\IP (v)$ are $\la +\rho $ and $\mu $ for every $v\in \Cr _\mu .$ The coefficient $g_{1j}^{n,\la +\rho }(\IP (v))$ only depends on those two rows; hence its value is the same for any $v\in \Cr _\mu $ and $v_\ast .$ Hence we have 
\begin{equation}\label{eq:GT_component_element_1}
G^{(n,\la +\rho )}(\IP (v))  =  \prod _{1\leq i\leq j\leq r} g_{ij}^{n,\la +\rho }(\IP (v)) =  \prod _{1\leq j\leq r} g_{1j}^{n,\la +\rho }(\IP (v_{\ast }))\cdot \prod _{2\leq i\leq j\leq r} g_{ij}^{n,\la +\rho }(\IP (v)) 
\end{equation}
The first product here is equal to $G^{(n,\la +\rho )}(\IP (v_{\ast }))$ by \eqref{eq:GTcoeffofvast}. The second factor, 
\begin{equation}\label{eq:lastrowsmuGT}
\prod _{2\leq i\leq j\leq r} g_{ij}^{n,\la +\rho }(\IP (v))=G^{(n,\mu )}(v),
\end{equation} 
because, as seen above, the Gelfand-Tsetlin pattern $\IP _{\mu }(v)$ corresponding to $v$ as an element of $\Cr _{\mu }$ is $\IP (v)=\IP _{\la +\rho }(v)$ minus its first row. Thus, substituting \eqref{eq:GTcoeffofvast} and \eqref{eq:lastrowsmuGT} into \eqref{eq:GT_component_element_1} gives \eqref{eq:GT_conversion_goal_cr}.
\end{proof}

%% file: Tokuyama_rephrase.tex
\section{Tokuyama's Theorem}\label{section:Tokuyama}

Tokuyama's theorem, in its original form, relates a Schur function to a generating function of strict Gelfand patterns. This is easily rephrased to relate a sum over a Weyl group to a sum over a highest weight crystal. This second form is more convenient for the purposes of generalizing the theorem to the metaplectic setting.

In the previous section, we followed notation from Brubaker-Bump-Friedberg \cite{bbf-wmdbook}, because that is most convenient to use for metaplectic definitions of Gelfand-Tsetlin coefficients. The notation and approach in Tokuyama's paper \cite{tokuyama-generating} is slightly different. Here we phrase Tokuyama's theorem using both sets of notation, and explain why the two versions are equivalent. 

Let $\x=(x_1,\ldots ,x_{r+1}),$ $\z=(z_1,\ldots ,z_{r+1}),$ $\la =(\la _1,\ldots ,\la _{r+1})$ and let $\rho =(r,r-1,\ldots ,1,0)$
be the Weyl vector. Let $s_{\la }(\x )$ (or $s_{\la }(\z )$) denote the Schur function associated to the highest-weight representation of $\GL _{r+1}$ with highest weight $\la .$ Recall that a Gelfand-Tsetlin pattern is an array of the form \eqref{eq:GTpatternform}, where rows are non-increasing and interleave.

As in Tokuyama \cite{tokuyama-generating}, we say a pattern $\IP $ is strict if $a_{i-1,j-1}>a_{i-1,j}$ holds for every $1\leq i\leq j\leq r.$ Following notation there, let $\Grm(\la )$ denote the set of Gelfand-Tsetlin patterns with top row $\la ,$ and let $\SG (\la )$ be the set of strict Gelfand-Tsetlin patterns with top row $\la .$

\begin{remark}\label{rmk:strictpatterns}
 Note that by Definition \ref{def:array_decorations}, a Gelfand-Tsetlin pattern $\IP $ is strict if and only it has no entries that are both circled and boxed. In every version of Gelfand-Tsetlin coefficients, such an entry corresponds to a factor of zero. Hence as long as each term of the sum involves the Gelfand-Tsetlin coefficients, summing over $\Grm (\la )$ is the same as summing over $\SG (\la )$.
\end{remark}

Recall that if $d_i$ is the sum of elements in the $i$-th row \eqref{eq:def_di}, then by \eqref{eq:def_pattern_weight} the weight of a pattern $\IP$ is  $\wght(\IP )=(d_r,d_{r-1}-d_r,\ldots ,d_0-d_1).$
In Tokuyama \cite{tokuyama-generating}, we have 
\begin{equation}\label{eq:M(IP)def}
 M(\IP )=(d_0-d_1,d_1-d_2,\ldots ,d_{r-1}-d_r,d_r).
\end{equation}
For a weight $\mu =(\mu _1,\mu _2,\ldots ,\mu _{r+1})$ write $\x^{\mu }=x_1^{\mu _1}\cdot x_2^{\mu _2}\cdots x_{r+1}^{\mu _{r+1}}.$
Recall the definition of the (nonmetaplectic) Gelfand-Tsetlin coefficient $G(\IP )=G^{(1)}(\IP )$ as a product of $g_{ij}(\IP )$ from \eqref{eq:GTcoeffdef} and \eqref{eq:GTfactorforGammadef_1}. Let us treat $t$ as an indeterminate for the time being. Then the factor $g_{ij}(\IP)$ corresponding to an entry $a_{ij}$ is as follows.
$$ g_{ij}(\IP) =\left\lbrace\begin{array}{lll}
                            1 & a_{i-1,j}=a_{ij}& a_{ij}\ {\mathrm{is\ circled}}\\
                            1-t & a_{i-1,j}<a_{ij}<a_{i-1,j-1}& a_{ij}\ {\mathrm{is\ undecorated}}\\
                            -t & a_{i-1,j}<a_{ij}=a_{i-1,j-1}&  a_{ij}\ {\mathrm{is\ boxed.}}\\
                            0 & a_{i-1,j}=a_{ij}=a_{i-1,j-1}&  a_{ij}\ {\mathrm{is\ circled\ and\ boxed.}}
                                                                                        \end{array}
\right.$$
In Tokuyama \cite{tokuyama-generating}, the entry $a_{ij}$ is called ``special'' if $a_{i-1,j}<a_{ij}<a_{i-1,j-1}$ and ``lefty'' if $a_{i-1,j}=a_{ij}.$ By Defintion \ref{def:array_decorations} ``special'' entries are undecorated, and ``lefty'' entries are boxed. (For strict patterns, ``lefty'' entries are boxed and not circled by Remark \ref{rmk:strictpatterns}.)

We are now ready to state Tokuyama's theorem in both the notation of \cite{tokuyama-generating} and in ours. 

\begin{theorem}\label{thm:Tokuyama}
 (Tokuyama's theorem) Let $\la =(\la _1,\la _2,\ldots ,\la _r,\la _{r+1})\in \Z^{r+1}$ where $\la _1\geq \la _2\geq \cdots \geq \la _{r+1}\geq 0,$ and $\rho =(r,r-1,\ldots ,1,0),$ $\SG(\la +\rho ),$ and $M(\IP )$ as defined above. Let $s(\IP )$ be the number of special entries of $\IP $ and $l(\IP )$ the number of lefty entries. 
\begin{equation}\label{eq:Tokuyama_thm_original_notation}
 s_{\la }(\z )\cdot \prod _{1\leq i<j<r+1} (z_i-t\cdot z_j)=\sum _{\IP \in \SG(\la +\rho )} (1-t)^{s(\IP )}\cdot (-t)^{l(\IP )}\cdot \z^{M(\IP )}.
\end{equation}
In the notation introduced in previous sections of this chapter, this can be re-written as
\begin{equation}\label{eq:Tokuyama_thm_our_notation}
 s_{\la }(\x )\cdot \prod _{1\leq i<j<r+1} (x_j-t\cdot x_i)=\sum_{\IP\in \Grm(\la +\rho )} G(\IP )\cdot \x ^{\wght(\IP )}.
\end{equation}
\end{theorem}

 The first form of the equation, \eqref{eq:Tokuyama_thm_original_notation} is Theorem 2.1 of \cite{tokuyama-generating}, substituting $-t$ for $t.$ We explain why \eqref{eq:Tokuyama_thm_our_notation} is equivalent. Note that (thinking of $t$ as an indeterminate) for any strict Gelfand-Tsetlin pattern we have $G(\IP )=(1-t)^{s(\IP )}\cdot (-t)^{l(\IP )}$
Furthermore, by Remark \ref{rmk:strictpatterns} we have $G(\IP )=0 $ if $\IP \in \Grm(\la +\rho )\setminus \SG(\la +\rho ).$
From \eqref{eq:def_pattern_weight} and \eqref{eq:M(IP)def} we see that the components of $\wght(\IP )$ are exactly the components of $M(\IP)$ in reverse order. So if we write $x_1=z_{r+1},$ $x_2=z_r,$ $\ldots ,$ $x_r=z_2,$ $x_{r+1}=z_1,$
we have $\x^{\wght(\IP )}=\z ^{M(\IP )}.$
Hence the right hand sides of \eqref{eq:Tokuyama_thm_original_notation} and \eqref{eq:Tokuyama_thm_our_notation} agree. It remains to check that the left hand sides agree as well. Note that with the choice $x_i=z_{r+2-i}$, we have $s_{\la }(\x )=s_{\la }(\z )$  and $\prod _{1\leq i<j<r+1} (z_i-t\cdot z_j)=\prod _{1\leq i<j<r+1} (x_j-t\cdot x_i).$


In the remainder of this section, we reformulate Theorem \ref{thm:Tokuyama} in terms of Demazure-Lusztig operators and a sum over a highest-weight crystal. This is done separately for the two sides. 

\subsection{The right hand side of Tokuyama's theorem as a sum over a crystal}\label{subsect:TokuyamaRHS}

The correspondence between elements of a crystal of highest weight $\la +\rho $ and Gelfand-Tsetlin patterns of top row $\la + \rho $ was established by Proposition \ref{prop:crystal_parametrization}. Following that parametrization and notation, we may write $G(v)=G^{(1)}(v):=G^{(1)}(\IP (v));$ we have $\wght (\IP (v))=\wght (v)$. Thus we may write the right hand side of \eqref{eq:Tokuyama_thm_our_notation} as
\begin{equation}\label{eq:RHS_Tokuyama_rewrite}
 \sum_{\IP\in \Grm(\la +\rho )} G(\IP )\cdot \x ^{\wght(\IP )}=\sum_{v\in \Cr _{\la +\rho }} G(v)\cdot \x ^{\wght(v)}.
\end{equation}

\subsection{The left hand side of Tokuyama's theorem in terms of Demazure-Lusztig operators}\label{subsect:TokuyamaLHS}

Recall notation for type $A$ root systems in Section \ref{subs:CGDemNotation}. This notation and the Weyl Character Formula for Schur functions allows us to rewrite the left hand side of \eqref{eq:Tokuyama_thm_our_notation} first as a sum over the Weyl group. Then we use Theorems \ref{THM:LONG_WORD} and \ref{THM:T_SUM} to write it in terms of Demazure-Lusztig operators. 

The Weyl group $W\cong S_{r+1}$ acts on $\La =\Z^{r+1}$ by permuting the coordinates. Thus for the long element $w_0$ we have 
\begin{equation}
\prod _{1\leq i<j<r+1} (x_j-t\cdot x_i) =\x^{w_0.\rho }\cdot \prod _{\alpha \in \Phi ^+} (1-t\cdot \x^{\alpha }).
\end{equation}
Now by the Weyl Character Formula we have 
$$s_{\la }(\x )=\frac{\sum _{w\in W} \sgn (w)\cdot w.(\x^{\la +\rho })}{\prod _{1\leq i<j<r+1} (x_i-x_j)}=\frac{\sum _{w\in W} \sgn (w)\cdot w.(\x^{\la +\rho })}{\x^{w_0\rho }\cdot \sgn(w_0)\cdot \prod _{\alpha \in \Phi ^+} (1-\x^{\alpha })}.$$
Thus the left hand side of \eqref{eq:Tokuyama_thm_our_notation} can be rewritten as 
 $$s_{\la }(\x )\cdot \prod _{1\leq i<j<r+1} (x_j-t\cdot x_i)=\frac{\prod _{\alpha \in \Phi ^+}(1-t\cdot \x^{\alpha })}{\prod _{\alpha \in \Phi ^+}(1-\x^{\alpha })}\cdot \sgn(w_0)\cdot \sum _{w\in W} \sgn (w)\cdot w.(\x^{\la +\rho }).$$
 Recall the notation $\ttDelta$ defined in \eqref{eq:deformeddenom_metaplectic} for a deformation the Weyl denominator;
write $\Delta $ for $\ttDelta $ when $t=1.$ Then we have 
\begin{equation}\label{eq:LHS_Tokuyama_rewrite_la+rho}
 s_{\la }(\x )\cdot \prod _{1\leq i<j<r+1} (x_j-t\cdot x_i)=\frac{\ttDelta}{\Delta }\cdot \sgn(w_0)\cdot \sum _{w\in W} \sgn (w)\cdot w.(\x^{\la +\rho }).
\end{equation}


Substituting $ww_0$ for $w$ we may write \eqref{eq:LHS_Tokuyama_rewrite_la+rho} as 
$$s_{\la }(\x )\cdot \prod _{1\leq i<j<r+1} (x_j-t\cdot x_i)=\frac{\ttDelta}{\Delta }\cdot \sum _{w\in W} \sgn (w)\cdot w.(\x^{w_0(\la +\rho )})$$
Now, to write this as a linear combination of Weyl group elements acting on $\x^{w_0\la },$ notice that as operators, 
$w\cdot \x^{w_0\rho }=(w.\x^{w_0\rho })\cdot w=\x^{ww_0\rho }\cdot w.$
With $\Phi(w)$ as in \eqref{def:Phi_w} we have $\Phi (w^{-1})=w(\Phi ^-)\cap \Phi ^+=ww_0(\Phi ^+)\cap \Phi ^+,$ and hence $ww_0\rho -w_0\rho =\sum _{\alpha \in \Phi(w^{-1})}\alpha ,$
whence as operators, $w\cdot \x^{w_0\rho }=\x^{w_0\rho }\cdot \prod _{\alpha \in \Phi (w^{-1})}\x^{\alpha }\cdot w.$
Thus we may rewrite \eqref{eq:LHS_Tokuyama_rewrite_la+rho} as 
$$s_{\la }(\x )\cdot \prod _{1\leq i<j<r+1} (x_j-t\cdot x_i)=\x^{w_0\rho }\cdot \frac{\ttDelta}{\Delta }\cdot \sum _{w\in W} \left(\sgn (w)\cdot \prod _{\alpha \in \Phi (w^{-1})}\x^{\alpha }\right)\cdot w.(\x^{w_0(\la )}),$$
and, combining the special cases of Theorem \ref{THM:LONG_WORD} and \ref{THM:T_SUM} for $n=1,$ as 
\begin{equation}\label{eq:LHS_Tokuyama_rewrite_Demazure}
s_{\la }(\x )\cdot \prod _{1\leq i<j<r+1} (x_j-t\cdot x_i)=\x^{w_0\rho }\cdot \ttDelta\cdot \Dem _{w_0}(\x^{w_0(\la )})=\x^{w_0\rho }\cdot \sum _{w\in W} \T _w(\x^{w_0(\la )}).
\end{equation}

\begin{remark}\label{rmk:rho_illdefined}
 In this section, we used ``the Weyl vector'' as $\rho =(r,r-1,\ldots ,1,0).$ As a result, we have 
$$\x^{\rho }=(x_1\cdot x_2\cdots x_{r+1})^r\cdot \prod _{\alpha \in \Phi ^{+}} \x^{\frac{1}{2}\alpha }.$$
In the computations above, the factor $(x_1\cdot x_2\cdots x_{r+1})^r$ was never written out explicitly, but all of the equations hold as written. The factor $x_1\cdot x_2\cdots x_{r+1}$ is symmetric under $W,$ so as an operator, it commutes with any element of the Weyl group.
\end{remark}

\subsection{The crystal version of Tokuyama's theorem}\label{subsect:TokuyamaNewVersion}

The following ``crystal version'' of Tokuyama's theorem is a direct consequence of \eqref{eq:Tokuyama_thm_our_notation}, \eqref{eq:LHS_Tokuyama_rewrite_Demazure} and \eqref{eq:RHS_Tokuyama_rewrite}. 

\begin{lemma}\label{lem:Tokuyama_thm_crystal_Demazure}
 Let $\la $ be a dominant, effective weight. Then we have 
\begin{equation}\label{eq:Tokuyama_thm_crystal_version}
\sum _{w\in W} \T _w(\x^{w_0(\la )})=\x^{-w_0\rho }\cdot \sum_{v\in \Cr _{\la +\rho }} G(v)\cdot \x ^{\wght(v)}.
\end{equation}
\end{lemma}

Lemma \ref{lem:Tokuyama_thm_crystal_Demazure} is the form of Tokuyama's theorem that is convenient to generalize, as we will see in Section \ref{chap:equivalent_reformulations}. 

%% file: DemCrystals.tex
\section{Demazure crystals and branching properties}\label{section:DemCrystals}

The statement of Theorem \ref{THM:MAIN} involves, on one side, a sum over a Demazure crystal. In this section we give the definition of Demazure subcrystals $\Cr _{\la +\rho }^{(w)}$ within a type $A$ highest weight crystal $\Cr _{\la +\rho },$ for certain elements $w$ of the Weyl group. In preparation for the proof of Theorem \ref{THM:MAIN}, we discuss how the branching properties of section \ref{subs:crystal_branching} restrict to Demazure crystals (Proposition \ref{prop:Demazure_branching} and Propositon \ref{prop:GTsumbreakupDem}). Finally, we introduce some terminology that will allow for lighter notation in the proof of Theorem \ref{THM:MAIN} in Section \ref{chap:equivalent_reformulations} and Section \ref{chap:technical_induction}.

We start by specifying the set of Weyl group elements $w$ that appear in the statement of Theorem \ref{THM:MAIN}. 

\begin{definition}\label{def:begsectlw}
Recall that in \eqref{eq:def_of_favourite_w0} we fixed the following long word:
 \begin{equation*}
  w_0=w_0^{(r)}=\sigma _1\sigma _2\sigma _1\cdots \sigma _{r-1}\cdots \sigma _1\sigma _r\cdots \sigma _1=\sigma_{\Omega _1}\sigma_{\Omega _2}\cdots \sigma _{\Omega _{\binom{r+1}{2}}}
 \end{equation*}
 We say that the element $w$ is a {\em{beginning section of the long word $w_0^{(r)}$}} if 
 \begin{equation}\label{eq:begsect}
w=\sigma_{\Omega _1}\sigma_{\Omega _2}\cdots \sigma _{\Omega _{l}} \text{ for some }l(=\ell (w))\leq \binom{r+1}{2}.
 \end{equation}
\end{definition}

Sometimes it is convenient to assume that $w$ is a beginning section of $w_0^{(r)},$ but not of $w_0^{(r-1)}.$ In this case $w$ is of the form 
\begin{equation}\label{eq:w_form}
 w=w_0^{(r-1)}\sigma _r\cdots \sigma _{r-k}.
\end{equation}
For such elements $w\in W,$ we write $k:=\ell(w)-\ell(w_0^{(r-1)})-1.$  


Now we are ready to define Demazure crystals corresponding to a beginning section \eqref{eq:begsect} of our favourite long word. 

\begin{definition}\label{def:Demazure_subcrystal}
 Let $\Cr _{\la +\rho }$ be a crystal of highest weight $\la +\rho ,$ and let $w$ be a beginning section of the long word $w_0^{(r)}$, as in \eqref{eq:begsect}. Then the {\em{Demazure crystal corresponding to $w$}} is the crystal $\Cr_{\la +\rho}^{(w)}$ with vertices
\begin{equation}\label{eq:def_of_Demazure_subcrystal}
 \Cr_{\la +\rho}^{(w)}=\left\lbrace v\in \Cr _{\la +\rho } \mid b_i(v)=0 \text{ for all } i>\ell(w) \right\rbrace.
\end{equation}
Here $b_i (v)$ ($1\leq i\leq \binom{r+1}{2}$) denotes the $i$-th entry of the Berenstein-Zelevinsky-Littelmann array $BZL(v)$ of an element $v\in \Cr _{\la +\rho }.$

To define a crystal structure $\Cr_{\la +\rho}^{(w)},$ we contend that as a directed graph it is a full subgraph of $\Cr_{\la +\rho}.$ That is, the edges of $\Cr_{\la +\rho}^{(w)}$  are exactly the edges of $\Cr_{\la +\rho}$ with both ends in $\Cr_{\la +\rho}^{(w)}.$ 
\end{definition}

\begin{remark}\label{rmk:Demcryst_w_BZL}
Definition \ref{def:Demazure_subcrystal} means that an element $v\in \Cr _{\la +\rho }$ belongs to $\Cr_{\la +\rho}^{(w)}$ if and only if the $BZL$-path of $v$ already reaches  the lowest element of the crystal $\Cr _{\la +\rho }$ after the first $\ell(w)$ steps. With $w$ a beginning segment of $w_0^{(r)}$ but not of $w_0^{(r-1)}$ as in \eqref{eq:w_form}, we have $\ell(w)=\binom{r}{2}+k+1,$ so $v\in \Cr _{\la +\rho }^{(w)}$ if and only if 
\begin{equation}\label{eq:BZL_of_Dem_subcrystal}
 BZL(v)=BZL_{\Omega }(v)=\left[\begin{array}{ccccccc}
         b_{\binom{r}{2}+1} & b_{\binom{r}{2}+2} & \cdots & b_{\binom{r}{2}+k+1} & 0 & \cdots &  0 \\
          & b_{\binom{r-1}{2}+1} & & \cdots  & & & b_{\binom{r}{2}}\\
          &  & \ddots & & & &  \\
          & & & & & b_2 & b_3 \\
          & & & & & & b_1 
        \end{array}
\right]
\end{equation}
i.e. $v\in \Cr _{\la +\rho }^{(w)}$ if and only if the last $\ell (w_0^{(r)})-\ell (w)=r-k-1$ segments of the $BZL$ path of $v$ are trivial. 
\end{remark}

The definition is illustrated by the following example, in type $A_2.$

\begin{example}\label{example:Demazure_crystal}
 Recall the crystal $\Cr _{3,1,0}$ of highest weight $(3,1,0)$ from Example \ref{example:310crystal}. The Demazure subcrystal corresponding to $w=\sigma _1\sigma _2$ is the highlighted part of the crystal in Figure \ref{fig:310crystal_Dem12sub}.
\begin{figure}[h!]
    \centering
    \includegraphics[width=0.4\textwidth]{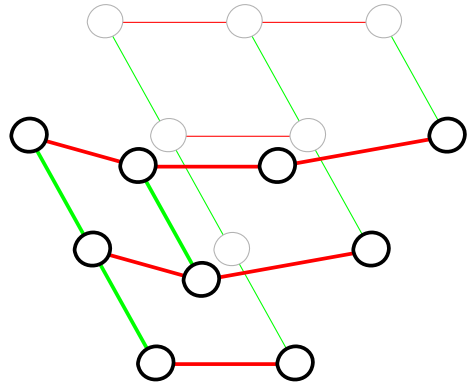}
\caption{The Demazure crystal $\Cr_{(3,1,0)}^{(\sigma_1\sigma _2)}$ within $\Cr_{(3,1,0)}.$}
\label{fig:310crystal_Dem12sub}
\end{figure}
\end{example}

For the remainder of this section, we assume that $w$ is a beginning section of $w_0^{(r)},$ but not of $w_0^{(r-1)}$ i.e. it is as in \eqref{eq:w_form}. Then the Demazure crystal $\Cr _{\la+\rho }^{(w)}$ inherits some of the branching properties discussed in section \ref{subs:crystal_branching}, and in particular, Proposition \ref{prop:banchcomp_explicit}. In particular, the type $A_{r-1}$ subcrystals $\Cr_{\mu }$ from the decomposition \eqref{eq:crystal_branching} are either disjoint from, or contained in $\Cr _{\la+\rho }^{(w)}.$ The set of $\mu $ such that $\Cr _{\mu }$ is contained in $\Cr _{\la+\rho }^{(w)}$ is easy to characterize from $\la $ and $w.$ We make this precise in the proposition below. 

\begin{prop}\label{prop:Demazure_branching}
Let $w=w_0^{(r-1)}\cdots \sigma _r\cdots \sigma _{r-k}$ and $\Cr _{\la+\rho }^{(w)}$ the corresponding Demazure crystal. Let $\Cr _{\mu }\subset \Cr _{\la+\rho }$ be a subcrystal of type $A_{r-1}$ from the decomposition \eqref{eq:crystal_branching}. Then $\Cr _{\mu }$ has either no vertices in $\Cr _{\la+\rho }^{(w)},$ or it is contained in it. Furthermore, for a $\mu $ that interleaves with $\la +\rho ,$ we have $\Cr _{\mu }\subseteq \Cr _{\la+\rho }^{(w)}$ if and only if $\mu _j=\la _{j+1}+r-j$ for $j>k+1 .$ Taking the union over $\mu $ like this, we have 
 \begin{equation}\label{eq:Dem_crystal_branching}
   \Cr _{\la +\rho }^{(w)}=\bigcup _{\mu} \Cr _{\mu }.
 \end{equation}
 \end{prop}
\begin{proof}
Most of the work for the proof has been done in Section \ref{section:HWtCrystalsandGTpatterns}. Note that by Proposition \ref{prop:crystal_branching}, $\Cr _{\mu }$ is a component of $\Cr _{\la+\rho }$ in the decomposition \eqref{eq:crystal_branching} if and only if $\mu $ and $\la +\rho $ interleave. Further, it was noted in Remark \ref{rmk:bzltolowest} that the last $r$ segments of the $BZL$ path agree for every element $v\in \Cr _{\mu }.$ Thus, by Remark \ref{rmk:Demcryst_w_BZL}, either every vertex of $\Cr _{\mu }$ is contained in $\Cr _{\la +\rho }^{(w)},$ or none of them are. To characterize the weights $\mu $ such that $\Cr _{\mu }\subseteq \Cr _{\la +\rho }^{(w)},$ recall that by Proposition \ref{prop:crystal_branching}, $v\in \Cr _{\la +\rho }$ belongs to $\Cr _{\mu }$ if and only if the top two rows of $\IP (v)$ are $\la +\rho $ and $\mu .$ Further, by Proposition \ref{prop:crystal_parametrization}, $\Ga (\IP (v))=BZL (v),$ so the top row of $BZL(v)$ is given by 
$$b_{\binom{r}{2}+j}(v)=\la _{j+1}+r-j-\mu _{j}\text{ for every }1\leq j\leq r.$$
In particular, we have $b_{i}(v)=0$ for every $i>\ell(w)=\binom{r}{2}+k+1$ if $\la _{j+1}+r-j=\mu _{j}$ holds for every $j>k+1.$
\end{proof}

The following proposition relates the right-hand side of Theorem \ref{THM:MAIN} to similar sums over complete highest-weight crystals of a lower rank. It will be key in the proof of that theorem, but is at this point a straightforward consequence of parts (b) and (c) of Proposition \ref{prop:crystal_branching}, and Proposition \ref{prop:Demazure_branching}. 

\begin{prop}\label{prop:GTsumbreakupDem}
Using the notation of Proposition \ref{prop:crystal_branching} and Proposition \ref{prop:Demazure_branching}, let $v_{\ast }^{(\mu )}$ denote the lowest element of the crystal $\Cr _{\mu }.$ Write $\rho _{r-1}=(r-1,\ldots ,1,0).$ Then we have 
\begin{equation}\label{eq:IW_RHS_rewriting_goal_dem}
\begin{split}
  \x^{-w_0(\rho )}\cdot \sum _{v\in \Cr_{\la +\rho }^{(w)}} G^{(n,\la+\rho )}(v)\cdot \x^{\wght(v)}= & \sum _{\mu } G^{(n,\la +\rho )}(v_{\ast })\cdot x_{r+1}^{d(\la +\rho )-d(\mu )-r}\cdot \\
 & \cdot \left(\y^{-w_0^{(r-1)}(\rho _{r-1})}\cdot \sum _{v\in \Cr_{\mu }} G^{(n,\mu )}(v)\cdot \y^{\wght _{\mu }(v)}\right).
\end{split}
\end{equation}
Here the sum is over all $\mu =(\mu _{1},\mu _2,\ldots ,\mu _r)$ that interleave with $\la +\rho $ and $\mu _j=\la _{j+1}+r-j$ for $j>k+1 .$ 
\end{prop}

Since all the notation necessary for our statements is a lot to keep track of, we present an example. 

\subsection{An example}\label{subsect:Dem_example}

Take $r=2,$ $\la +\rho =(3,1,0)$ and $w=\sigma _1\sigma _2$ (so $k=0$). Then $\Cr _{\la+\rho }^{(w)}=\Cr_{(3,1,1)}^{(\sigma _1\sigma _2)}$ is as in Figure \ref{fig:310crystal_Dem12sub}. Under the branching \eqref{eq:crystal_branching} it is the union of three crystals of type $A_1,$ of highest weight $(3,0),$ $(2,0)$ and $(1,0),$ respectively. 

Determining the highest weights $(3,0),$ $(2,0)$ and $(1,0)$ is easy from Proposition \ref{prop:crystal_parametrization} and \eqref{eq:BZL_of_Dem_subcrystal}. Recall that a vertex belongs to the Demazure crystal $\Cr _{\la+\rho }^{(w)}$ if $BZL(v)=\Gamma (\IP (v))$ has zeros in the last $l(w_0)-l_{w}=k+1$ places in the first row. We have $\la +\rho =(a_{00},a_{01},a_{02})=(3,1,0).$ This means that the component $\Cr _{\mu }$ of \eqref{eq:crystal_branching} belongs to $\Cr_{(3,1,1)}^{(\sigma _1\sigma _2)}$ if $\mu =(\mu _1,\mu _2)$ satisfies $\Gamma _2=\mu _2-a_{0,2}=0,$ hence $\mu _2=0.$ Figure \ref{fig:310D12_branching} shows the three components of $\Cr_{(3,1,0)}^{(\sigma _1\sigma _2)}.$ 
Let $n=1.$ The Gelfand-Tsetlin coefficients assigned to the vertices of $\Cr_{(3,1,1)}^{(\sigma _1\sigma _2)}$ can be read off of Figure \ref{fig:310D12IW}. 

\begin{figure}[h!]
\centering
\begin{minipage}{0.42\textwidth}
\centering
    \includegraphics{Crystal_310_Dem12_branching-ps.mps}
\caption{Branching}
\label{fig:310D12_branching}
\end{minipage}\hfill
\begin{minipage}{0.48\textwidth}
\centering
\includegraphics{310D12IW-ps.mps}
\caption{Coefficients}
\label{fig:310D12IW}

\end{minipage}
\end{figure}

Let us restrict our attention to the top $A_1$ string, $\Cr _{(3,0)}\subseteq \Cr_{(3,1,0)}^{(\sigma _1\sigma _2)}.$ 
For the vertices $v\in \Cr _{(3,0)}\subseteq \Cr_{(3,1,0)}^{(\sigma _1\sigma _2)}$ we have 
$$\IP (v)=\left(\begin{array}{ccccc}
              3 & & 1 & & 0 \\
               & 3 &  & 0 & \\
               & & a_{22} & &  
             \end{array} \right),\hskip .5 cm \text{and}\hskip .5 cm \Gamma(\IP (v))=BZL(v)=\left[\begin{array}{cc}
                       3 & 0 \\
                        &  \Ga _{22} \\
                      \end{array} \right],$$
where $\Ga _{22}=a_{22}-a_{12}=a_{22}.$

The table below show the vertices of this string, and the corresponding Gelfand-Tsetlin coefficients. 
$$\begin{array}{c||c|c|c}
   v & \wght(v) & G^{(1,\la +\rho )}(v) & G^{(1,\mu )}(v) \\ \hline \hline 
   v_{\ast } & (0,3,1) & -t & 1\\ \hline 
   v_1 & (1,2,1) &  (-t)(1-t) & 1-t\\ \hline 
   v_2 & (2,1,1) & (-t)(1-t) & 1-t\\ \hline 
   v_3 & (3,0,1) & (-t)(-t) & -t
  \end{array}$$

Figure \ref{fig:310D12IW_onestring} shows the vertices labeled within $\Cr _{(3,0)}\subseteq \Cr_{(3,1,0)}^{(\sigma _1\sigma _2)}\subseteq \Cr_{(3,1,0)}.$ 
\begin{figure}[h!]
    \centering
    \includegraphics[width=0.4\textwidth]{310D12IW_onestring-ps.mps}
\caption{$\Cr _{(3,0)}$ within $\Cr_{(3,1,0)}^{(\sigma _1\sigma _2)}.$}
\label{fig:310D12IW_onestring}
\end{figure}

We see that if $\mu =(3,0)$ we have $d(\mu )=3$ and $d(\la +\rho )=4.$ Further, $G^{(n,\la +\rho )}(v)=G^{(n,\mu )}(v)\cdot G^{(n,\la +\rho )}(v_{\ast })$ holds for $v\in \Cr_{(3,0)}.$

\subsection{Notation for branching of the Demazure crystal}\label{subs:new_parametrization}

Proposition \ref{prop:Demazure_branching} and Proposition \ref{prop:GTsumbreakupDem} implies that when dealing with sums overthe Demazure crystal $\Cr_{\la +\rho }^{(w)},$ one can treat the components $\Cr _{\mu }$ as units. Computations for the proof of Theorem \ref{THM:MAIN} will often only involve $\la +\rho $ and $\mu ,$ the top two rows of the Gelfand-Tsetlin patterns parameterizing $\Cr _{\mu }\subseteq \Cr_{\la +\rho }^{(w)}.$ The following notation and terminology serves to facilitate these computations.

Recall that the first row of the $\Gamma $-array, $(\Gamma_{11},\ldots ,\Gamma _{1r})$ is the same for every element of a component $\Cr _{\mu }\subseteq \Cr_{\la +\rho }^{(w)}.$ With $\la $ fixed, we phrase our notation in terms of this $r$-tuple. Lemma \ref{lemma:reparametr_meaningul} justifies the choices made in the following definition. 

\begin{definition}\label{def:r-1_lowest_modifications}
 Let $\lambda =(\la _1,\ldots ,\la _r,\la _{r+1}),$ $\mu =(\mu _1,\ldots ,\mu _r),$ and $\Gamma =(\Gamma_{11},\ldots ,\Gamma _{1r}) \in \Z^{r}.$ (Set $\Gamma _{1,r+1}:=0.$)
We call $\Gamma $ $\lambda ${\em{-admissible}} if
\begin{equation}\label{eq:def_admissible}
 \Gamma _{1,j+1}\leq \Gamma _{1,j}\leq \Gamma _{1,j+1}+\la_{j}-\la_{j+1}+1 \text{ for every } 1\leq j\leq r.
\end{equation}
We call $\Gamma $ $(\la ,k)${\em{-admissible}} if $\Gamma $ is $\la $-admissible and $\Gamma _{i,j}=0$ for $k+1<j.$
We call $\Gamma $ {\em{non-strict}} if
\begin{equation}\label{eq:nonstrict_condition}
 \Gamma _{1,j-1}=\Gamma _{1,j}=\Gamma _{1,j+1}+\la _j-\la_{j+1}+1 \text{ for at least one } 1< j\leq r,
\end{equation}
and {\em{strict}} if it is not non-strict. 

We define a weight function and Gelfand-Tsetlin coefficients for a $\Gamma $ $r$-tuple that is $\la $-admissible:
\begin{equation}\label{eq:def_stringweight}
 \wght ^{(\la )}(\Gamma )=(\la_{r+1}+\Gamma _{1,r}, \la _r+\Gamma_{1,r-1}-\Gamma _{1,r},\ldots ,\la _2+\Gamma _{11}-\Gamma _{1,2},\la _{1}-\Ga_{11});
\end{equation}
\begin{equation}\label{eq:def_string_GTcoeff}
 \begin{split}
G^{(\lambda )}_1(\Gamma )=G^{(n,\lambda )}_1(\Gamma )&=\prod _{j=1}^r g_{1j}^{(\la )}(\Gamma )=\prod _{j=1}^r g_{1j}^{(n,\la )}(\Gamma ),\\  
g_{1j}^{(n,\la )}(\Gamma )&=\left\lbrace\begin{array}{ll}
                                        1 & \Gamma _{1,j+1}=\Gamma _{1,j}<\Gamma _{1,j+1}+\la_{j}-\la_{j+1}+1\\
                                        h^{\flat}(\Gamma _{1,j}) & \Gamma _{1,j+1}<\Gamma _{1,j}<\Gamma _{1,j+1}+\la_{j}-\la_{j+1}+1\\
                                        g^{\flat}(\Gamma _{1,j}) & \Gamma _{1,j+1}<\Gamma _{1,j}=\Gamma _{1,j+1}+\la_{j}-\la_{j+1}+1\\
                                        0 & {\mathrm{otherwise}}
                                       \end{array}
\right.
 \end{split}
\end{equation}
For convenience, we say $\Ga $ is associated to $\la $ and $\mu $ and write $\Ga =\Ga (\la,\mu )$ if 
\begin{equation}\label{eq:Gamma_mu_condition}
 \Gamma _{1,j}-\Gamma _{1,j+1}=\mu _{j}-(\la _{j+1}+r-j) .
\end{equation}
is satisfied. This is the case if $\Ga $ is the first row of an array $\Ga (\IP )$ of a pattern $\IP $ with top two rows $\la +\rho _r$ and $\mu .$
\end{definition}

Parts $(i)-(v)$ of the following lemma justify the choices in Definition \ref{def:r-1_lowest_modifications}. Part $(vi)$ will be convenient in later computations. 

\begin{lemma}\label{lemma:reparametr_meaningul}
Let $\la =(\la _1,\ldots ,\la _r,\la _r+1),$ $\mu =(\mu _1,\ldots ,\mu _r)$ and $\Ga =(\Gamma_{11},\ldots ,\Gamma _{1r})=\Ga (\la,\mu )$ associated to $\la $ and $\mu $ by \eqref{eq:Gamma_mu_condition}. Then the following statements hold. 
 \begin{enumerate}[(i)]
  \item The tuple $\Ga $ is $\la $-admissible if and only if the weights $\la +\rho _r$ and $\mu $ interleave.
  \item Let $w$ be as in \eqref{eq:w_form}. Then $\Ga $ is $(\la ,k)$-admissible if and only if $\Cr _{\mu }\subseteq \Cr _{\la +\rho }^{(w)}.$ 
  \item The tuple $\Ga $ is strict if and only if $\mu $ is strongly dominant. 
  \item Let $v_{\ast }$ be the lowest element of a component $\Cr _{\mu }\subseteq \Cr _{\la +\rho }^{(w)}$. Then
\begin{equation}\label{eq:repar_wght_good}
\wght (v_{\ast })-w_0^{(r)}(\rho _r)=\wght ^{(\la )}(\Ga ).
\end{equation}
  \item Let $v_{\ast }$ be the lowest element of a component $\Cr _{\mu }\subseteq \Cr _{\la +\rho }^{(w)}.$ Then
\begin{equation}\label{eq:repar_GT_good}
G^{(n,\la +\rho )}(v_{\ast })=\left\lbrace\begin{array}{ll}
G^{(n,\lambda )}_1(\Gamma ) & \text{ if $\mu $ is strongly dominant;}\\
0 & \text{ otherwise .} 
\end{array}\right.
\end{equation}
 \item With the notation as above,we have 
\begin{equation}\label{eq:repar_goalwght}
 \x ^{\wght ^{(\la )}(\Ga )} =\y^{w_0^{(r-1)}(\mu -\rho _{r-1})}\cdot x_{r+1}^{d(\la +\rho )-d(\mu )-r}.
\end{equation}
 \end{enumerate}
\end{lemma}
\begin{proof}
 Note first that the condition \eqref{eq:Gamma_mu_condition} is satisfied exactly if $\Ga =(\Gamma_{11},\ldots ,\Gamma _{1r})$ is the first row of the array $\Gamma (\IP ),$ When $\IP$ is a pattern with top rows $\la +\rho $ and $\mu .$ With this observation, the proof is straightforward from Propositions \ref{prop:banchcomp_explicit}, \ref{prop:Demazure_branching} and \ref{prop:GTsumbreakupDem}. For (i), we have that by \eqref{eq:Gamma_mu_condition},
$$\la _{j}+r-j+1 \geq \mu _j\geq \la _{j+1}+r-j \Longleftrightarrow \la_{j}-\la_{j+1}+1 \geq \Gamma _{1,j}-\Gamma _{1,j+1}\geq 0.$$
Again by \eqref{eq:Gamma_mu_condition} we have that $\Ga $ is $(\la ,k)$-admissible if and only if for any $k+1<j$ we have 
$$\Gamma _{1,j}=\mu _{j}-(\la _{j+1}+r-j)+\Gamma _{1,j+1}=\Gamma _{1,j+1}=0 \Longleftrightarrow \mu _j=\la _{j+1}+r-j.$$
By Proposition \ref{prop:Demazure_branching}, this is equivalent to $\Cr _{\mu }\subseteq \Cr _{\la +\rho }^{(w)}.$ This proves (ii). Part (iii) is true because 
$$\mu _{j-1}=\mu _{j} \Longleftrightarrow \mu _{j-1}=\la _j+r-j+1=\mu _j $$
and by \eqref{eq:Gamma_mu_condition}
$$\Gamma _{1,j-1}=\Gamma _{1,j} \Longleftrightarrow \Gamma _{1,j-1}-\Gamma _{1,j}=\mu _{j-1}-(\la _{j}+r-j+1)=0;$$
$$\Gamma _{1,j}=\Gamma _{1,j+1}+\la _j-\la_{j+1}+1\Longleftrightarrow \Gamma _{1,j}-\Gamma _{1,j+1}=\mu _{j}-(\la _{j+1}+r-j)=\la _j-\la_{j+1}+1 .$$

For part (iv), recall that by part (c) of Proposition \ref{prop:crystal_parametrization}, if $v_{lowest}$ is the lowest element of $\Cr _{\la+\rho },$ then $\wght(v)-\wght(v_{lowest})$ can be expressed from the entries $b_i(v)$ of $BZL(v).$ We have $\wght(v_{lowest})=w_0^{(r)}(\la +\rho ).$ Furthermore, by Remark \ref{rmk:bzltolowest}, for the lowest element $v_\ast $ of a component $\Cr _{\mu },$ $b_i(v_{\ast })=0$ if $i\leq \binom{r}{2}.$ Further, since $\Ga $ is the top row of $BZL(v_\ast )=\Ga (\IP (v_\ast )),$ we also have $b_i(v_{\ast })=\Ga _{1,i-\binom{r}{2}}.$ Now \eqref{eq:weight_of_v_from_BZL} implies that
$$\wght(v_\ast )-w_0^{(r)}(\rho ) =w_0^{(r)}(\la)+\sum _{j=1}^{r}\Ga _{1j}\cdot \alpha _{r+1-j}=\wght ^{(\la )}(\Ga ).$$

To prove (v), recall that by Remark \ref{rmk:nonstrict_mu} if $\mu $ is not strongly dominant (i.e. $\Ga $ is nonstrict), then $G^{(n,\la +\rho )}(v_{\ast })=0.$ Furthermore, if $\mu $ is strongly dominant, then by \eqref{eq:GTcoeffofvast}, the Gelfand-Tsetlin coefficient corresponding to $v_\ast $ only depends on the first row of the $BZL$-array: 
$$G^{(n,\la +\rho )}(\IP (v_{\ast }))=\prod _{1\leq j\leq r} g_{1j}^{n,\la +\rho }(\IP (v_{\ast })).$$
Now since the first two rows of $\IP (v_{\ast })$ are $\la+\rho $ and $\mu ,$ and $\Ga $ is the first row of $\Ga (\IP (v_\ast )),$ the statement follows immediately from comparing \eqref{eq:def_string_GTcoeff} and \eqref{eq:GTfactorforGammadef_n}. 

Finally, recall that we write $\x=(x_1,\ldots ,x_r,x_{r+1}),$ $\y=(x_1,\ldots ,x_r)$ and $d(\la)$ (or $d(\mu )$) for the sum of the components of the weight $\la $ (respectively, $\mu $).
Note that $\wght _{\mu }(v_{\ast })=w_0^{(r-1)}(\mu )$ and $d(\la +\rho )-d(\mu )=\la_1+r-\Ga _{11}.$ Then from part (iv) above, and part (a) of Proposition \ref{prop:crystal_branching} we have 
$$
\begin{array}{rl}
   \x ^{\wght ^{(\la )}(\Ga )} = &  \x ^{\wght (v_{\ast })}\cdot \x ^{-w_0^{(r)}(\rho _r)}\\
= & \y^{\wght _{\mu }(v_{\ast })}\cdot x_{r+1}^{d(\la +\rho )-d(\mu )}\cdot \y^{w_0^{(r-1)}(\rho _{r-1})}\cdot x_{r+1}^{-r} \\
= & \y^{w_0^{(r-1)}(\mu -\rho _{r-1})}\cdot x_{r+1}^{d(\la +\rho )-d(\mu )-r}.
  \end{array}
$$

\end{proof}

%% file: main_ind.tex
\section{Reduction and proof of the Main Theorem}\label{chap:equivalent_reformulations}

We are ready to begin the proof of Theorem \ref{THM:MAIN}. The expression on the left-hand side involves a large sum of Demazure-Lusztig operators,
\begin{equation}\label{eq:operator_r}
 \left(\sum _{u\leq w_0^{(r)}} \T _u \right)=\Delta _t^{(r)}\cdot \Dem _{w_0^{(r)}}.
\end{equation}
The idea behind the proof is that one may replace this expression by progressively simpler ones, eventually reducing the statement of Theorem \ref{THM:MAIN} to the description of the polynomial 
\begin{equation}\label{eq:simpleDL_LHS}
(\T_r \T_{r-1}\cdots \T_1)(\x^{w_0(\la )}).
\end{equation}
The statement describing the polynomial \eqref{eq:simpleDL_LHS} is then proved by induction in Section \ref{chap:technical_induction}.

In the proof we restrict our attention to the case where the Weyl group element $w$ (a beginning section of $w_0^{(r)}$) has length at least $\binom{r}{2}+1.$ This leads to no loss of generality by the following remark. 

\begin{remark}\label{rmk:restriction_on_w}
The statement of Theorem \ref{THM:MAIN} in type $A_r$ but for $\ell(w)\leq \binom{r}{2}$ is equivalent to an instance of the theorem in type $A_{r-1}.$ Let $\la $ be as in the statement of the theorem, $\la ' =(\la _2,\la _3,\ldots ,\la _{r+1})$ and suppose $\ell(w)\leq \binom{r}{2}.$ Then in fact $w$ is a beginning section of $w_0^{(r-1)}.$ The statement \eqref{eq:thm_main_statement} for type $A_r,$ $\la $ and $w$ is the analogous statement for type $A_{r-1},$ $\la '$ and $w,$ except both sides are multiplied by $x_{r+1}^{\la _1}.$
On the left-hand side, $\T _1,\ldots ,\T _{r-1}$ all commute with multiplication by $x_{r+1}.$ As for the right-hand side, in the decomposition \eqref{eq:crystal_branching}, $\Cr _{\la +\rho }^{(w)}$  is contained in the component $\Cr _{\la '+\rho _{r-1}}$ of the lowest element, $v_{lowest}\in \Cr _{\la +\rho }.$ We have $G^{\la +\rho }(v_{lowest })=1.$ The statement now follows from Proposition \ref{prop:crystal_branching}.
\end{remark}

Hence from now on, we shall assume that $w$ is as in \eqref{eq:w_form}. Recall that $w$ is fixed by a choice of the pair $r,k$ where $0\leq k<r.$ Call the statement of Theorem \ref{THM:MAIN} for such a fixed $w$ and fixed $n$ (but for any dominant, effective weight $\la $) $IW_{r,k}^{(n)}.$ Proving $IW_{r,k}^{(n)}$ for any pair $0\leq k<r$ proves Theorem \ref{THM:MAIN}. 

Theorem \ref{thm:metaplectic_Tokuyama} is the special case of Theorem \ref{THM:MAIN} where $w=w_0^{(r)},$ i.e. $k=r-1.$ We will sometimes use the notation $Tok_{r}^{(n)}=IW_{r,r-1}^{(n)}.$

\begin{remark}\label{rmk:suppress_n}
 Much, but not all of the notation introduced above, in previous chapters, and in what follows, depends on the value of $n.$ In particular the meaning of $\Dem _i,$ $\T_i,$ (the action of) $\sigma _i,$ $G ^{(\la ,w)}(v),$ $Tok_{r}^{(n)}$ and $IW_{r,k}^{(n)}$ depend on $n,$ but $w_0,$ $W,$ $\Cr_{\la +\rho}^{(w)}$ and $\wght(v)$ do not. We will usually suppress $n$ from the notation. When reading the statements and proofs below, one should keep in mind that the meaning varies with $n.$ The entire argument of the proof is about a(n arbitrarily) fixed $n.$
\end{remark}

The reduction of $IW_{r,k}$ to the simpler statement is itself an induction by $r.$ We will phrase two more statements, $M_{r,k}$ (Proposition \ref{prop:Mrk}) and $N_{r,k}$ (Proposition \ref{prop:Nrk}). These involve smaller expressions of Demazure-Lusztig operators on the left hand side, and make use of the notation of Definition \ref{def:r-1_lowest_modifications}. $N_{r,r-1}$ describes the polynomial in \eqref{eq:simpleDL_LHS}.

The technical ingredients of the reduction are stated as lemmas or auxiliary propositions along the way. The proof, using these auxiliary statements, is in section \ref{subs:proof_summary}. The last ingredient is the proof of Proposition \ref{prop:N_r_r-1}, which is a rather technical induction, and forms the contents of Section \ref{chap:technical_induction}. 

\subsection{Auxiliary statements}

First we rewrite both sides of $IW_{r,k}$ in terms of the operator appearing in $Tok_{r-1},$ of the form \eqref{eq:operator_r} For the left-hand side, this is accomplished by the following lemma. It is really just a statement about the Bruhat-order; we omit the proof. 

\begin{lemma}\label{lem:Bruhat_shorterwords}
 Let $w=w_0^{(r-1)}\cdot \sigma _r\cdots \sigma _{r-k}.$ Then 
 \begin{equation}\label{eq:Bruhat_shorterwords}
\sum _{u\leq w} \T _u =\left(\sum _{u\leq w_0^{(r-1)}} \T _u \right)\cdot (1+\T_r+\T_r\T_{r-1}+\cdots +\T_r\cdots \T_{r-k}).
 \end{equation}
\end{lemma}

The first factor on the right-hand side of \eqref{eq:Bruhat_shorterwords} is 
\begin{equation}\label{eq:operator_r-1}
\sum _{u\leq w_0^{(r-1)}} \T _u ,
\end{equation}
the operator on the left-hand side of $Tok_{r-1}.$ By Theorem \ref{THM:LONG_WORD}, it is equal to $\Delta _t^{(r-1)}\cdot \Dem _{w_0^{(r-1)}}.$ The second factor will appear as the operator in $M_{r,k}$ (Proposition \ref{prop:Mrk}).

The following Proposition reproduces the right-hand side of $IW_{r,k}$ in terms of the operator \eqref{eq:operator_r-1}. It is a consequence of Proposition \ref{prop:GTsumbreakupDem} and Lemma \ref{lemma:reparametr_meaningul}. It is proved in section \ref{subs:proof_INDUCTION_ON_GTSIDE}.

\begin{prop}\label{PROP:INDUCTION_ON_GTSIDE}
 Assume $IW_{r-1,r-2}(=Tok_{r-1})$ holds. Then we have 
\begin{equation}\label{eq:A_r-1_induction_on_GTside}
\x^{-w_0(\rho )}\cdot \sum _{v\in \Cr_{\la +\rho }^{(w)}} G^{(n,\la+\rho )}(v)\cdot \x^{\wght(v)}=\left(\sum _{u\leq w_0^{(r-1)}} \T _u \right) \sum _{\substack{\Gamma =(\Gamma _{11},\ldots ,\Gamma _{1r})\\ \Gamma \ (\la ,k)-{\mathrm{admissible}}}} G^{\la }_1(\Gamma )\cdot \x^{\wght ^{(\la )}(\Gamma )}
\end{equation}
\end{prop}

Lemma \ref{lem:Bruhat_shorterwords} and Proposition \ref{PROP:INDUCTION_ON_GTSIDE} together produce both sides of $IW_{r,k}$ as the operator in \eqref{eq:operator_r-1} applied to a polynomial. The fact that the ``inputs'' are the same up to annihilation by this operator is the statement that we will call $M_{r,k}.$ The next proposition phrases the statement $M_{r,k}$ explicitly for any $0\leq k<r.$

\begin{prop}\label{prop:Mrk}
 Let $\lambda =(\la _1,\ldots ,\la_r ,\la _{r+1})$ be any dominant weight, $0\leq k<r$ integers. Then we have 
\begin{equation}\label{eq:Mrk}
 (1+\T_r+\T_r\T_{r-1}+\cdots +\T_r\cdots \T_{r-k})\x ^{w_0(\la )}\equiv \sum _{\substack{\Gamma =(\Gamma _{11},\ldots ,\Gamma _{1r})\\ \Gamma \ (\la ,k)-{\mathrm{admissible}}}} G^{\la }_1(\Gamma )\cdot \x^{\wght ^{(\la )}(\Gamma )}
\end{equation}
Here $\equiv $ means that the difference of the left  and right hand side is annihilated by $\Dem _{w_0^{(r-1)}}.$ Call this statement (that \eqref{eq:Mrk} holds for any $\la $ dominant weight) $M_{r,k}.$ 
\end{prop}

The statement $M_{r,k}$ lends itself to an obvious simplification. On the left hand side, there is a sum of $k+1$ strings of Demazure-Lusztig operators. The statement $N_{r,k}$ involves only one of them. 

\begin{prop}\label{prop:Nrk}
 Let $\lambda =(\la _1,\ldots ,\la_r ,\la _{r+1})$ be any dominant weight, $0\leq k<r$ integers. Then we have 
\begin{equation}\label{eq:Nrk}
 (\T_r\cdots \T_{r-k})\x ^{w_0(\la )}\equiv \sum _{\substack{\Gamma =(\Gamma _{11},\ldots ,\Gamma _{1r})\\ \Gamma \ (\la ,k)-{\mathrm{admissible}}\\ \Ga _{1,k+1}\neq 0}} G^{\la }_1(\Gamma )\cdot \x^{\wght ^{(\la )}(\Gamma )}
\end{equation}
Here $\equiv $ means that the difference of the left  and right hand side is annihilated by $\Dem _{w_0^{(r-1)}}.$ Call this statement (that \eqref{eq:Nrk} holds for any $\la $ dominant weight) $N_{r,k}.$
\end{prop}

\begin{remark}\label{rmk:effective_unnecessary}
 Note that in both $M_{r,k}$ and $N_{r,k},$ $\la $ is not required to be effective, i.e. it may have negative components. We may however always assume that it is effective, replacing $\la $ by $\kappa =(\la _1+K,\ldots ,\la_r+K ,\la _{r+1}+K).$ This can be done because as an operator, multiplication by $(x_1\cdot x_2\cdots x_{r+1})^K$ commutes with $\T _i$ and $\Dem _i$ for any $1\leq i\leq r,$ and 
\begin{equation*}
  \x^{w_0(\kappa )}=  \x^{w_0(\la )}\cdot (x_1\cdot x_2\cdots x_{r+1})^K, \hskip .5 cm \x^{\wght^{\kappa }(\Gamma )}=  \x^{\wght^{\lambda }(\Gamma )}\cdot (x_1\cdot x_2\cdots x_{r+1})^K .
\end{equation*}
\end{remark}

The following lemma is straightforward.

\begin{lemma}\label{LEMMA:MRK_AND_NRK_EQUIVALENT}
 Proposition \ref{prop:Nrk} implies Proposition \ref{prop:Mrk}. That is, we have 
\begin{equation}\label{eq:MrkequivvNrk}
 \forall r,k \hskip .3 cm N_{r,k}\hskip .5 cm \Longrightarrow \hskip .5 cm \forall r,k \hskip .3 cm M_{r,k}.
\end{equation}
\end{lemma}

As a last step in the sequence of replacing Theorem \ref{THM:MAIN} with simpler statements, we note that in the statement $N_{r,k},$ the parameter $k$ is the interesting one. This is the content of Lemma \ref{LEM:INDEX-SHIFT} below. The proof is straightforward by renaming variables, and keeping in mind that multiplication by $x_i$ commutes with $\T _j$ and $\Dem _j$ if $i\notin \{j,j+1\}.$

\begin{lemma}\label{LEM:INDEX-SHIFT}
 If $N_{k+1,k}$ is true, then $N_{r,k}$ is true for every $r>k.$ In fact, $N_{k+1,k}$ implies a slightly stronger statement than $N_{r,k}:$ the difference of the left-hand side and the right-hand side is annihilated not just by $\Dem _{w_0^{(r-1)}},$ but by the Demazure-operator corresponding to the long word in the group $\langle \sigma _{r-k},\sigma _{r-k+1},\ldots ,\sigma _{r-1}\rangle .$
\end{lemma}

The statement $N_{k+1,k}$ will be proved in Section \ref{chap:technical_induction} as Proposition \ref{prop:N_r_r-1}. We are now ready to give the proof of Theorem \ref{THM:MAIN}.

\subsection{The proof of Theorem \ref{THM:MAIN}}\label{subs:proof_summary}

By Proposition \ref{prop:N_r_r-1} (proved in Section \ref{chap:technical_induction}), we have that $N_{k+1,k}$ holds for any nonnegative $k.$ By Lemma \ref{LEM:INDEX-SHIFT}, this implies that $N_{r,k}$ holds for any pair of integers $0\leq k<r,$ i.e. Proposition \ref{prop:Nrk} is true. By Lemma \ref{LEMMA:MRK_AND_NRK_EQUIVALENT}, this proves Proposition \ref{prop:Mrk}, i. e. $M_{r,k}$ for any pair of integers $0\leq k<r.$ 

We prove $IW_{r,k}$ for any pair of integers $0\leq k<r$ by induction on $r.$ 

To start, notice that both $M_{1,0}$ and $IW_{1,0}$ state that if $\la _{1}\geq \la _{2},$ then 
\begin{equation}\label{eq:IW10orM10}
 \begin{split}
  (1+T_1)x_1^{\la _2}x_2^{\la _1}= & \sum _{\Gamma _{11}=0}^{\la_1-\la _2+1}G^{(\la )}_1(\Gamma _{11})\cdot x_1^{\la _2+\Gamma _{11}}x_2^{\la_2-\Gamma _{11}}\\
& =\frac{1}{x_2}\sum _{\IP } G(\IP )\cdot \x^{wt(\IP )},
 \end{split}
\end{equation}
where the sum is over all Gelfand-Tsetlin patterns $\IP $ of the form
$$\IP =\left(\begin{array}{ccc}
                                                              \la _1+1& & \la _2\\
                                                               & \la _2+\Gamma _{11}& 
                                                            \end{array}
\right).$$
Thus $IW_{1,0}$ is the same as $M_{1,0},$ and in particular, $IW_{r,k}$ is true if $r=1.$

Now let $r>1,$ $0\leq k<r$ and assume that $IW_{r-1,r-2}=Tok_{r-1}$ is true. We know $M_{r,k}$ holds, hence 
\begin{equation}\label{eq:Mrk_ininduction}
 (1+\T_r+\T_r\T_{r-1}+\cdots +\T_r\cdots \T_{r-k})\x ^{w_0(\la )}\equiv \sum _{\substack{\Gamma =(\Gamma _{11},\ldots ,\Gamma _{1r})\\ \Gamma \ (\la ,k)-{\mathrm{admissible}}}} G^{\la }_1(\Gamma )\cdot \x^{\wght ^{(\la )}(\Gamma )},
\end{equation}
i.e. the difference of the two sides of \eqref{eq:Mrk_ininduction} is annihilated by $\Dem _{w_0^{(r-1)}}.$ By Theorem \ref{THM:LONG_WORD}, the difference is then also annihilated by 
$$\Delta _t^{(r-1)}\cdot \Dem _{w_0^{(r-1)}} =\sum _{u\leq w_0^{(r-1)}}\T _u.$$
That is, we have 
\begin{equation}\label{eq:r-1op_applied}
\begin{split}
  \left(\sum _{u\leq w_0^{(r-1)}}\T _u\right)(1+\T_r+\T_r\T_{r-1}+\cdots & +\T_r\cdots \T_{r-k})\x ^{w_0(\la )} \\
& =\left(\sum _{u\leq w_0^{(r-1)}}\T _u\right)\sum _{\substack{\Gamma =(\Gamma _{11},\ldots ,\Gamma _{1r})\\ \Gamma \ (\la ,k)-{\mathrm{admissible}}}} G^{\la }_1(\Gamma )\cdot \x^{\wght ^{(\la )}(\Gamma )}.
\end{split}
\end{equation}
Rewriting the left hand side of \eqref{eq:r-1op_applied} by Lemma \ref{lem:Bruhat_shorterwords}, and the right hand side by Proposition \ref{PROP:INDUCTION_ON_GTSIDE} we arrive at 
\begin{equation}\label{eq:r-1op_applied_rewritten}
  \left(\sum _{u\leq w_0^{(r)}}\T _u\right)\x ^{w_0(\la )}  =\x^{-w_0(\rho )}\cdot \sum _{v\in \Cr_{\la +\rho }^{(w)}} G^{(n,\la+\rho )}(v)\cdot \x^{\wght(v)}.
\end{equation}
This is exactly the statement $IW_{r,k}.$

Thus $IW_{r,k}$ is true for any pair of integers $0\leq k<r.$ By Remark \ref{rmk:restriction_on_w}, this completes the proof of Theorem \ref{THM:MAIN}. \qed

\subsection{The proof of Proposition \ref{PROP:INDUCTION_ON_GTSIDE}}\label{subs:proof_INDUCTION_ON_GTSIDE}

We prove that if $Tok_{r-1}$ (equivalently, $IW_{r-1,r-2}$) holds, then 
\begin{equation}\label{eq:goal_induction:_on_gtside}
 \x^{-w_0(\rho )}\cdot \sum _{v\in \Cr_{\la +\rho }^{(w)}} G^{(n,\la+\rho )}(v)\cdot \x^{\wght(v)}=\left(\sum _{u\leq w_0^{(r-1)}} \T _u \right) \sum _{\substack{\Gamma =(\Gamma _{11},\ldots ,\Gamma _{1r})\\ \Gamma \ (\la ,k)-{\mathrm{admissible}}}} G^{\la }_1(\Gamma )\cdot \x^{\wght ^{(\la )}(\Gamma )}.
\end{equation}

By Proposition \ref{prop:GTsumbreakupDem}, we have 
\begin{equation}\label{eq:IW_rewrite_1}
\begin{split}
  \x^{-w_0(\rho )}\cdot \sum _{v\in \Cr_{\la +\rho }^{(w)}} G^{(n,\la+\rho )}(v)\cdot \x^{\wght(v)}= & \sum _{\mu } G^{(n,\la +\rho )}(v_{\ast })\cdot x_{r+1}^{d(\la +\rho )-d(\mu )-r}\cdot \\
 & \cdot \left(\y^{-w_0^{(r-1)}(\rho _{r-1})}\cdot \sum _{v\in \Cr_{\mu }} G^{(n,\mu )}(v)\cdot \y^{\wght _{\mu }(v)}\right).
\end{split}
\end{equation}
Here the sum is over all $\mu =(\mu _{1},\mu _2,\ldots ,\mu _r)$ that interleave with $\la +\rho $ and $\mu _j=\la _{j+1}+r-j$ for $j>k+1.$ We claim that 
\begin{equation}\label{eq:tok_r-1_loose}
 \y^{-w_0^{(r-1)}(\rho _{r-1})}\cdot \sum _{v\in \Cr_{\mu }} G^{(n,\mu )}(v)\cdot \y^{\wght _{\mu }(v)}=\left(\sum _{u\leq w_0^{(r-1)}} \T _u \right) \y^{w_0^{(r-1)}(\mu -\rho _{r-1})}.
\end{equation}

Since $\mu $ interleaves with $\la +\rho ,$ it is dominant and effective. We distinguish between two cases according to whether $\mu $ is strongly dominant or not. 

If $\mu$ is strongly dominant, then $\mu -\rho _{r-1}$ is dominant and effective. In this case \eqref{eq:tok_r-1_loose} is the statement $Tok_{r-1}$ ($IW_{r-1,r-2}$) for the weight $\mu -\rho _{r-1},$ hence it is true by the assumption that $Tok _{r-1}$ holds.

Suppose now that $\mu $ is not strongly dominant, i.e. we have $\mu _j=\mu _{j+1}$ for some $1\leq j\leq r-1.$ We show that then both sides of \eqref{eq:tok_r-1_loose} are zero. The left hand side is zero by Remark \ref{rmk:nonstrict_mu}. We show that the operator on the right hand side of \eqref{eq:tok_r-1_loose} annihilates the monomial $\y^{w_0^{(r-1)}(\mu -\rho _{r-1})}.$ By Theorem \ref{THM:LONG_WORD}, this operator is $\Delta _{t}^{(r-1)}\cdot \Dem _{w_0^{(r-1)}}.$ Since $w_0^{(r-1)}$ is the long element in the Weyl group generated by $\sigma _1,\ldots ,\sigma _{r-1},$ by Lemma \ref{lemma:Demazure_annihilation} it suffices to prove $\Dem _i\y^{w_0^{(r-1)}(\mu -\rho _{r-1})}=0$ for at least one index $1\leq i=r-j\leq r-1.$ Let $i=r+1-(j+1)=r-j,$ $i+1=r-j+1.$ Then in the monomial $\y^{w_0^{(r-1)}(\mu -\rho _{r-1})},$ $x_i$ appears with exponent $\mu _{j+1}-r+j+1=\mu _j-r+j+1,$ while $x_{i+1}$ appears with exponent $\mu _{j}-r+j.$ Thus, by Corollary \ref{cor:annihilation}, indeed $\Dem _i\y^{w_0^{(r-1)}(\mu -\rho _{r-1})}=0.$ 
Thus the right hand side of \eqref{eq:tok_r-1_loose} is indeed zero if $\mu $ is not strongly dominant.   

Having established \eqref{eq:tok_r-1_loose}, we have that 
\begin{equation}\label{eq:IW_rewrite_2}
\begin{split}
  \x^{-w_0(\rho )}\cdot \sum _{v\in \Cr_{\la +\rho }^{(w)}} G^{(n,\la+\rho )}(v)\cdot \x^{\wght(v)}= & \sum _{\mu } G^{(n,\la +\rho )}(v_{\ast })\cdot x_{r+1}^{d(\la +\rho )-d(\mu )-r}\cdot \\
 & \cdot \left(\sum _{u\leq w_0^{(r-1)}} \T _u \right) \y^{w_0^{(r-1)}(\mu -\rho _{r-1})}\\
 = & \left(\sum _{u\leq w_0^{(r-1)}} \T _u \right) \\
  & \sum _{\mu } G^{(n,\la +\rho )}(v_{\ast })\cdot x_{r+1}^{d(\la +\rho )-d(\mu )-r}\cdot \y^{w_0^{(r-1)}(\mu -\rho _{r-1})}.
\end{split}
\end{equation}
Here the second equation is true because multiplication by $x_{r+1}$ commutes with $\T _u$ for every $u\leq w_0^{(r-1)}.$ By Lemma \ref{lemma:reparametr_meaningul}, we have that if $\Ga =\Ga (\la ,\mu)=(\Ga _1,\ldots ,\Ga _r)$ as in Definition \ref{def:r-1_lowest_modifications}, then $\Ga $ is $(\la ,k)$ admissible exactly if $\mu$ appears in the summation in \eqref{eq:IW_rewrite_1}, furthermore
$$G^{(n,\la +\rho )}(v_{\ast })=\left\lbrace\begin{array}{ll}
G^{(n,\lambda )}_1(\Gamma ) & \text{ if $\mu $ is strongly dominant;}\\
0 & \text{ otherwise ;} 
\end{array}\right.$$
and 
$$\y^{w_0^{(r-1)}(\mu -\rho _{r-1})}\cdot x_{r+1}^{d(\la +\rho )-d(\mu )-r}=\x ^{\wght ^{(\la )}(\Ga )} .$$
Thus we may rewrite \eqref{eq:IW_rewrite_2} further as 
\begin{equation}\label{eq:IW_rewrite_4}
 \x^{-w_0(\rho )}\cdot \sum _{v\in \Cr_{\la +\rho }^{(w)}} G^{(n,\la+\rho )}(v)\cdot \x^{\wght(v)}=  \left(\sum _{u\leq w_0^{(r-1)}} \T _u \right) \sum _{\substack{\Gamma =(\Gamma _{11},\ldots ,\Gamma _{1r})\\ \Gamma \ (\la ,k)-{\mathrm{admissible}}}} G^{(n,\la )}_1(\Ga )\cdot \x ^{\wght ^{(\la )}(\Ga )}.
\end{equation}
This is exactly \eqref{eq:goal_induction:_on_gtside}; the proof of Proposition \ref{PROP:INDUCTION_ON_GTSIDE} is complete. 

%% file: tech_ind.tex
\section{Proof of the statement $N_{r,r-1}$}\label{chap:technical_induction}

In Section \ref{chap:equivalent_reformulations}, the proof of Theorem \ref{THM:MAIN} and Theorem \ref{thm:metaplectic_Tokuyama} was reduced to describing the action of the string of Demazure-Lusztig operators $\T_r\ldots \T_1$ on a monomial, i.e. the statement $N_{r,r-1}.$ This section consists of the proof of the statement $N_{r,r-1}.$ We recall the statement in Proposition \ref{prop:N_r_r-1} below.

\begin{prop}\label{prop:N_r_r-1}
Let $\lambda =(\la _1,\ldots ,\la_r ,\la _{r+1})$ be any dominant weight. Then we have 
\begin{equation}\label{eq:Nrr-1}
 (\T_r\cdots \T_1)\x ^{w_0(\la )}\equiv \sum _{\substack{\Gamma =(\Gamma _{11},\ldots ,\Gamma _{1r})\\ \Gamma \ \la -{\mathrm{admissible}}\\ \Ga _{1,r}\neq 0}} G^{\la }_1(\Gamma )\cdot \x^{\wght ^{(\la )}(\Gamma )}.
\end{equation}
Here $\equiv $ means that the difference of the left  and right hand side is annihilated by $\Dem _{w_0^{(r-1)}}.$
\end{prop}

Recall that the relevant notation has been introduced in section \ref{subs:new_parametrization}. In this section, we use $v$ for denoting $v=t^n=q^{-1}.$ 

Let us abbreviate both sides of the equation \eqref{eq:Nrr-1}: 
\begin{equation}\label{eq:LHS_and_RHS_shorthand}
 \Le _r^{(\la )}(\x):=(\T_r\cdots \T_1)\x ^{w_0(\la )}\hskip .5 cm \text{and}\hskip .5 cm \Ri _r^{\la }(\x):=\sum _{\substack{\Gamma =(\Gamma _{11},\ldots ,\Gamma _{1r})\\ \Gamma \ \la -{\mathrm{admissible}}\\ \Ga _{1,r}\neq 0}} G^{\la }_1(\Gamma )\cdot \x^{\wght ^{(\la )}(\Gamma )}.
\end{equation}

The proof is by induction on $r.$ The base case is fairly straightforward using the definitions and Claim \ref{claim:GaussSumRelationship}. We omit this rank one computation and contend that both sides turn out to be equal to the following expression:
\begin{equation*}\label{eq:N_{1,0}_result}
\Le _1^{(\la )}(\x)=\Ri _1^{\la }(\x)  =\frac{(1-v)\cdot \x^{n\alpha }}{1-\x^{n\alpha }}\left(x_1^{\la _2}x_2^{\la _1}-\x^{-r_{n}\left(\la _1-\la _2\right)\alpha }\cdot x_1^{\la _1}x_2^{\la _2} \right)+v\cdot g_{1+\la _1-\la _2}\cdot x_1^{\la _1+1}x_2^{\la _2-1}.
\end{equation*} 

For the induction step, we assume that $N_{k+1,k}$ holds for $k<r-1.$ The goal is to prove 
\begin{equation}\label{eq:induction_step_goal}
 \Le _r^{(\la )}(\x )\equiv \Ri _r^{(\la )}(x),\text{ i.e. }\Dem _{w_0^{(r-1)}}\left(\Le _{r}^{(\la )}(\x )-\Ri _{r}^{(\la )}(\x )\right)=0.
\end{equation}

\begin{claim}\label{claim:rhs_step_enough}
 It suffices to show that if $N_{k+1,k}$ holds for $k<r-1,$ then
\begin{equation}\label{eq:rhs_step_enough}
 \Ri _r^{(\la )}(x)\equiv \T _r \left(x_{r+1}^{\la _1}\cdot \Ri _{r-1}^{(\mu )}(\y )\right),\text{ i.e. }\Dem _{w_0^{(r-1)}}\left(\T _r \left(x_{r+1}^{\la _1}\cdot \Ri _{r-1}^{(\mu )}(\y )\right)-\Ri _{r}^{(\la )}(\x )\right)=0.
\end{equation}
\end{claim}
\begin{proof}
The proof is straightforward using the fact that multiplication by $x_{r+1}^{\la _1}$ commutes with the operators $T_1,\ldots ,T_{r-1},$ multiplication by $x_{r+1}^{\la _1}$ and $\T _r$ both commute with $\Dem _{w_0^{(r-2)}},$ and Lemma \ref{lemma:Demazure_annihilation}.
\end{proof}

The remainder of this section will consist of proving \eqref{eq:rhs_step_enough} from the assumption that $N_{k+1,k}$ holds for $k<r-1.$ The argument will proceed as follows. After introducing some convenient notation (in \ref{sect:conventions}), we shall simplify the induction step (in \ref{sect:simplifying_step}). Computing $\T _r \left(x_{r+1}^{\la _1}\cdot \Ri _{r-1}^{(\mu )}(\y )\right)$ directly, and comparing the result with $\Ri _{r}^{(\la )}(\x ),$ we find that there is a polynomial ``left over''. The fact that this polynomial is annihilated by $\Dem _{w_0^{(r-1)}}$ is called ${\mathbf{F_{r}}}$ (Proposition \ref{PROP:F_ANNIH}). In fact, the computation shows that assuming $N_{r-1,r-2},$ the statements  ${\mathbf{F_{r}}}$ and $N_{r,r-1}$ are equivalent (Lemma \ref{lem:induction_step_same_as_F_annih}). Thus it remains to prove Proposition \ref{PROP:F_ANNIH} by showing the statement ${\mathbf{F_{r}}};$ this is done in section \ref{sect:proof:PROP:F_ANNIH}. This will also (partially) be a proof by induction: by Lemma \ref{lem:induction_step_same_as_F_annih}, the assumption of $N_{k+1,k}$ for $k<r-1$ implies in particular that ${\mathbf{F_{j}}}$ holds for $j<r.$

\subsection{Notation and conventions}\label{sect:conventions}

The following conventions allow us to relate the statements $N_{r,r-1},$ $N_{r-1,r-2}$ and $N_{r-2,r-3}$ with more transparency.
Let weights be denoted by $\la  = (\la _1,\la _2,\la _3,\ldots ,\la _{r+1}),$ $\mu  = (\la _2,\ldots ,\la _{r+1}),$ and $\nu  = (\la _3,\ldots ,\la _{r+1});$ the variables with $\x =(x_1,\ldots ,x_{r-1},x_r,x_{r+1}),$ $\y =(x_1,\ldots ,x_r)$ and $\z =(x_1,\ldots ,x_{r-1}).$ Furthermore, let $\Ga ' =(\Ga _{11},\Ga _{12},\Ga _{13},\ldots ,\Ga _{1r}),$ $\Ga =(\Ga _{12},\Ga _{13},\ldots ,\Ga _{1r})$ and $\Ga _0 =(\Ga _{13},\ldots ,\Ga _{1r}).$
With this notation, we have that 
\begin{equation}\label{eq:lamu_gammarelate} 
\Ga '\text{ is }\la -\text{admissible if and only if }  \left\lbrace \begin{array}{l}
                                                                     \Ga \text{ is }\mu -\text{admissible and } \\
                                                                     \Ga _{12}\leq \Ga _{11}\leq \Ga _{12}+\la _1-\la _2+1;
                                                                    \end{array}\right. 
\end{equation}
and 
\begin{equation}\label{eq:lamu_summandrelate} 
G^{(\la )}(\Ga ')  =  g^{(\la )}_{11}(\Ga ')\cdot G^{(\mu )}_1(\Ga ), \text{ and }
\x ^{\wght^{(\la )}(\Ga ')}  =  \y ^{\wght^{(\mu )}(\Ga )}\cdot x_r^{\Ga _{11}}\cdot x_{r+1}^{\la _1-\Ga _{11}}.
\end{equation}
Similarly, 
\begin{equation}\label{eq:munu_gammarelate} 
\Ga \text{ is }\mu -\text{admissible if and only if }  \left\lbrace \begin{array}{l}
                                                                     \Ga _0\text{ is }\nu -\text{admissible and } \\
                                                                     \Ga _{13}\leq \Ga _{12}\leq \Ga _{13}+\la _2-\la _3+1;
                                                                    \end{array}\right. 
\end{equation}
and 
\begin{equation}\label{eq:munu_summandrelate}
G^{(\mu )}(\Ga ) = g^{(\mu )}_{12}(\Ga )\cdot G^{(\nu )}_1(\Ga _0), \text{ and }
\y ^{\wght^{(\mu )}(\Ga )}=\z ^{\wght^{(\nu )}(\Ga _0)}\cdot x_{r-1}^{\Ga _{12}}\cdot x_{r}^{\la _2-\Ga _{12}}.
\end{equation}
Notice that in indexing the $\mu $-admissible vector $\Ga ,$ we write $g^{(\mu )}_{12}(\Ga )$ for the Gelfand-Tsetlin coefficient corresponding to the first entry, $\Ga _{12}.$ Equations \eqref{eq:lamu_summandrelate} and \eqref{eq:munu_summandrelate} are a direct consequence of the notation introduced. In particular, the relationship between the monomials is true even if $\Ga '$ (or $\Ga $) is not $\la $-admissible (respectively, $\mu $-admissible). 

We will make repeated use of the following function on pairs of (positive) integers: 
\begin{equation}\label{eq:deltadef}
\delta (A,B)=\left\lbrace\begin{array}{ll}
                                        h^{\flat}(A) &\ifi  A<B\\
                                        h^{\flat}(A)- 1 &\ifi  A=B\\
                                        0 &\ifi  A>B
                                       \end{array}
\right.
\end{equation}

We are now ready to tackle the induction step. 

\subsection{Simplifying the induction step}\label{sect:simplifying_step}

We set out to prove \eqref{eq:rhs_step_enough}. To this end, we first rewrite $\T _r \left(x_{r+1}^{\la _1}\cdot \Ri _{r-1}^{(\mu )}(\y )\right).$ By the conventions \eqref{eq:munu_summandrelate}, and the fact that $\T_r$ commutes with multiplication by $x_1,$ $\ldots ,$ $x_{r-1},$ we have 
\begin{equation}\label{eq:RHSstep_rewrite}
 \T _r \left(x_{r+1}^{\la _1}\cdot \Ri _{r-1}^{(\mu )}(\y )\right)  = \sum _{\substack{\Gamma =(\Gamma _{12},\ldots ,\Gamma _{1,r})\\ \Gamma \ \mu -{\mathrm{admissible}}\\ \Gamma _{1,r}\neq 0}}G^{(\mu )}_1(\Gamma )\cdot \z ^{\wght^{(\nu )}(\Ga _0)}\cdot x_{r-1}^{\Ga _{12}}\cdot \T _r (x_{r}^{\la _2-\Ga _{12}}x_{r+1}^{\la _1}).
\end{equation}
Since $N_{1,0}$ is the base case of the induction, $N_{r,0}$ is true by Lemma \ref{LEM:INDEX-SHIFT}. Thus we have 
\begin{equation}\label{eq:Nr0}
 \T _r (x_{r}^{\la _2-\Ga _{12}}x_{r+1}^{\la _1})=\sum _{\Ga _{11}=1}^{\la _1-(\la _2-\Ga _{12})+1}g^{(\la _1,\la _2-\Gamma_{12})}_{11}(\Gamma _{11})x_r^{\la _2+\Gamma_{11}-\Gamma _{12}}x_{r+1}^{\la _1-\Gamma _{11}}.
\end{equation}
It follows from \eqref{eq:def_string_GTcoeff} and \eqref{eq:deltadef} that 
\begin{equation}\label{eq:coeffdiff_1}
 g^{(\la _1,\la _2-\Gamma_{12})}_{11}(\Gamma _{11})-g^{(\la)}_{11}(\Gamma ')=\delta (\Gamma _{11},\Ga _{12}).
\end{equation}
Now we may substitute \eqref{eq:Nr0} and \eqref{eq:coeffdiff_1} into \eqref{eq:RHSstep_rewrite}, and use the conventions \eqref{eq:lamu_gammarelate}, \eqref{eq:lamu_summandrelate} to conclude that 
$$\T _r \left(x_{r+1}^{\la _1}\cdot \Ri _{r-1}^{(\mu )}(\y )\right) = \Ri _r^{(\la )}(x) + \sum _{1\leq \Ga _{11}} x_{r+1}^{\la _1-\Gamma _{11}}\cdot \sum _{\substack{\Gamma =(\Gamma _{12},\ldots ,\Gamma _{1,r})\\ \Gamma \ \mu -{\mathrm{admissible}}\\ \Gamma _{1,r}\neq 0}}\delta (\Gamma _{11},\Ga _{12})G^{(\mu )}_1(\Gamma )\y ^{\wght^{(\mu )}(\Ga )}x_r^{\Gamma_{11}}.$$
Thus proving \eqref{eq:rhs_step_enough} is equivalent to showing that 
\begin{equation}\label{eq:after_stepping}
 \sum _{1\leq \Ga _{11}} x_{r+1}^{\la _1-\Gamma _{11}}\cdot \sum _{\substack{\Gamma =(\Gamma _{12},\ldots ,\Gamma _{1,r})\\ \Gamma \ \mu -{\mathrm{admissible}}\\ \Gamma _{1,r}\neq 0}}\delta (\Gamma _{11},\Ga _{12})\cdot G^{(\mu )}_1(\Gamma )\cdot \y ^{\wght^{(\mu )}(\Ga )}\cdot x_r^{\Gamma_{11}}\equiv 0,
\end{equation}
i.e. it is annihilated by $\Dem _{w_0^{(r-1)}}.$ Since multiplication by $x_{r+1}$ commutes with $\Dem _{w_0^{(r-1)}},$ \eqref{eq:after_stepping} is equivalent to showing that the part corresponding to a fixed power of $x_{r+1}$ is annihilated by $\Dem _{w_0^{(r-1)}}.$
This motivates the following notation. Let $a$ be a positive integer, and $\mu ,$ $\y$ as in Section \ref{sect:conventions}.
 \begin{equation}\label{eq:Fdef}
 F_{\mu,a}(\y )=\sum _{\substack{\Gamma =(\Gamma _{12},\ldots ,\Gamma _{1r})\\ \Gamma \ \mu -{\mathrm{admissible}}\\ \Gamma _{1,r}\neq 0}} \delta (a,\Ga _{12})\cdot G^{(\mu )}_1(\Gamma )\cdot \y ^{\wght^{(\mu )}(\Ga )}\cdot x_r^{a}.
\end{equation}

\begin{prop}\label{PROP:F_ANNIH}
 Let $\mu  = (\la _2,\la _3,\ldots ,\la _{r+1})$ be any dominant weight. Then for any positive integer $a$ we have 
\begin{equation}\label{eq:F_annihilated_prop}
 F_{\mu,a}(\y )\equiv 0,\text{ i.e., }\Dem _{w_0^{(r-1)}} F_{\mu,a}(\y )=0.
\end{equation}
Call this statement (that \eqref{eq:F_annihilated_prop} holds for any dominant weight $\mu $ and positive integer $a$) ${\mathbf{F_{r}}}.$ 
\end{prop}

The argument in the present section amounts to the following lemma. 

\begin{lemma}\label{lem:induction_step_same_as_F_annih}
 If $N_{r-1,r-2}$ holds, then $N_{r,r-1}$ (for $\la ,$ $\x $ as above) is equivalent to the statement 
$$\forall a \hskip .5 cm F_{\mu,a}(\y )\equiv 0,\text{ or, equivalently, }\forall a\hskip .5 cm \Dem _{w_0^{(r-1)}} F_{\mu,a}(\y )=0.$$
\end{lemma}
\qed

Now to complete the induction step, it remains to prove Proposition \ref{PROP:F_ANNIH}, i.e. that $F_{\mu,a}(\y )$ is annihilated by $\Dem _{w_0^{(r-1)}}.$ This is the content of section \ref{sect:proof:PROP:F_ANNIH}. 

\subsection{Proof of Proposition \ref{PROP:F_ANNIH}}\label{sect:proof:PROP:F_ANNIH}
We distinguish between the cases where $a$ is divisible by $n$ or not. The case when it is {\em{not}} is significantly easier to handle. 

\subsubsection{The non-divisible case}\label{subs:when_a_notdiv_by_n}

The goal is to prove that if $n\nmid a,$ then $\Dem _{w_0^{(r-1)}} F_{\mu,\Ga _{11}}(\y )=0.$ Recall that by Claim \ref{claim:GaussSumRelationship}, $n\nmid a$ implies $h^{\flat }(a)=0.$
By \eqref{eq:deltadef}, this means that $\delta (a,\Ga _{12})=0$ unless $a=\Ga _{12},$ and $\delta (a,\Ga _{12})=-1.$
Thus in this case we have 
\begin{equation}\label{eq:F_when_notdivisible}
 F_{\mu,a}(\y )=-\sum _{\substack{\Gamma =(\Gamma _{12},\ldots ,\Gamma _{1r})\\ \Gamma \ \mu -{\mathrm{admissible}}\\ \Ga_{12}=a\\ \Gamma _{1,r}\neq 0}} G^{(\mu )}_1(\Gamma )\cdot \y ^{\wght^{(\mu )}(\Ga )}\cdot x_r^{a}.
\end{equation}
We will show that each term in the summation is either itself zero, or is annihilated by a Demazure-Lusztig operator corresponding to a simple reflection. Fix a term $\Gamma =(\Gamma _{12},\ldots ,\Gamma _{1r}),$ and take $\Ga _{1,r+1}:=0$ and $\Ga _{11}=a.$ Then $\Ga _{11}=\Ga _{12}$ and $\Ga _{1r}>\Ga _{1,r+1}.$ Let $j$ be the smallest index such that $\Ga _{1j}>\Ga _{1,j+1}$ ($2\leq j\leq r$). In this case $a=\Ga _{11}=\cdots =\Ga _{1,j-1}=\Ga _{1,j},$ so $n\nmid \Ga _{1,j}$ and hence $h^{\flat}(\Gamma _{1,j})=0.$ By \eqref{eq:def_string_GTcoeff}, we have $$g_{1j}^{(\mu )}(\Gamma )=\left\lbrace\begin{array}{ll}
                                        h^{\flat}(\Gamma _{1,j}) & \Gamma _{1,j+1}<\Gamma _{1,j}<\Gamma _{1,j+1}+\la_{j}-\la_{j+1}+1;\\
                                        g^{\flat}(\Gamma _{1,j})  & \Gamma _{1,j+1}<\Gamma _{1,j}=\Gamma _{1,j+1}+\la_{j}-\la_{j+1}+1.
                                       \end{array}
\right.$$
This means that $G^{(\mu )}_1(\Gamma )=0$ unless $\Ga _{1,j}=\Gamma _{1,j+1}+\la_{j}-\la_{j+1}+1.$ We show that in the latter case $\Dem _{r-j+1}$ annihilates the corresponding term. Observe that $\Ga_{1,j-1}=\Ga_{1,j}=\Gamma _{1,j+1}+\la_{j}-\la_{j+1}+1$ implies that $\y ^{\wght^{(\mu )}(\Ga )}\cdot x_r^{\Gamma_{11}}$ has a factor of $x_{r-j+1}^{\la_{j}+1}x_{r-j+2}^{\la_{j}}$ (and no other factors of $x_{r-j+1}$ or $r_{r-j+2}$). By Corollary \ref{cor:annihilation}, $\Dem _{r-j+1}$ indeed kills this term. Since $1\leq r-j+1\leq r-1,$ by Proposition \ref{lemma:Demazure_annihilation} part (ii) $\Dem _{w_0^{(r-1)}}$ annihilates all nonzero terms. This completes the proof of the non-divisible case.

\subsubsection{The divisible case}

From now on we assume that $a$ is divisible by ${\mathbf{n}}$. Since $\delta (a,\Ga _{12})$ will appear repeatedly in the computations below, we introduce the following shorthand. By Claim \ref{claim:GaussSumRelationship}, we have
\begin{equation}\label{eq:neqdeltadef}
 \delta _a(\Ga _{12})=\delta (a,\Ga _{12})=\left\lbrace\begin{array}{ll}
                                        1-v &\ifi  a<\Ga _{12};\\
                                        -v &\ifi  a=\Ga _{12};\\
                                        0 &\ifi  a>\Ga _{12}.
                                       \end{array}
\right.
\end{equation}

The goal is to prove ${\mathbf{F_{r}}}:$
\begin{equation}\label{eq:a_divn_goal}
 \Dem _{w_0^{(r-1)}}F_{\mu,a}(\y )=\Dem _{w_0^{(r-1)}}\sum _{\substack{\Gamma =(\Gamma _{12},\ldots ,\Gamma _{1r})\\ \Gamma \ \mu -{\mathrm{admissible}}\\ \Gamma _{1,r}\neq 0}}  \delta _a(\Ga _{12})\cdot G^{(\mu )}_1(\Gamma )\cdot \y ^{\wght^{(\mu )}(\Ga )}\cdot x_r^{a}=0.
\end{equation}

Lemma \ref{lem:induction_step_same_as_F_annih} has the following convenient implication. Since as part of the inductive hypothesis, we assume that both  $N_{r-1,r-2}$ and $N_{r-2,r-3}$ are true, we have the statement ${\mathbf{F_{r-1}}}:$ 
$\Dem _{w_0^{(r-2)}}F_{\nu,\Ga_{12}}(\z)=0$ for any $\Ga _{12}.$ This is true even when $r=2,$ since in this case, $F_{\nu,\Ga_{12}}(\z)$ is itself zero. 

The strategy to prove \eqref{eq:a_divn_goal} is the following. Using the conventions introduced in Section \ref{sect:conventions}, in particular \eqref{eq:munu_gammarelate} and \eqref{eq:munu_summandrelate}, we will rewrite the sum defining $F_{\mu,a}(\y )$ into smaller pieces according to $\Ga _0.$ Then we write $F_{\mu,a}(\y )$ as a difference of two pieces. One piece is annihilated by $\Dem _{w_0^{(r-2)}}$ as a consequence of ${\mathbf{F_{r-1}}},$ the other is annihilated by $\Dem _{r-1}.$ By Lemma \ref{lemma:Demazure_annihilation}, this implies that $F_{\mu,a}(\y )$ is indeed annihilated by $\Dem _{w_0^{(r-1)}}.$ 

We start by breaking up the sum in $F_{\mu,a}(\y )$ according to $\Ga _0.$ By the conventions introduced above, we have 
\begin{equation}\label{eq:bigcomputation_1}
 \begin{split}
F_{\mu,a}(\y ) = &  \sum _{\substack{\Gamma =(\Gamma _{12},\ldots ,\Gamma _{1r})\\ \Gamma \ \mu -{\mathrm{admissible}}\\ \Gamma _{1,r}\neq 0}} \delta _a(\Ga _{12})\cdot g^{(\mu )}_{12}(\Ga )\cdot G^{(\nu )}_1(\Ga _0)\cdot \z ^{\wght^{(\nu )}(\Ga _0)}\cdot x_{r-1}^{\Ga _{12}}\cdot x_{r}^{\la _2-\Ga _{12}+a} \\
= &  \sum _{\substack{\Gamma _0=(\Gamma _{13},\ldots ,\Gamma _{1r})\\ \Gamma _0\ \nu -{\mathrm{admissible}}\\ \Gamma _{1,r}\neq 0}} G^{(\nu )}_1(\Ga _0)\cdot \z ^{\wght^{(\nu )}(\Ga _0)}\cdot \sum _{\Ga _{12}=\Ga_{13}}^{\Ga _{13}+\la_2-\la_3+1} \delta _a(\Ga _{12})\cdot g^{(\mu )}_{12}(\Ga )\cdot x_{r-1}^{\Ga _{12}}\cdot x_{r}^{\la _2-\Ga _{12}+a}  
 \end{split}
\end{equation}
Recall that $\Ga _0 =(\Ga _{13},\ldots ,\Ga _{1r}),$ where $\Ga _{13}\geq \Ga _{1r}\geq 1,$ and if $\Ga _{12}<\Ga _{13}$ then $g^{(\mu )}_{12}(\Ga )=0.$ Hence we may change the lower bound of the summation over $\Ga _{12}$ to $1.$ (If $r=2,$ then there is no change at all.) 
\begin{equation}\label{eq:bigcomputation_1andhalf}
F_{\mu,a}(\y ) =   \sum _{\substack{\Gamma _0=(\Gamma _{13},\ldots ,\Gamma _{1r})\\ \Gamma _0\ \nu -{\mathrm{admissible}}\\ \Gamma _{1,r}\neq 0}} G^{(\nu )}_1(\Ga _0)\cdot \z ^{\wght^{(\nu )}(\Ga _0)}\cdot \sum _{\Ga _{12}=1}^{\Ga _{13}+\la_2-\la_3+1} \delta _a(\Ga _{12})\cdot g^{(\mu )}_{12}(\Ga )\cdot x_{r-1}^{\Ga _{12}}\cdot x_{r}^{\la _2-\Ga _{12}+a}  
\end{equation}

Next we derive from ${\mathbf{F_{r-1}}}$ a term annihilated by $\Dem _{w_0^{(r-2)}}.$ 
Recall that ${\mathbf{F_{r-1}}}$ implies that $\Dem _{w_0^{(r-2)}} F_{\nu,\Ga_{12}}(\z ) =0$ holds for any $0<\Ga_{12}.$ The operator $\Dem _{w_0^{(r-2)}}$ is linear and commutes with multiplication by $x_r.$ This implies that for any positive integer $a,$ we have 
\begin{equation}\label{eq:summing_smaller_Fs_formal}
 \Dem _{w_0^{(r-2)}} \sum _{0<\Ga_{12}} \delta _a(\Ga _{12})\cdot x_{r}^{\la _2-\Ga _{12}+a} \cdot  F_{\nu,\Ga_{12}}(\z )=0.
\end{equation}
We take a closer look at this polynomial annihilated in \eqref{eq:summing_smaller_Fs_formal}. It is
\begin{equation}\label{eq:summing_smaller_Fs}
 \sum _{0<\Ga_{12}} \delta _a(\Ga _{12})\cdot x_{r}^{\la _2-\Ga _{12}+a} \cdot \sum _{\substack{\Gamma _0=(\Gamma _{13},\ldots ,\Gamma _{1r})\\ \Gamma _0\ \nu -{\mathrm{admissible}}\\ \Gamma _{1,r}\neq 0}} \delta (\Ga_{12},\Ga _{13})\cdot G^{(\nu )}_1(\Gamma _0)\cdot \z ^{\wght^{(\nu )}(\Ga _0)}\cdot x_{r-1}^{\Ga _{12}}.
\end{equation}
We may change the order of summation to get
\begin{equation}\label{eq:summing_smaller_Fs_modpart}
 \sum _{\substack{\Gamma _0=(\Gamma _{13},\ldots ,\Gamma _{1r})\\ \Gamma _0\ \nu -{\mathrm{admissible}}\\ \Gamma _{1,r}\neq 0}} G^{(\nu )}_1(\Gamma _0)\cdot \z ^{\wght^{(\nu )}(\Ga _0)}\cdot \sum _{0<\Ga_{12}} \delta _a(\Ga _{12})\cdot  \delta (\Ga_{12},\Ga _{13})\cdot x_{r-1}^{\Ga _{12}} \cdot x_{r}^{\la _2-\Ga _{12}+a}  .
\end{equation}
Since $\delta (\Ga_{12},\Ga _{13})=0$ when $\Ga _{12}>\Ga _{13},$ the sum is unchanged if we put the upper bound $\Ga _{13}+\la_2-\la_3+1$ on the second summation:
\begin{equation}\label{eq:summing_smaller_Fs_mod}
 \sum _{\substack{\Gamma _0=(\Gamma _{13},\ldots ,\Gamma _{1r})\\ \Gamma _0\ \nu -{\mathrm{admissible}}\\ \Gamma _{1,r}\neq 0}} G^{(\nu )}_1(\Gamma _0)\cdot \z ^{\wght^{(\nu )}(\Ga _0)}\cdot \sum _{\Ga_{12}=1}^{\Ga _{13}+\la_2-\la_3+1} \delta _a(\Ga _{12})\cdot  \delta (\Ga_{12},\Ga _{13})\cdot x_{r-1}^{\Ga _{12}} \cdot x_{r}^{\la _2-\Ga _{12}+a}  .
\end{equation}
Comparing \eqref{eq:summing_smaller_Fs_mod} to \eqref{eq:bigcomputation_1andhalf}, we see that their shape is very similar. It suffices to show that their sum is annihilated by $\Dem _{r-1}.$ We write the sum explicitly as follows. 
\begin{equation}\label{eq:diff_for_Dem_r-1}
\begin{split}
F_{\mu,a}(\y ) +\sum _{0<\Ga_{12}} \delta _a(\Ga _{12})\cdot x_{r}^{\la _2-\Ga _{12}+a}  \cdot  F_{\nu,\Ga_{12}}(\z ) & = 
 \sum _{\substack{\Gamma _0=(\Gamma _{13},\ldots ,\Gamma _{1r})\\ \Gamma _0\ \nu -{\mathrm{admissible}}\\ 
 \Gamma _{1,r}\neq 0}} G^{(\nu )}_1(\Ga _0)\cdot \z ^{\wght^{(\nu )}(\Ga _0)}\cdot  \\
 \cdot \sum _{\Ga _{12}=1}^{\Ga _{13}+\la_2-\la_3+1} \delta _a(\Ga _{12})\cdot (g^{(\mu )}_{12}(\Ga )+&\delta (\Ga_{12},\Ga _{13}))\cdot x_{r-1}^{\Ga _{12}}\cdot x_{r}^{\la _2-\Ga _{12}+a} \\
& =  \sum _{\substack{\Gamma _0=(\Gamma _{13},\ldots ,\Gamma _{1r})\\ \Gamma _0\ \nu -{\mathrm{admissible}}\\ \Gamma _{1,r}\neq 0}} G^{(\nu )}_1(\Ga _0)\cdot \z ^{\wght^{(\nu )}(\Ga _0)}\cdot x_{r-1}^{\Ga _{13}-\la _3}\cdot \\
\cdot \sum _{\Ga _{12}=1}^{\Ga _{13}+\la_2-\la_3+1} \delta _a(\Ga _{12})\cdot (g^{(\mu )}_{12}(\Ga )+&\delta (\Ga_{12},\Ga _{13}))\cdot x_{r-1}^{\la _3-\Ga _{13}+\Ga _{12}}\cdot x_{r}^{\la _2-\Ga _{12}+a} 
\end{split}
\end{equation}


Note that the exponent of $x_{r-1}$ in $\z ^{\wght^{(\nu )}(\Ga _0)}$ is $\la _3-\Ga _{13}.$ This means that $\z ^{\wght^{(\nu )}(\Ga _0)}\cdot x_{r-1}^{\Ga _{13}-\la _3}$ contains no factors of $x_{r-1}$ or $x_r.$ In particular, multiplication by $\z ^{\wght^{(\nu )}(\Ga _0)}\cdot x_{r-1}^{\Ga _{13}-\la _3}$ commutes with $\Dem _{r-1}.$ This implies that the sum in \eqref{eq:diff_for_Dem_r-1} is annihilated by $\Dem _{r-1}$ if and only if the part corresponding to a fixed $\Ga _0$ is. 

Consider
\begin{equation}\label{eq:h_newcoeff_def}
 h^{(\mu, \Ga _0)}(\Ga _{12}):=g^{(\mu )}_{12}(\Ga )+\delta (\Ga _{12},\Ga_{13})=\left\lbrace\begin{array}{ll}
                              h^{\flat}(\Ga _{12}) & \ifi \Ga _{12}<\Ga_{13}+\la_{2}-\la_{3}+1;\\
                              g^{\flat}(\Ga _{12}) & \ifi \Ga _{12}=\Ga_{13}+\la_{2}-\la_{3}+1;\\
                              0 & \ifi \Ga _{12}>\Ga_{13}+\la_{2}-\la_{3}+1.\\
                              \end{array}
\right.
\end{equation}
Let $f_{a,\Ga _0}$ be the polynomial that corresponds to a single $\Ga _0$ in \eqref{eq:diff_for_Dem_r-1}: 
\begin{equation}\label{eq:fdef}
 f_{a,\Ga _0}(x_{r-1},x_r):=\sum _{\Ga _{12}=1}^{\Ga _{13}+\la_2-\la_3+1}\delta _a(\Ga _{12})\cdot h^{(\mu, \Ga _0)}(\Ga _{12})\cdot x_{r-1}^{\la _3+\Ga_{12}-\Ga _{13}}x_r^{\la _2+a-\Gamma _{12}}.
\end{equation}

\begin{lemma}\label{LEMMA:FANNIHILATED}
\begin{equation}\label{eq:f'annihilated}
 \Dem _{r-1}f_{a,\Ga _0}=0.
\end{equation}
\end{lemma}

The proof of this lemma is a rank one computation using the definition of the group action. For completeness, it is included in Appendix \ref{app:rankonecomp}. 

To summarize, we have
\begin{equation}\label{eq:parts_together}
\begin{split}
F_{\mu,a}(\y ) = & \sum _{\substack{\Gamma _0=(\Gamma _{13},\ldots ,\Gamma _{1r})\\ \Gamma _0\ \nu -{\mathrm{admissible}}\\ \Gamma _{1,r}\neq 0}} G^{(\nu )}_1(\Ga _0)\cdot \z ^{\wght^{(\nu )}(\Ga _0)}\cdot x_{r-1}^{\Ga _{13}-\la _3}\cdot f_{a,\Ga _0}(x_{r-1},x_r)\\
& - \sum _{0<\Ga_{12}} \delta _a(\Ga _{12})\cdot x_{r}^{\la _2-\Ga _{12}+a} \cdot  F_{\nu,\Ga_{12}}(\z )
\end{split}
\end{equation}
The first term here is annihilated by $\Dem _{r-1};$ the second is annihilated by $\Dem _{w_0^{(r-2)}}.$ Hence $F_{\mu,a}(\y )$ is annihilated by $\Dem _{w_0^{(r-1)}}.$ This completes the proof of the statement ${\mathbf{F_r}}$ (in the divisible case), and hence the proof of Proposition \ref{prop:N_r_r-1}.

%% file: whittaker.tex
\section{Whittaker functions}\label{section:Whittaker}


We mentioned in the Introduction that Theorem \ref{thm:metaplectic_Tokuyama} establishes a combinatorial link between metaplectic analogues of the Casselman-Shalika formula. Furthermore, the more general Theorem \ref{THM:MAIN} gives a crystal description of certain Iwahori-Whittaker functions. In this section, we make these statements more explicit. 

We recall results of \cite{cgp} to compare the ``Demazure-Lusztig side'' of Theorem \ref{thm:metaplectic_Tokuyama} to constructions in \cite{mcnamara2} and more specifically in type $A$ to \cite{co-cs} (section \ref{section:Dem_Whitt}). We relate the ``crystal side'' of Theorem \ref{thm:metaplectic_Tokuyama} to Whittaker functions via comparison to \cite{mcnamara} (section \ref{section:crystal_Whitt}). In section \ref{section:IwWhit_DL} we compare the Demazure-Lusztig expression from Theorem \ref{THM:MAIN} to constructions of Iwahori-Whittaker functions. We recall a relevant statement in three different contexts. For finite dimensional and affine Kac-Moody groups in the nonmetaplectic setting, through comparison with \cite{bbl-demazure} and \cite{patnaik2014unramified} respectively, and in the metaplectic setting by recalling results of  \cite{patnaik2015casselman}. 

\subsection{Metaplectic Whittaker functions via Demazure-Lusztig operators}\label{section:Dem_Whitt}

Formulae about Demazure operators of the long word, in particular Theorem \ref{THM:LONG_WORD} \cite[Theorem 3.]{cgp} and Theorem \ref{THM:T_SUM} \cite[Theorem 4.]{cgp} allow us to interpret the left-hand side of Theorem \ref{THM:MAIN} as the value of a Whittaker function $\W ,$ constructed as a sum over the Weyl group in terms of the Chinta-Gunnells action. Such a construction can be found in \cite{co-cs} for type $A,$ and in \cite{mcnamara2} in greater generality. The connection with results of \cite{mcnamara} are made explicit in \cite[Section 6]{cgp}. We recall that result here, and give the translation to results of \cite{co-cs} in the type $A$ case. 

\subsubsection{Metaplecic Whittaker functions}

We only sketch the definition of the Whittaker function $\W$ and refer the reader to \cite[Section 6]{cgp} for details and precise conditions in our notation. Let $F$ be a non-archimedean local field containing the $2n$-th roots of unity, $\varpi $ the uniformizer of its ring of integers $\OO$. Let $G$ be a split, connected reductive group over $F$ that arises as a special fibre over a group scheme $\G$ defined over $\Z.$ Let $K=\G (\OO)$ be the maximal compact subgroup of $G$, $T$ a maximal split torus, $B$ a Borel containing $T,$ $B^-$ its opposite, $U$ the unipotent radical of $B$ and $U^{-}$ of the opposite Borel $B^-.$ If $\La $ is the group of cocharacters of $T$ then we may define a sublattice $\La _0$ of $\La $ as in \eqref{eqn:lambda0_intypeA} (for the definition in general type, see \cite[Section 2, (3)]{cgp}). Let $\tG$ be an $n$-fold metaplectic cover of $G.$ This in particular means that there is a short exact sequence
\begin{equation}\label{eq:exact_metapl}
 1\ra \mu _n \ra \tilde{G}\ra G\ra 1
\end{equation}
where $\mu_n$ is the group of $n$-th roots of unity. We think of $\mu _n$ as being identified with a subgroup of $\C^{\ast}$ and let $\tT,$ the metaplectic torus (respectively, $\tB$) be the preimage of $T$ (respectively, of $B$) in $\tG.$ 
We shall give $\W$ as a complex-valued Whittaker function corresponding to an unramified principal series representation of $\tG. $ Let $\chi $ be a character of $\La _0$ and $\psi :U^-\ra \C$ an unramified character. Then $\chi $ determines an extension $\widetilde{\chi}$ to $\tT$ as well as a representation $\iota (\chi )$ of $\tT$ (induced from its centralizer). In turn, $\iota (\chi )$ determines an unramified principal series representation of $\tG.$ Let $\phi_K$ be a spherical vector; then there is a $\xi_{\widetilde{\chi}}\in (\iota (\chi ))^{\ast }$ a complex valued linear functional corresponding to $\chi$ and $\phi _K.$ From these we arrive at the Whittaker function 
\begin{equation}\label{eq:Cvalued_Whittaker}
 \W =\W _{\widetilde{\chi }}:\ g\mapsto \xi _{\tilde{\chi }}\left(\int _{U^-} \phi _K(ug)\psi (u)du\right).
\end{equation}

\subsubsection{Comparison}
It is a consequence of the construction (see \cite[Section 6]{cgp}) that $\W$ satisfies 
$$\W(\zeta ugk)=\zeta \phi (u)\W (g),\hskip .cm \zeta \in \mu _n,\ u\in U,\ g\in \tilde{G},\ k\in K.$$
This fact and the Iwasawa decomposition $G=UTK$ together imply that it suffices to compute $\W$ on $\tilde{T}.$ The identification $\widetilde{\chi}(\varpi^{\la})=\x^{\la }$ interprets the action of the Weyl group on $\La$ as $W$ acting on $\widetilde{\chi}.$ Setting $v=t^n=q^{-1}$ where $q$ is the order of the residue field $\OO/\varpi\OO,$ it makes sense to talk about a value $(\delta ^{-1/2}\W_{\widetilde{\chi}})(\varpi ^{\la })$ in terms of the expressions produced by metaplectic Demazure and Demazure-Lusztig operators acting on monomials. (Here $\delta$ is the modular quasicharacter of $\tB.$) In particular, by \cite[Theorem 15.2]{mcnamara2}, one has
\begin{theorem}\label{thm:Whit_DL_McN}
\cite[Theorem 16.]{cgp} For a dominant coweight $\la $
\begin{equation}\label{eq:Whit_DL_McN}
\begin{split}
(\delta ^{-1/2}\W_{\widetilde{\chi}})(\varpi ^{\la }) = & \prod _{\alpha \in \Phi ^+}(1-q^{-1}x^{m(\alpha )\alpha })\Dem _{w_0} (\x^{w_0\la })\\
= & \left(\sum_{w\in W}\T_w\right) (\x^{w_0\la })
\end{split}
\end{equation}
\end{theorem}
Note that the equality of the two lines on the right-hand side is a consequence of \cite[Theorem 4.]{cgp} (i.e.~Theorem \ref{THM:T_SUM}). Theorem \ref{eq:Whit_DL_McN} is valid in any type. 

We finish this comparison by arriving at the same statement in type $A$ using results of \cite{co-cs}. As in \cite[Section 9]{co-cs} define 
\begin{equation}\label{eq:jwxdef}
 j(w,x)= \frac{\prod _{\alpha \in \Phi ^+} (1-x^{m(\alpha )\alpha })}{\prod _{\alpha \in w(\Phi ^+)} (1-x^{m(\alpha )\alpha })}.
\end{equation}

We have the following formula of Chinta-Offen:

\begin{theorem}\label{thm:CO}
\cite[Theorem 4]{co-cs} Let $\la $ be a dominant coweight. Then 
 $$(\delta ^{-1/2}\W _{\tilde {\chi }})(\varpi ^{\la })=c_{w_0}(x)\cdot \sum _{w\in W} j(w,x)\cdot w(x^{w_0\la })$$
where $w$ acts on $x^{\la }$ as in Definition \ref{def:action}. 
\end{theorem}

We may rewrite $j(w,x)$ in a more familiar form. A simple computations shows that for any $w\in W$ we have
\begin{equation}\label{eq:jwxrewrite}
 j(w,x)= \sgn (w)\cdot \prod _{\alpha \in \Phi (w^{-1})} x^{m(\alpha )\alpha }. 
\end{equation}
Combining Theorem \ref{thm:CO} with \cite[Theorem 3.]{cgp} (i.e.~Theorem \ref{THM:LONG_WORD}) we arrive at the type $A$ special case of Theorem \ref{thm:Whit_DL_McN} once again.

\subsection{Construction via crystals}\label{section:crystal_Whitt}

McNamara \cite[Section 8]{mcnamara} expresses the value of a Whittaker function $I_{\la }$ as a sum over a the highest weight crystal $\Cr _{\la +\rho }.$ Here $\la $ is a dominant weight and $I_{-w_0\la }$ is the same as $(\delta ^{-1/2}\W _{\tilde {\chi }})(\varpi ^{\la })$ above, up to a relatively trivial constant factor. (It follows from our argument below that the value of this factor is in fact one.) The crystal $\Cr _{\la +\rho }$ is parametrized in \cite{mcnamara} in terms of Lusztig data; to compare results we sketch the translation of Mcnamara's result into the notation of Gelfand-Tsetlin arrays. 

\subsubsection{Lusztig data and McNamara's result}
We recall notation from \cite[Section 8.]{mcnamara}. Note first that the long word chosen in {\em{loc. cit}} agrees with our choice of $w_0$ from \eqref{eq:def_of_favourite_w0}. Let $\la =(\la _1,\ldots ,\la _{r+1})$ and let us use the notation for a root system of type $A_r$ as before (see section \ref{subs:CGDemNotation}). In particular, recall that we have $\Phi ^+=\{\alpha _{i,j}=e_i-e_j\mid 1\leq i<j\leq r+1\}.$ 
\begin{prop}\label{prop:Lusztig_par}
\cite[Proposition 8.3]{mcnamara} Elements of $\Cr_{\la+\rho }$ are in bijection with tuples $\mm =(m_{ij})_{1\leq i<j\leq r+1}$ where $0\leq m_{ij}$ integers and
\begin{equation}\label{eq:mij_upper_bound_old}
 \sum _{k=j}^{r+1} m_{i,k} \leq \la _{i}-\la _{i+1}+1+\sum _{k=j}^{r} m_{i+1,k+1}.
\end{equation}
\end{prop}
We write $\mm \in \Cr_{\la+\rho }$ for tuples as above. For an $\alpha =\alpha _{i,j}\in \Phi ^+$ we say $m_{\alpha }:=m_{i,j}$ is {\em{circled}} if $m_{i,j}=0,$ and {\em{boxed}} if equality holds in \eqref{eq:mij_upper_bound_old}. 
Furthermore, define 
\begin{equation}\label{eq:rij_def}
 r_{i,j}=r_{\alpha }=\sum _{k\leq i} m_{k,j}.
\end{equation}
Now we may use the functions $h^{\flat }$ and $g^{\flat }$ defined in Section \ref{subsubsect:gflathflatdef} to define a coefficient corresponding to $\mm\in \Cr _{\la +\rho }$. Let
\begin{equation}\label{eq:def_wmalpha}
 w(\mm,\alpha )=\left\lbrace\begin{array}{ll}
                            1 & \text{if $m_{\alpha }$ is circled, but not boxed,}\\
                            h^{\flat }(r_{\alpha }) & \text{if $m_{\alpha }$ is not circled and not boxed,}\\
                            g^{\flat }(r_{\alpha }) & \text{if $m_{\alpha }$ is boxed, but not circled,}\\
                            0 & \text{if $m_{\alpha }$ is both circled and boxed}
                                                                                        \end{array}\right.
\end{equation}
It is straightforward to check that \eqref{eq:def_wmalpha} defines the same weight as \cite[(8.2)]{mcnamara} by comparing section \ref{subsubsect:gflathflatdef} with \cite[(2.1)]{mcnamara} and \cite[Lemma 2.5]{bbf-wmd2}.

With the notation as above, one has

\begin{theorem}\label{thm:McN_crystal_Whittaker_Lusztig_notation}
\cite[Theorem 8.6]{mcnamara2} The value of the integral $I_{\la}$ which calculates the metaplectic Whittaker function is zero unless $\la $ is dominant; and for dominant $\la $ it is given by 
\begin{equation}\label{eq:Peters_crystal_formula}
 I_{\la}=\sum _{\mm\in \Cr _{\la +\rho }} \prod _{\alpha \in \Phi ^+} w(\mm ,\alpha )\cdot x^{m_{\alpha }\alpha }.
\end{equation}
\end{theorem}
As before, we write $\x$ such that that $\chi (\varpi ^{\la })=\x^{\la }$ for the unramified $\chi $ used to define the principal series representation.

\subsubsection{Translation into Gelfand-Tsetlin language}

It is convenient to compare the Lusztig data of $\Cr_{-w_0\la +\rho }$ to the Gelfand-Tsetlin arrays for $\Cr_{\la +\rho }.$ For $\Cr_{-w_0\la +\rho },$ the condition \eqref{eq:mij_upper_bound_old} is replaced by $\mm\in \Cr_{-w_0\la +\rho }$ if and only if $m_{\alpha}\geq 0$ and
\begin{equation}\label{eq:mij_upper_bound_new}
 \sum _{k=j}^{r+1} m_{i,k} \leq \la _{r+1-i}-\la _{r+2-i}+1+\sum _{k=j}^{r} m_{i+1,k+1};
\end{equation}
and $m_{i,j}$ is boxed if \eqref{eq:mij_upper_bound_new} is satisfied with an equality. 

Consider the following bijection between $\mm\in \Cr_{-w_0\la +\rho }$ and $\Gamma (\IP_v)$ for $v\in \Cr _{\la +\rho }:$
\begin{equation}\label{eq:bijection_r_Gamma}
 \mm \mapsto \Gamma (\IP )\hskip .5 cm \text{if}\hskip .5 cm \Gamma _{h,k}=r_{r+1-k,r+2-h}. 
\end{equation}
Note that $(h,k)$ satisfy $1\leq h\leq k\leq r$ if and only if $i=r+1-k$ and $j=r+2-h$ satisfy $1\leq i<j\leq r+1.$ Further, $ m_{i,j}=r_{i,j}-r_{i-1,j}.$ The bijection may be expressed in terms of the corresponding $\Gamma $-array as 
\begin{equation}\label{eq:m_from_Ga}
 m_{i,j}=\Ga _{h,k}-\Ga _{h,k+1}.
\end{equation}
Thus $m_{i,j}$ is circled if and only if $\Ga _{h,k}=\Ga _{h,k+1},$ i.e. $\Ga _{h,k}$ is circled by Definition \ref{def:array_decorations}. Similarly, one may check that $m_{i,j}$ is boxed if and only if $\sum _{t=j}^{r+1} m_{i,t} = a_{h-1,k-1}-a_{h,k}$ i.e. if and only if $\Ga _{h,k}$ is boxed. Comparing the definition of $w(\mm ,\alpha )$ in \eqref{eq:def_wmalpha} with the definition of the Gelfand-Tsetlin coefficients in \ref{def:GTcoeff}, we find that for any $\mm \in \Cr_{-w_0\la +\rho }$ and corresponding $v\in \Cr_{\la +\rho },$ we have 
\begin{equation}\label{eq:coeff_rewrite}
 \prod _{\alpha \in \Phi ^+} w(\mm ,\alpha )=G^{(n,\la +\rho )}(v).
\end{equation}
An other direct computation shows that 
\begin{equation}\label{eq:wght_v_r}
 \wght(v)-w_0(\la +\rho )= \sum _{\alpha \in \Phi ^+} m _{\alpha }\alpha .
\end{equation}

Then Theorem \ref{thm:McN_crystal_Whittaker_Lusztig_notation}, \eqref{eq:coeff_rewrite} and \eqref{eq:wght_v_r} immediately yields the following. 

\begin{theorem}\label{prop:McN_crystal_Whittaker_our_notation}
 The value of the integral $I_{-w_0\la}$ which calculates the metaplectic Whittaker function is zero unless $\la $ is dominant; and for dominant $\la $ it is given by 
\begin{equation}\label{eq:Peters_crystal_formula_ournotation}
 I_{-w_0\la}=\x^{-w_0(\la +\rho )}\cdot \sum _{v\in \Cr_{\la +\rho }} G^{(n,\la +\rho )}(v)\cdot \x^{\wght(v)}.
\end{equation}
\end{theorem}

The content of Theorem \ref{prop:McN_crystal_Whittaker_our_notation} identifies the ``crystal side'' of Theorem \ref{thm:metaplectic_Tokuyama} as a metaplectic Whittaker function. 

\subsection{Constructions of Iwahori-Whittaker functions}\label{section:IwWhit_DL}

The operators $\T_u$ lend themselves to the study of Whittaker functions not only through expressions of the long word, as seen in section \ref{section:Dem_Whitt} but also through their relationship to Iwahori-Whittaker functions. This was mentioned in \ref{subsubsection:BBLextend}; in this section, we elaborate on this connection by recalling results of \cite{bbl-demazure}, \cite{patnaik2014unramified} and \cite{patnaik2015casselman}. These sources express Iwahori-Whittaker functions in terms reminiscent of the left-hand side of Theorem \ref{THM:MAIN}:
\begin{equation}\label{eq:Tsum_op_whchap}
\sum_{u\leq w} \T_u
\end{equation}
in the nonmetaplectic finite-dimensional, loop group, and metaplectic finite-dimensional setting, respectively. 

\subsubsection{The Whittaker functional and Iwahori fixed vectors}\label{subsubect:BBL}
Brubaker, Bump and Licata \cite{bbl-demazure} consider the values of a Whittaker functional on Casselman's basis of functions in a principal series representation fixed by the Iwahori subgroup. They work with the classical, nonmetaplectic Demazure-Lusztig operators $\IP _w$ ($w\in W$). (Their definition \cite[(2), (3)]{bbl-demazure} essentially agrees with our Definition \ref{def:dl} when $n=1$.)

We recall some additional notation; see \cite{bbl-demazure} for the precise definitions of the objects involved. 
Let $\hat{T}$ be a split maximal torus of the Langlands dual $\hat{\bG};$ as explained in \cite[Section 2]{bbl-demazure}, $\z \in \hat{T}(\C )$ corresponds to an unramified character $\tau_z$ of $T.$ An element $\la \in \La $ corresponds to a coset $T(F)/T(\OO);$ let $a_{\la }$ be a coset representative. Consider the principal series representation $\pi =\Ind _B^G(\tau).$ Let $\Omega _{\tau }:\Ind _B^G(\tau)\ra \C$ denote the Whittaker functional. Let $J$ be the Iwahori subgroup (i.e. the preimage of $B^-(\OO/\varpi \OO)$ in $K$),  the space $\Ind _B^G(\tau)^J$ of Iwahori fixed vectors has a standard basis $\{\Phi _w^{\z}\}_{w\in W}$, whose elements are supported on Iwahori double-cosets. 

The Iwahori Whittaker functions $\W _{\la ,y}(\z)$ are (a convenient normalization of) the values of a Whittaker functional on standard basis elements:
\begin{equation}\label{eq:Whi_object_first}
 \W _{\la ,w}(\z)=\delta ^{-1/2} (a_{\la })\Omega _{\z ^{-1}}(\pi (a_{-\la })\Phi _w ^{z^{-1}}).
\end{equation}
We may also consider \cite[Section 5]{bbl-demazure} the modification $\tilde {\W} _{\la ,w}(\z)=\sum _{y\leq w} \W _{\la ,y}(\z).$ 
The connection between Iwahori Whittaker functions and Demazure-Lusztig operators is expressed by the following theorem. 

\begin{theorem}\label{thm:BBL_1}
\cite[Theorem 1.]{bbl-demazure} For any dominant weight $\la ,$ we have $\W _{\la ,1}(\z)=\z ^{\la }.$
Furthermore, if $w\in W$ and $\sigma _i$ is a simple reflection such that $\sigma _iw>w$ by the Bruhat order, then 
$$\W _{\la ,\sigma _iw}=\IP _i \W _{\la,w} (\z ).$$
\end{theorem}

The following straightforward corollary illustrates the relevance of operators \eqref{eq:Tsum_op_whchap}.
\begin{corollary}\label{cor:BBL_2}
 For any dominant weight $\la $ and $w\in W,$ we have 
$$\tilde {\W} _{\la ,w}(\z)=\left(\sum _{y\leq w} \IP _y \right) \z ^{\la }.$$
\qed
\end{corollary}

\subsubsection{Iwahori-Whittaker sums on loop groups}\label{subsubect:LoopGroup}

In this section we shift our perspective slightly. We recall results of Patnaik \cite{patnaik2014unramified} that demonstrate the use of Demazure-Lusztig operators in the study of Whittaker functions in yet an other setting: on $p$-adic points of an affine Kac-Moody group. 

Let $G$ be an affine Kac-Moody group over a non-archimedean field; $\varpi ,$ $q,$ $K,$ $U,$ $U^-$ as before. Let $W$ now denote the (affine) Weyl group of $G;$ and $\Pi _0$ the basis of the corresponding finite root system. Let $I$ and $I^{-}$ denote the Iwahori subgroups. In \cite[Section 4]{patnaik2014unramified} the Whittaker function $\W$ is defined on $G$. Furthermore, determining $\W$ is reduced to the computation of values $\W(\varpi ^{\lv}),$ for any $\lv\in \Lv $ affine coweight. The main theorem \cite[Theorem 7.1]{patnaik2014unramified}, a generalization of the Casselman-Shalika formula for the computation of $\W(\varpi ^{\lv})$ is proved through the introduction of Iwahori-Whittaker sums $\W _{w,\lv}$ and a recursion result on $\W _{w,\lv}$ in terms of Demazure-Lusztig operators. The recursion result \cite[Proposition 5.5]{patnaik2014unramified} is recalled in Proposition \ref{prop:Manish_loop_recursion} below. 

The definition \cite[(2.29)]{patnaik2014unramified} of Demazure-Lusztig operators $T_a$ ($a\in \Pi _0$) essentially agrees with Definition \ref{def:dl} when $n=1$ and the root system is of type $A.$ 
The Iwahori-Whittaker sums $\W _{w,\lv}$ are defined \cite[Definition 4.5]{patnaik2014unramified} by summing an unramified principal character $\psi $ of $U^-$ along fibres of the multiplication map $m_{w,\lv}:\ UwI^-\times_{I^-}I^-\varpi ^{\lv}U^-\ra G:$ 
\begin{equation}\label{eq:Manish_IW_def}
\W _{w,\lv} := \sum _{\mu^{\vee}} e^{\mu^{\vee}} q^{\innprod{\rho}{\mu^{\vee}}} \sum _{x\in m_{w,\lv}^{-1}(\varpi ^{\mu^{\vee}})} \psi (x)
\end{equation}
(For details on how to interpret the unramified character $\psi $ on elements of the fibre, see {\em{loc.cit.}}) 

The Whittaker function may then be written as a sum of these Iwahori-Whittaker sums \cite[(4.21)]{patnaik2014unramified}: $\W (\pi ^{\lv}) = \sum _{w\in W} \W _{w,\lv}.$ The following proposition phrases the recursion of the Iwahori-Whittaker sums in terms of Demazure-Lusztig operators. 

\begin{prop}\label{prop:Manish_loop_recursion}
\cite[Proposition 5.5]{patnaik2014unramified} Fix $\lv \in \Lv$ and let $w,w'\in W$ be related by $w=\sigma _a w'$ with $a\in \Pi _0$ a simple root, and $\ell(w)=\ell(w')+1.$ Then the following identity holds in $\C _{\sigma _a}[\Lv ]:$
\begin{equation}\label{eq:Manish_loop_recursion}
\W_{w,\lv}=T_a(\W_{w',\lv}).
\end{equation}
(Here the simple reflection in $T_a$ acts on $\W_{w',\lv}$ termwise; the ring $\C _{\sigma _a}[\Lv ]$ is an extension of $\C [\Lv]$ containing $T_a(\W_{w',\lv}).$)
\end{prop}

This recursion is used in \cite[Section 7.2]{patnaik2014unramified} to conclude that $\W_{w,\lv}=q^{\rho,\lv}T_w(e^{\lv})$ \cite[(7.3)]{patnaik2014unramified} and to compute the value of $\W (\varpi ^{\lv})$, proving the generalization of the Casselman-Shalika formula. 

\subsubsection{The metaplectic setting}\label{subsubect:MMP_PA}

Recent results of Manish Patnaik and the present author \cite{patnaik2015casselman} indicate that Demazure-Lusztig operators can be used to express Iwahori-Whitaker functions directly in the metaplectic setting as well. In particular, the techniques of \cite{patnaik2014unramified} seen above are applicable in the finite dimensional metaplectic setting. The Iwahori-Whittaker functions $\W _{w,\lv}$ can again be defined as a formal generating series using fibres of the map $m_{w,\lv}$; an argument similar to that in \cite{patnaik2014unramified} proves that the value of the metaplectic Whittaker function $\W(\pi ^{\lv})$ can be expressed as a sum of these: $\W(\pi ^{\lv})= q^{-\innprod{2\rho}{\lv}}\sum _{w\in W} \W _{w,\lv}$ \cite[Section 5.2]{patnaik2015casselman}. It turns out that then the $\W _{w,\lv}$ satisfy a similar recursion to the one in Proposition \ref{prop:Manish_loop_recursion} above \cite[Theorem 5.4]{patnaik2015casselman}, and, consequently, writing $\T _w$ for the metaplectic Demazure-Lusztig operators, we have $\W _{w,\lv} = q^{\innprod{2\rho}{\lv}}\T _w (e^{\lv})$ \cite[Corollary 5.4]{patnaik2015casselman}. 

%% file: rankone_comp.tex
\section{Proof of Lemma \ref{LEMMA:FANNIHILATED}}\label{app:rankonecomp}

We prove Lemma \ref{LEMMA:FANNIHILATED}:
\begin{equation}\label{eq:f_annihilated_goal}
 \Dem _{r-1}f_{a,\Ga _0}(x_{r-1},x_r)=\Dem _{r-1}\sum _{\Ga _{12}=1}^{\Ga _{13}+\la_2-\la_3+1}\delta _a(\Ga _{12})h^{(\mu, \Ga _0)}(\Ga _{12}). x_{r-1}^{\la _3+\Ga_{12}-\Ga _{13}}x_r^{\la _2+a-\Gamma _{12}}=0.
\end{equation}

By Lemma \ref{lemma:Demazure_annihilation}, it suffices to show $x_{r}^n\cdot f_{a,\Ga _0}$ is symmetric under the action of $\sigma _{r-1}.$ 
Since $\delta _a(\Ga _{12})=0$ if $a>\Ga _{12},$ we have 
\begin{equation}\label{eq:f_timesx_rton}
 x_{r}^n\cdot f_{a,\Ga _0}=\sum _{\Ga _{12}=a}^{\Ga _{13}+\la_2-\la_3+1}\delta _a(\Ga _{12})\cdot h^{(\mu, \Ga _0)}(\Ga _{12})\cdot x_{r-1}^{\la _3+\Ga_{12}-\Ga _{13}}x_r^{\la _2+a-\Gamma _{12}+n}.
\end{equation}

The proof is a straightforward computation. The strategy is as follows. By \eqref{eq:h_newcoeff_def}, $h^{(\mu, \Ga _0)}$ depends on the residue of $\Ga _{12}$ modulo $n.$ We write 
\begin{equation}\label{eq:k_and_u_def}
 (\Ga _{13}+\la _2-\la _3+1)-a=nk+u,\hskip .5 cm 1\leq u\leq n.
\end{equation}
We define 
\begin{equation}\label{eq:P_k_u_def}
 P_{k,u}(\x)=\frac{x_{r}^n\cdot f_{a,\Ga _0}}{(1-v)\cdot (x_{r-1}x_r)^{\la _3+a-\Ga _{13}}}.
\end{equation}
Since by \eqref{eq:def:typeAaction} $\sigma _{r-1}$ commutes with multiplication by $(x_{r-1}x_r),$ proving \eqref{eq:f_annihilated_goal} is equivalent to showing that $\sigma _{r-1}(P_{k,u}(\x))=P_{k,u}(\x).$ In what follows, we write $\sigma $ for $\sigma _{r-1}$ and $\alpha $ for $\alpha _{r-1}.$ To prove that $P_{k,u}(\x)$ is invariant under $\sigma ,$ we show that
\begin{equation}\label{eq:P_is_nice}
P_{k,u}(\x)=\frac{1-v\cdot \x ^{-n\alpha }}{\x ^{-n\alpha }-1}\cdot x_r^{nk+u+n-1}+\frac{1-v\cdot \x ^{n\alpha }}{\x ^{n\alpha }-1}\cdot \sigma (x_r^{nk+u+n-1}). 
\end{equation}
This is sufficient by Lemma \ref{lemma:h_exchange}.

We are now ready to start the computation. By \eqref{eq:f_timesx_rton}, \eqref{eq:k_and_u_def} and \eqref{eq:P_k_u_def} we have 
\begin{equation}\label{eq:P_rewrite_bydef}
  P_{k,u}(\x) = 
  \frac{1}{1-v}\cdot x_r^{nk+u+n-1}\cdot \sum _{\Ga _{12}=a}^{a+nk+u}\delta _a(\Ga _{12})\cdot h^{(\mu, \Ga _0)}(\Ga _{12})\cdot x_{r-1}^{\Ga_{12}-a}x_r^{-(\Gamma _{12}-a)}
\end{equation}
Recall that since $n|a,$ we have 
$$\delta _a(\Ga _{12})=\left\lbrace\begin{array}{ll}
                                        -v & a=\Ga _{12};\\
                                        1-v & a<\Ga _{12};
                                       \end{array}
\right. $$
$$h^{(\mu, \Ga _0)}(\Ga _{12})=\left\lbrace\begin{array}{ll}
                                                               h^{\flat}(\Ga _{12})=h^{\flat}(\Ga _{12}-a) & \ifi \Ga _{12}<\Ga_{13}+\la_{2}-\la_{3}+1=a+nk+u;\\
                                                               g^{\flat}(\Ga _{12})=g^{\flat}(\Ga _{12}-a) & \ifi \Ga _{12}=\Ga_{13}+\la_{2}-\la_{3}+1=a+nk+u.\\
                                       \end{array}
\right. $$
Furthermore, by Claim \ref{claim:GaussSumRelationship} we have $h^{\flat}(\Ga _{12}-a)=0$ if $n\nmid \Ga _{12}-a.$ Hence the terms of $P_{k,u}(\x)$ where $n\nmid \Ga _{12}-a$ are zero. Furthermore, by Claim \ref{claim:GaussSumRelationship} we also have $h^{\flat}(\Ga _{12}-a)=1-v$ if $n|\Ga _{12}-a$ and $g^{\flat}(\Ga _{12}-a)=v\cdot g_{\Ga _{12}-a}.$ Substituting into \eqref{eq:P_rewrite_bydef}, we get 
\begin{equation}\label{eq:P_rewrite_u_nonzero}
 \begin{split}
  P_{k,u}(\x) = x_r^{nk+u+n-1}\cdot &\frac{1}{1-v}\cdot  \sum _{\Ga _{12}=a}^{a+nk+u}\delta _a(\Ga _{12})\cdot h^{(\mu, \Ga _0)}(\Ga _{12})\cdot \x^{(\Ga_{12}-a)\alpha } \\
=  \frac{1}{1-v}\cdot x_r^{nk+u+n-1} 
 & \cdot \left((-v)\cdot (1-v)+ \sum _{i=1}^k (1-v)^2\cdot \x^{i n\alpha } +(1-v)\cdot v\cdot g_u\cdot \x^{(nk+u)\alpha }\right)\\
=  x_r^{nk+u+n-1} \cdot &\left((-v)+ (1-v)\cdot \frac{\x^{(nk+n)\alpha }-\x^{n\alpha }}{\x^{n\alpha }-1}+v\cdot g_u\cdot \x^{(nk+u)\alpha }\right)
 \end{split}
\end{equation}

To rewrite this in the form of \eqref{eq:P_is_nice}, note that by the definition of the Chinta-Gunnells action in type $A$ \eqref{eq:def:typeAaction}, we have 
\begin{equation}\label{eq:auxiliary_actioncompute}
 \begin{split}
  \frac{1-v\cdot \x ^{n\alpha }}{\x ^{n\alpha }-1}\cdot \sigma (x_r^{nk+u+n-1}) = \\
  = \frac{x_{r-1}^{nk+u+n-1}}{\x ^{n\alpha }-1} \cdot &\left(\x^{-r_{n}(nk+u+n-1)\alpha }\cdot (1-v) -v\cdot g_{nk+u+n}\cdot \x^{(1-n)\alpha }\cdot (1-\x^{n\alpha })\right)\\
=  x_{r}^{nk+u+n-1}\cdot &\x^{(nk+u+n-1)\alpha }\cdot \left(\frac{\x^{(1-u)\alpha }\cdot (1-v)}{\x ^{n\alpha }-1} +v\cdot g_{u}\cdot \x^{(1-n)\alpha }\right)\\
=  x_{r}^{nk+u+n-1}\cdot &\left(\frac{\x^{(nk+n)\alpha }\cdot (1-v)}{\x ^{n\alpha }-1} +v\cdot g_{u}\cdot \x^{(nk+u)\alpha }\right)
 \end{split}
\end{equation}

Comparing \eqref{eq:auxiliary_actioncompute} to \eqref{eq:P_rewrite_u_nonzero} we see that 
%

\begin{equation}\label{eq:P_difference}
 \begin{split}
  P_{k,u}(\x) -\frac{1-v\cdot \x ^{n\alpha }}{\x ^{n\alpha }-1}\cdot \sigma (x_r^{nk+u+n-1}) = & x_r^{nk+u+n-1} \cdot \left((-v)+ (1-v)\cdot \frac{-\x^{n\alpha }}{\x^{n\alpha }-1}\right) \\
= & x_r^{nk+u+n-1} \cdot \frac{1-v\cdot \x^{-n\alpha }}{\x^{-n\alpha }-1}.
 \end{split}
\end{equation}

This completes the proof of \eqref{eq:P_is_nice}, and thus of Lemma \ref{LEMMA:FANNIHILATED}.

%% file: mettok.bbl
\newcommand{\etalchar}[1]{$^{#1}$}
\def\cprime{$'$}
\providecommand{\bysame}{\leavevmode\hbox to3em{\hrulefill}\thinspace}
\providecommand{\MR}{\relax\ifhmode\unskip\space\fi MR }
\providecommand{\MRhref}[2]{%
  \href{http://www.ams.org/mathscinet-getitem?mr=#1}{#2}
}
\providecommand{\href}[2]{#2}
\begin{thebibliography}{BBCG12}

\bibitem[BBCG12]{brubaker2012metaplectic}
B.~Brubaker, D.~Bump, G.~Chinta, and P.~E. Gunnells, \emph{Metaplectic
  {W}hittaker functions and crystals of type {B}}, Multiple Dirichlet Series,
  L-functions and Automorphic Forms, Springer, 2012, pp.~93--118.

\bibitem[BBF06]{bbf-wmd2}
B.~Brubaker, D.~Bump, and S.~Friedberg, \emph{Weyl group multiple {D}irichlet
  series. {II}. {T}he stable case}, Invent. Math. \textbf{165} (2006), no.~2,
  325--355.

\bibitem[BBF11a]{bbf-annals}
\bysame, \emph{Weyl group multiple {D}irichlet series, {E}isenstein series and
  crystal bases}, Ann. of Math. (2) \textbf{173} (2011), no.~2, 1081--1120.

\bibitem[BBF11b]{bbf-wmdbook}
\bysame, \emph{Weyl group multiple {D}irichlet series: type {A} combinatorial
  theory}, Annals of Mathematics Studies, vol. 175, Princeton University Press,
  Princeton, NJ, 2011.

\bibitem[BBF12]{beineke2012crystal}
J.~Beineke, B.~Brubaker, and S.~Frechette, \emph{A crystal definition for
  symplectic multiple {D}irichlet series}, Multiple Dirichlet Series,
  L-functions and Automorphic Forms, Springer, 2012, pp.~37--63.

\bibitem[BBL14]{bbl-demazure}
B.~Brubaker, D.~Bump, and A.~Licata, \emph{Whittaker functions and {D}emazure
  operators}, Journal of Number Theory (2014).

\bibitem[BD10]{bucur2010moments}
A.~Bucur and A.~Diaconu, \emph{Moments of quadratic {D}irichlet {L}-functions
  over rational function fields}, Mosc. Math. J \textbf{10} (2010), no.~3,
  485--517.

\bibitem[BZ93]{berenstein1993string}
A.~Berenstein and A.~Zelevinsky, \emph{String bases for quantum groups of type
  ${A}_r$}, Adv. Soviet Math \textbf{16} (1993), no.~1, 51--89.

\bibitem[BZ{\etalchar{+}}96]{berenstein1996canonical}
A.~Berenstein, A.~Zelevinsky, et~al., \emph{Canonical bases for the quantum
  group of type ${A}_r$ and piecewise-linear combinatorics}, Duke Mathematical
  Journal \textbf{82} (1996), no.~3, 473--502.

\bibitem[CG10]{cg-jams}
G.~Chinta and P.~E. Gunnells, \emph{Constructing {W}eyl group multiple
  {D}irichlet series}, J. Amer. Math. Soc. \textbf{23} (2010), no.~1, 189--215.

\bibitem[CG12]{chinta2012littelmann}
\bysame, \emph{Littelmann patterns and {W}eyl group multiple {D}irichlet series
  of type {D}}, Multiple Dirichlet Series, L-functions and Automorphic Forms,
  Springer, 2012, pp.~119--130.

\bibitem[CGP14]{cgp}
G.~Chinta, P.~E. Gunnells, and A.~Pusk\'{a}s, \emph{Metaplectic {D}emazure
  operators and {W}hittaker functions}, arXiv preprint arXiv:1408.5394 (2014),
  to appear in Indiana Univ. Math. Journal.

\bibitem[CO13]{co-cs}
G.~Chinta and O.~Offen, \emph{A metaplectic {C}asselman-{S}halika formula for
  ${GL}_r$}, American Journal of Mathematics \textbf{135} (2013), no.~2,
  403--441.

\bibitem[FZ15]{friedberg2015eisenstein}
S.~Friedberg and L.~Zhang, \emph{Eisenstein series on covers of odd orthogonal
  groups}, American Journal of Mathematics \textbf{137} (2015), no.~4,
  953--1011.

\bibitem[HK02]{hong2002introduction}
J.~Hong and S.-J. Kang, \emph{Introduction to quantum groups and crystal
  bases}, vol.~42, American Mathematical Soc., 2002.

\bibitem[Hum78]{humph}
J.~E. Humphreys, \emph{Introduction to {L}ie algebras and representation
  theory}, Graduate Texts in Mathematics, vol.~9, Springer-Verlag, New York,
  1978, Second printing, revised.

\bibitem[K{\etalchar{+}}92]{kashiwara1992crystal}
M.~Kashiwara et~al., \emph{Crystal base and {L}ittelmann's refined {D}emazure
  character formula}, Kyoto University, Research Institute for Mathematical
  Sciences, 1992.

\bibitem[Kas95]{kashiwara1995crystal}
M.~Kashiwara, \emph{On crystal bases}, Representations of groups (Banff, AB,
  1994) \textbf{16} (1995), 155--197.

\bibitem[KB96]{kirillov1996groups}
A.~N. Kirillov and A.~D. Berenstein, \emph{Groups generated by involutions,
  {G}elfand-{T}setlin patterns and combinatorics of {Y}oung tableaux}, titi
  \textbf{1} (1996), t1.

\bibitem[KP84]{MR743816}
D.~A. Kazhdan and S.~J. Patterson, \emph{Metaplectic forms}, Inst. Hautes
  \'Etudes Sci. Publ. Math. (1984), no.~59, 35--142.

\bibitem[Kub71]{kubota1971some}
T.~Kubota, \emph{Some results concerning reciprocity law and real analytic
  automorphic functions}, 1969 Number Theory Institute (Proc. Sympos. Pure
  Math., Vol. XX, State Univ. New York, Stony Brook, NY, 1969), 1971,
  pp.~382--395.

\bibitem[Lit98]{littelmann1998cones}
P.~Littelmann, \emph{Cones, crystals, and patterns}, Transformation groups
  \textbf{3} (1998), no.~2, 145--179.

\bibitem[LLL16]{lee2016whittaker}
K.-H. Lee, C.~Lenart, and D.~Liu, \emph{Whittaker functions and {D}emazure
  characters}, arXiv preprint arXiv:1602.06451 (2016).

\bibitem[Lus90]{lusztig1990canonical}
G.~Lusztig, \emph{Canonical bases arising from quantized enveloping algebras},
  Journal of the American Mathematical Society \textbf{3} (1990), no.~2,
  447--498.

\bibitem[LZ12]{lee2012weyl}
K.-H. Lee and Y.~Zhang, \emph{Weyl {G}roup {M}ultiple {D}irichlet {S}eries for
  {S}ymmetrizable {K}ac-{M}oody {R}oot {S}ystems}, arXiv preprint
  arXiv:1210.3310 (2012).

\bibitem[Mat69]{matsumoto1969sous}
H.~Matsumoto, \emph{Sur les sous-groupes arithm{\'e}tiques des groupes
  semi-simples d{\'e}ploy{\'e}s}, Annales Scientifiques de l'Ecole Normale
  Superieure, vol.~2, Soci{\'e}t{\'e} math{\'e}matique de France, 1969,
  pp.~1--62.

\bibitem[McN11]{mcnamara}
P.~J. McNamara, \emph{Metaplectic {W}hittaker functions and crystal bases},
  Duke Math. J. \textbf{156} (2011), no.~1, 1--31.

\bibitem[McN16]{mcnamara2}
P.~McNamara, \emph{The metaplectic {C}asselman-{S}halika formula}, Transactions
  of the American Mathematical Society \textbf{368} (2016), no.~4, 2913--2937.

\bibitem[Pat14]{patnaik2014unramified}
M.~M. Patnaik, \emph{Unramified {W}hittaker functions on p-adic loop groups},
  arXiv preprint arXiv:1407.8072 (2014), to appear in American Journal of
  Mathematics.

\bibitem[PP15]{patnaik2015casselman}
M.~M. Patnaik and A.~Pusk{\'a}s, \emph{On the {C}asselman-{S}halika formula for
  metaplectic groups}, arXiv preprint arXiv:1509.01594 (2015).

\bibitem[Tok88]{tokuyama-generating}
T.~Tokuyama, \emph{A generating function of strict {G}el\cprime fand patterns
  and some formulas on characters of general linear groups}, J. Math. Soc.
  Japan \textbf{40} (1988), no.~4, 671--685.

\bibitem[Whi14]{whitehead2014affine}
I.~Whitehead, \emph{Affine {W}eyl group multiple {D}irichlet series: Type
  $\tilde{A}$}, arXiv preprint arXiv:1408.0937 (2014).

\end{thebibliography}
